\documentclass[11pt]{amsart}

\usepackage{url}
\usepackage{psfrag}
\usepackage{graphicx}
\usepackage{comment}
\usepackage{array}
\usepackage{amsmath,amsfonts,amssymb,amsxtra,amsthm}
\usepackage{mathrsfs}
\usepackage{lineno}
\usepackage{verbatim}
\usepackage{xy}
\usepackage{times}
\xyoption{all}
\usepackage{stmaryrd}
\usepackage{appendix}
\usepackage{xr}

\DeclareMathOperator{\betti}{b_1}

\DeclareMathOperator{\comp}{Comp}

\DeclareMathOperator{\depth}{depth}

\DeclareMathOperator{\dirlim}{\underrightarrow{\lim}}

\DeclareMathOperator{\Env}{Env}

\DeclareMathOperator{\Hom}{Hom}

\DeclareMathOperator{\img}{Im}
\DeclareMathOperator{\inn}{Inn}

\DeclareMathOperator{\Ker}{Ker} 

\DeclareMathOperator{\Mod}{Mod}

\DeclareMathOperator{\ninj}{NInj}

\DeclareMathOperator{\phia}{\Phi_a}
\DeclareMathOperator{\phias}{\Phi_{as}}
\DeclareMathOperator{\phir}{\Phi_r}

\DeclareMathOperator{\rk}{rk}

\DeclareMathOperator{\scott}{sc}
\DeclareMathOperator{\seq}{Seq}
\DeclareMathOperator{\stabker}{\underrightarrow{\Ker}}
\DeclareMathOperator{\st}{St}

\def\jsj{\textsc{JSJ}}

\def\qcjsj{\textsc{QCJSJ}}

\def\qh{\textsc{QH}}

\def\hnn{\textsc{HNN}}

\def\gad{\textsc{GAD}}

\newcommand{\E}{\mathcal{E}}
\newcommand{\G}{\mathcal{G}}
\newcommand{\R}{\mathcal{R}}

\renewcommand{\H}{\mathcal{H}} 

\newcommand{\LL}{\mathcal{L}}

\def\wt#1{\widetilde{#1}}

\def\sl2c{\ensuremath{{SL}(2,\mathbb{C})}}
\def\t1sl2c{{\mathfrak sl}_2\mathbb{C}}
\def\free{\ensuremath{\mathbb{F}}}
\def\zee{\mathbb{Z}}

\def\define{\raisebox{0.3pt}{\ensuremath{:}}\negthinspace\negthinspace=}

\newcommand{\intoonto}{\DOTSB\lhook\joinrel\twoheadrightarrow}

\def\onto{\twoheadrightarrow}
\def\into{\hookrightarrow}

\def\ncl#1{\mathord{\langle}\mskip -4mu plus 0mu minus 0mu \mathord{\langle}#1\mathord{\rangle}\mskip -4mu plus 0mu minus 0mu \mathord{\rangle}}

\newcommand{\grushko}[3]{\ensuremath{#1_1*\cdots*#1_{#2}*\free_{#3}}}

\def\cent{Z}

\makeatletter
\newcommand{\Doubletwo}[3]{ \left\{ #1 \mid #3\right\} }
\newcommand{\Doubleone}[1]{ \left\{ #1 \right\} }
\newcommand{\Doublethr}[3]{ \left\{ #1 \right\}_{#3} }
\newcommand{\set}[1]{%
\@ifnextchar:{\Doubletwo{#1}}{\@ifnextchar_{\Doublethr{#1}}{\Doubleone{#1}}}%
}
\makeatother

\makeatletter
\newcommand{\grouptwo}[3]{ \langle #1 \mid #3\rangle }
\newcommand{\groupone}[1]{ \langle #1 \rangle }
\newcommand{\group}[1]{%
\@ifnextchar:{\grouptwo{#1}}{\groupone{#1}}%
}
\makeatother

\def\rto{\boxslash}
\def\strict{\ensuremath{\boldsymbol{S}}}
\def\strictr{\ensuremath{\boldsymbol{AS}}}

\newcommand{\term}[1]{{\emph{#1}}}

\newcommand{\latin}[1]{{\it #1}} 
\newcommand{\french}[1]{{\it #1}}

\newcommand{\mobius}{M\"obius}

\newtheorem{theorem}{Theorem}[section]
\newtheorem{lemma}[theorem]{Lemma}
\newtheorem{example}[theorem]{Example}
\newtheorem{corollary}[theorem]{Corollary}

\theoremstyle{definition}
\newtheorem{definition}[theorem]{Definition}

\theoremstyle{remark}
\newtheorem{remark}[theorem]{Remark}

\newcommand{\mnote}[1]{}

\title[Aligning \jsj\ decompositions]{Krull dimension for limit groups {II}:\\
  aligning \jsj\ decompositions}

\author[Lars Louder]{Larsen Louder}
\address{Department of Mathematics\\
University of Michigan \\
Ann Arbor, MI 48109-1043\\
USA}
\email[Larsen Louder]{llouder@umich.edu, lars@d503.net}

\keywords{limit group, krull dimension, JSJ, fully residually free}

\subjclass[2000]{Primary: 20F65; Secondary: 20E05, 20E06}

\thanks{Most of this research was done while at the University of
  Utah. The author also gratefully acknowledges support from the
  National Science Foundation.}

\begin{document}

\begin{abstract}
  This is the second paper in a sequence on Krull dimension for limit
  groups, answering a question of Z.~Sela. In it we develop the notion
  of a resolution of a sequence of limit groups and show how to derive
  resolutions of low complexity from resolutions of high complexity.
\end{abstract}

\maketitle


\section{Introduction}

\label{sec:introduction}

\thispagestyle{empty}

\par A basic and important fact in algebraic geometry is that
varieties have finite Krull dimension, that is, given a variety $V,$
chains of proper inclusions of irreducible subvarieties have length
bounded above by the dimension of $V.$ Remarkably, solutions to
systems of equations defined over a free group $\free$ can also be
decomposed into irreducible
components~\cite{sela::dgog1,baumslag::ag,km}. Associated to
each irreducible component is a \term{limit group} which plays the
role of the coordinate ring of a variety: the points of the component
are tautologically identified with homomorphisms from the limit group
to the free group. A chain of irreducible subvarieties corresponds to
a sequence of epimorphisms of limit groups and finiteness of Krull
dimension comes out of this analysis as the supremum of lengths of
chains of epimorphisms of limit groups, beginning with a fixed limit
group. Just as (affine) varieties are subsets of the affine space,
varieties over $\free$ are contained in $\free^n$ for some $n.$ Finite
dimension reduces to bounding lengths of chains of epimorphisms of
limit groups requiring the same number of generators.

As in any inductive proof, a complexity is assigned to sequences of
epimorphisms of limit groups. This paper gives a means for deriving,
given a sequence of some complexity, uniformly many sequences of
strictly lower complexity. Moreover, the derived sequences are bound
together in a graph of sequences of groups, a resolution, of the
original sequence, analogous to a graphs of groups decomposition,
modeled on a common \jsj\ decomposition induced by the original
sequence. Using the main results from~\cite{louder::scott}
and~\cite{louder::roots} we lift a dimension bound for the derived
sequences to a bound on the length of the resolution
(Theorem~\ref{thr:stratified}), and from there to a bound on the
length of original sequence, depending only on the complexity of the
original sequence. An important feature is that the derived sequences
are obtained through a special process of iteratively adjoining roots,
passing to limit group quotients, and then passing to further limit
group quotients. These sequences are the motivation for the sequences
of adjunctions of roots discussed in~\cite{louder::roots}

\par Solution sets of systems of equations defined over the free group
have received considerable attention over the last decade,
particularly in the positive resolution of Tarski's question about
elementary equivalence of nonabelian free groups by Sela
\cite{sela::dgog1,sela::dgog2,sela::dgog3,sela::dgog4,sela::dgog5,sela::dgog6}
See also~\cite{km} for an alternative approach by Kharlampovich and
Myasnikov.

The object of this sequence (the present paper
and~\cite{louder::roots,louder::scott,louder::strict}) is to prove
that limit groups have finite \term{Krull dimension}
(Theorem~\ref{thr:krulldimension}). In other words, if
\[
\free_n\onto L_1\onto\dotsb\onto L_k
\] 
is a sequence of proper epimorphisms of limit groups, then $k$ is
bounded by a function of $n.$ The question comes from algebraic
geometry and logic. A system of equations over the free group is a
collection of words $\Sigma=\set{\omega_i}$ in finitely many variables
$\set{x_i}_{i=1..n}.$ The set of solutions to $\Sigma$ in $\free$ is
  identified with a subset of $\free^n$:
\[
V_{\Sigma}=\set{(a_1,\dotsc,a_n)\in\free^n}:{\omega_i(\overline{a})=1\mbox{ for all $\omega_i\in\Sigma$}}
\] 

Associated to any such system of equations ${\Sigma}$ is a finitely
generated group $G_{\Sigma}\define\group{x_i}/\ncl{\Sigma},$ and there
is a tautological one-to-one correspondence between the sets
$\Hom(G_{\Sigma},\free)$ and $V_{\Sigma}.$ Let $g\in\group{x_i}.$ The
set $V_{\set{g}}\cap V_{\Sigma}$ is a Zariski closed subset of
$V_{\Sigma}.$ If $G_{\Sigma}$ is residually free, and if $V_{\Sigma}$
is irreducible, that is, it isn't contained in the union of finitely
many proper closed subsets, then for all finite collections of
elements $\set{g_is}\subset G_{\Sigma}\setminus\set{1},$ $V_{\Sigma}$
properly contains the union
\[
  \bigcup_i\left(V_{\set{g_i}}\cap V_{\Sigma}\right)
\] 
A point in the complement is a homomorphism $G_{\Sigma}\to\free$ which
doesn't kill any element of the finite set, i.e., $G_{\Sigma}$ is
\term{$\omega$--residually free.} If $G_{\Sigma}$ isn't residually
free, we may always pass to the universal residually free quotient
$RF(G_{\Sigma})$ by killing all elements of $G_{\Sigma}$ which die
under every homomorphism to the free group. The quotient map
$G_{\Sigma}\onto RF(G_{\Sigma})$ induces a bijection
$\Hom(RF(G_{\Sigma}),\free)\to\Hom(G_{\Sigma},\free).$

\begin{definition}
  \label{def:limitgroup}
  Let $G$ be a finitely generated group. A sequence
  $f_n\in\Hom(G,\free)$ such that for all $g\in G,$ either $f_n(g)=1$
  for sufficiently large $n$ or $f_n(g)\neq 1$ for sufficiently large
  $n,$ is \term{stable}. The \term{stable kernel} of a stable sequence
  $f_n$ is the normal subgroup of $G$ generated by all elements which
  have trivial image for sufficiently large $n,$ and is denoted by
  $\stabker(f_n).$ A quotient of a finitely generated group by the
  stable kernel of a stable sequence is a \term{limit group}.
\end{definition}

\par Sometimes we say that the sequence $f_n$ \term{converges to
  $G/\stabker(f_n)$}. To justify this terminology, consider that a
homomorphism $f\colon G\to\free$ can be thought of as a point in the
space of marked free groups, and that a convergent sequence of
homomorphisms can be thought of as a sequence of points converging to
a point in the completion of the space of marked free groups
corresponding to the limit group associated to the
sequence. See~\cite{guirardel-champetier} for a discussion of limit
groups from this point of view.

\par Although $\omega$--residually free groups are limit groups, the
converse follows from their finite presentability and is more
difficult to prove.  See~\cite[Theorem~4.6]{sela::dgog1}
or~\cite{guirardel::fp} for proofs of finite presentability, and for a
proof which doesn't use finite presentability,
see~\cite[Lemma~1.10]{bf::lg}. Since a limit group $L$ is
$\omega$--residually free there is always a sequence $f_n\colon
L\to\free$ converging to $L.$

\par A \term{sequence of limit groups} is a finite sequence
\[\LL=(\dotsc,\LL(i_j),\LL(i_{j+1}),\dotsc)\]
of limit groups, denoted by a calligraphic letter, for example
$\LL$ or $\R,$ indexed by a strictly increasing sequence of
natural numbers $(i_j),$ equipped with
homomorphisms \[\varphi_{m,n}\colon\LL(m)\to\LL(n)\]
such that
\[\varphi_{n,l}\circ\varphi_{m,n}=\varphi_{m,l}\] 
for all $m<n<l.$ Sometimes we ignore the fact that a sequence isn't
necessarily indexed by adjacent natural numbers, and simply refer to
the $i$--th element in a sequence $\LL$ as $\LL(i).$ A
sequence is a \term{chain} if all the maps are epimorphisms. We use
this terminology because the dual sequence of
$\Hom$--sets \[\dotsb\supset\Hom(\LL(i),\free)\supset\Hom(\LL(i+1),\free)\supset\dotsb\]
is a chain of varieties.

This explains the terminology ``Krull dimension.'' A limit group $L$
determines a variety $V_L=\Hom(L,\free),$ and a chains of epimorphisms
of limit groups originating from $L$ correspond to chains of
irreducible subsets of $V_L.$


The length of a sequence is denoted by $\Vert\LL\Vert.$ The
\term{proper length} of a chain $\LL,$ denoted by
$\Vert\LL\Vert_{pl},$ is the number of indices $n$ such that
$\LL(n-1)\onto\LL(n)$ is not an isomorphism. The
\term{rank} of a limit group $L$ is the minimal number of elements
needed to generate $L,$ and is denoted by $\rk(L).$ The \term{rank} of
a chain of limit groups is the rank of the first group in the chain:
$\rk(\LL)\define\rk(\LL(1)).$ The first betti number
of a chain is defined in the same manner and is denoted by
$\betti(\LL).$

\begin{theorem}[Krull dimension for limit groups]
  \label{thr:krulldimension}
  For all $N$ there is a constant $D=D(N)$ such that if
  $\LL$ is a chain of rank at most $N$ then
  $\Vert\LL\Vert_{pl}\leq D.$
\end{theorem}

This paper and its sibling~\cite{louder::roots} contain a complete
proof of Theorem~\ref{thr:krulldimension}. The sequel contains a proof
of Theorem~\ref{thr:addrootstolimitgroups}, which is used in
Section~\ref{sec:liftdimensionbound} to lift a dimension bound for
sequences of lower complexity to sequences of 1
complexity. \cite{louder::roots} contains the portion of the proof
specific to limit groups, and~\cite{louder::scott} contains the
remainder of the proof. Theorem~\ref{thr:addrootstolimitgroups} is a
``Krull-like'' statement about certain sequences of limit groups, each
obtained from the last by adjoining roots, passing to a (certain)
limit group quotient, and then perhaps passing to a further quotient
limit group.  First, some definitions and notation.

The centralizer of $E<G$ is denoted by $\cent_G(E).$ Let $\Delta$ be a
graph of groups decomposition of $G.$ If $V$ is a vertex group of
$\Delta,$ denote the set of images edge groups incident to $V$ by
$\E(V).$

\begin{definition}[Adjoining roots]
  \label{def:adjunctionofroots}
  \label{def:adjoinroots-alternate}
  Let $G$ be a limit group, $\E$ a collection of abelian
  subgroups of $G.$ Suppose that for each element $E\in\E,$
  we are given a finite index free abelian supergroup $F(E).$ A limit
  group quotient $G'$ of
  \[G*_{E\in\E}F(E)\] such that the restriction of the
  quotient map to $G$ is injective is said to be obtained from $G$ by
  adjoining roots $F(\E)$ to $\E.$ An adjunction
  triple is a tuple $(G,H,G')$ such that $G\into H\onto G',$ $H$
  obtained from $G$ by adjoining roots, such that $G$ embeds in $G'$
  under the quotient map.


  A \term{sequence of adjunctions of roots} is a pair of sequences of
  limit groups and a family $\E$ of collections of subgroups
  $\E_i$ of $\G(i),$ $(\G,\H,\E),$ with \term{base
    sequence} $\G,$ such that
  \begin{itemize}
  \item $(\G(i),\H(i+1),\G(i+1))$ is an adjunction triple; $\H(i+1)$
    is obtained from $\G(i)$ by adjoining roots to $\E_i$
  \item Each $E'\in\E_{i+1}$ in $\G(i+1)$ centralizes, up to
    conjugacy, the image of an element $E$ of $\E_i.$ If
    $E\in\E_i$ is mapped to $E'\in\E_{i+1}$ then the
    image of $\cent_G(E)$ in $\cent_{G'}(E')$ must be finite index.
  \end{itemize}
\end{definition}

The \term{complexity} of $(\G,\H,\E)$ is the triple
\[\comp((\G,\H,\E))\define(\betti(\G),\depth_{\text{pc}}(\H),\Vert\E\Vert).
\]
See~\cite[Definition~\ref{ROOTS-def:depth}]{louder::roots} for the
definition of $\depth_{\text{pc}},$ the depth of the principle cyclic
analysis lattice of a limit group. Complexities are not compared
lexicographically: $(b',d',e')\leq (b,d,e)$ if $b'\leq b,$ $d'\leq d,$
and $e'\leq e+2(d-d')b.$

Let $(\G,\H,\E)$ be a sequence of adjunctions of roots. The
quantity $\ninj((\G,\H,\E))$ is the number of indices $i$
such that $\H(i)\onto\G(i)$ is \emph{not} an isomorphism.

\begin{theorem}[Adjoining roots (\cite{louder::roots})]
  \label{thr:addrootstolimitgroups}
  Let $(\G,\H,\E)$ be a sequence of adjunctions of
  roots. There is a function $\ninj(\comp((\G,\H,\E)))$ such
  that \[\ninj((\G,\H,\E))\leq\ninj(\comp((\G,\H,\E)))\]
\end{theorem}

This is~\cite[Theorem~\ref{ROOTS-maintheorem}]{louder::roots}.

\subsection*{Acknowledgments}

\par I am deeply indebted to my advisor Mladen Bestvina, and extend
many thanks to Matt Clay, Mark Feighn, Chlo\'e Perin, and Zlil Sela
for listening carefully and critically to my musings on Krull.

\section{Preliminaries}

In this section we set up some notation, review theorems from the
theory of limit groups, and do some basic setup for later sections. We
start by giving some basic properties of limit groups and defining
generalized abelian decompositions, or \gad s.

\begin{theorem}[Basic properties of limit
  groups~\cite{bf::lg},~\cite{sela::dgog1}] The following properties
  are shared by all limit groups.
  \label{thr:basic-prop}
  \begin{itemize}
  \item Commutative transitivity; maximal abelian subgroups are
    malnormal; every abelian subgroup is contained in a unique maximal
    abelian subgroup.
  \item Abelian subgroups are finitely generated and free.
  \item Finite presentability and coherence.
  \end{itemize}
\end{theorem}

\begin{definition}[Generalized abelian decomposition (\cite{bf::lg})]
  \label{def:gad}
  A \term{generalized} \term{abelian} \term{decomposition}, shortened
  to \gad, or \term{abelian decomposition}, of a freely indecomposable
  finitely generated group $G,$ is a graph of groups
  $\Delta(R_i,A_j,Q_k,E_l)$ of $G$ over vertex groups $R_i,$ $A_j,$
  $Q_k,$ and abelian edge groups $E_l$ such that:
  \begin{itemize}
    \item $A_j$ is abelian for all $j.$
    \item $Q_k$ is the fundamental group of a surface $\Sigma_k$ with
      boundary and $\chi(\Sigma_k)\leq-2$ or $\Sigma_k$ is a torus
      with one boundary component. The $Q_k$ are ``quadratically
      hanging.''
    \item Any edge group incident to a quadratically hanging vertex
    group is conjugate into a boundary component.
  \end{itemize}
  The subgroups $R_i$ are the \term{rigid} vertex groups of $\Delta.$
\end{definition}

\par The \term{peripheral subgroup} of an abelian vertex group $A$ of
a \gad\ is the direct summand $P(A)$ of $A$ which is the intersection
of all kernels of maps $A\to\zee$ which kill all images of incident
edge groups. The peripheral subgroup is primitive in $A$: if
$\alpha\in A\setminus P(A)$ then no power of $\alpha$ is in
$P(A).$ This differs slightly from \cite{bf::lg} in that they use this
term to refer to the subgroup generated by incident edge groups.

\par If $\Delta$ is a \gad\ of a finitely generated group $G,$ we say
that a splitting $G=G_1*_EG_2$ or $G=G_1*_E$ is \term{visible} in
$\Delta$ if it is a one edged splitting from an edge of $\Delta,$ a
one edged splitting obtained by cutting a \qh\ vertex group along
an essential simple closed curve, or a one edged splitting inherited
from a one edged splitting of an abelian vertex group in which the
peripheral subgroup is elliptic.

\begin{definition}[The modular group $\Mod$]
  \label{def:modulargroup}
  To a \gad\ $\Delta$ of $G$ we associate the \term{restricted modular
    group} $\Mod(G,\Delta),$ which is the subgroup of the automorphism
  group of $G$ generated by inner automorphisms and all Dehn twists in
  one edged splittings visible in $\Delta.$

  \par Let $G$ be a finitely generated group with a Grushko
  decomposition
  \[G=G_1*\dotsb*G_p*\free_q\] and for each $i$ a \gad\ $\Delta_i$ of
  $G_i.$ We say that an automorphism $\varphi$ of $G$ is
    \term{$\Delta_i$--modular on $G_i$}, or simply \term{modular}, if
    there are automorphisms $i_g\in\inn(G)$ and
    $\varphi'\in\Mod(G_i)<\Mod(G)$ such that the restriction of
    $\varphi$ to $G_i$ agrees with the restriction of
    $i_g\circ\varphi'$ to $G_i.$  The set of automorphisms which are
    $\Delta_i$--modular on $G_i$ for all $i$ forms the \term{modular
      group of $G$ with respect to $\set{\Delta_i}$}, and is denoted
    by $\Mod(G,\set{\Delta_i}).$ The group of automorphisms of $G$
    generated by all such groups, as $\set{\Delta_i}$ varies over all
    collections of \gad's of freely indecomposable free factors of
    $G,$ is denoted by $\Mod(G).$
\end{definition}

\par To a finitely presented freely indecomposable limit group one can
associate the abelian \jsj\ decomposition, a canonical \gad\ such that
every other \gad\ can be obtained from it by folding, sliding,
conjugation, cutting \qh\ vertex groups along simple closed curves,
and collapsing subgraphs. The \jsj\ theory is well developed and we
assume some familiarity with it. See~\cite{dunwoodyjsj}
or~\cite{fuji::jsj} for treatments not specific to limit groups,
or~\cite{sela::jsj} and~\cite{bf::lg} for treatments more specific to
limit groups.  We assume the following normalizations on the abelian
\jsj\ of a limit group:

\begin{itemize}
\item The \jsj\ decomposition is 2-acylindrical: if $T$ is the
  associated Bass-Serre tree, then stabilizers of segments of length
  at least three are trivial. Through folding, sliding, and
  conjugation, ensuring that edges incident to rigid vertex groups
  have nonconjugate centralizers (in the vertex group in question),
  and collapsing, to ensure that no two abelian vertex groups are
  adjacent, we can arrange that if $g$ fixes a segment of length two
  with central vertex $v,$ then the stabilizer of $v$ is abelian.
\item Let $E$ be an edge group of a limit group $L.$ Then there is an
  unique abelian vertex group $A$ such that $A=\cent_L(E).$ This is
  somewhat awkward as the normal impulse upon encountering a
  surjective map \[(\mbox{edge group})\to(\mbox{vertex group})\] is to
  collapse the edge. This hypothesis does, however, make the
  construction of strict homomorphisms in
  subsection~\ref{subsection::freelyindecomposable} formally somewhat
  easier.
\item Edge groups incident to rigid vertex groups are closed under
  taking roots in the ambient group. Edge groups incident to a rigid
  vertex group $R$ are nonconjugate in $R.$ An edge group is closed
  under taking roots in all adjacent non-\qh\ vertex groups.
\item If $E$ is incident to a \qh\ vertex group it is conjugate to a
  boundary component.
\item If $E$ is an edge group and $A$ is a valence one abelian vertex
  group such that $E$ is finite index in $A,$ then $E$ is attached to
  a boundary component $b$ of a \qh\ vertex group and there are no
  other incident edge groups with image conjugate to $b.$ The index of
  $E$ in $A$ must be at least three. (If it is one, then the splitting
  is trivial, and if it is two, then the \qh\ vertex group can be
  enlarged by gluing a \mobius\ band to the boundary component
  representing $E.$)
\end{itemize}

\par The \term{modular group} of a freely indecomposable limit group
$G$ is the group $\Mod(L,\jsj(L)),$ or simply $\Mod(L)$ for
convenience. This definition agrees with the previous definition of
$\Mod(L)$ since $(1)$ $\Mod(L,\Delta)$ is a subgroup of
$\Mod(L,\jsj(L))$ and $(2)$ \jsj\ decompositions exist.

\par Let $\Delta$ be a \gad\ of a limit group $L,$ and let $L_P$ be
the subgroup of $L$ constructed by replacing $A_j$ by $P(A_j)$ for all
$j.$ Let $R$ be a rigid vertex group of $\Delta.$ The \term{envelope}
of $R,$ $\Env(R,\Delta),$ is the subgroup of $L_P$ generated by $R$
and the centralizers of edge groups incident to $R$ in $L_P.$ Let
$H<G$ be a pair of groups. An automorphism of $G$ is \term{internal on
  $H$} if it agrees with the restriction of an inner automorphism. The
envelope of a rigid vertex group of $L$ can also be thought of as the
largest subgroup of $L$ containing $R$ for which the restriction of
every element of $\Mod(L,\Delta)$ is internal.

\par A limit group is \term{elementary} if it is free, free abelian,
or the fundamental group of a closed surface. 


Let $L$ be a freely indecomposable nonelementary limit group. The
cyclic \jsj\ of $L$ is the \jsj\ decomposition associated to the
family of all one-edged cyclic splittings in which all noncyclic
abelian subgroups are elliptic. Since limit groups have cyclic
splittings, the cyclic \jsj\ of a freely indecomposable nonelementary
limit groups is nontrivial.

\begin{definition}[Cyclic analysis lattice; depth]
  \label{def:alatticeheightbound}
  \label{def:depth}
  The cyclic analysis lattice of a limit group $L$ is a tree
  of groups constructed as follows:
  \begin{enumerate}
  \item Level $0$ of the analysis lattice is $L.$
  \item Let $L_1*\dotsb*L_p*\free_q$ be a Grushko factorization of
    $L.$ Level $1$ of the analysis lattice consists of the groups
    $L_i$ and $\free_q.$ There is an edge connecting $L$ to each group
    in level $1.$ The free factor $\free_q$ and any surface or free
    abelian freely indecomposable free factors of $L$ are elementary,
    and we take them to be terminal leaves of the tree.
  \item Let $L_{i,1},\dotsb,L_{i,j_i}$ vertex groups of the cyclic
    \jsj\ of $L.$  Level $3$ of the lattice consists of the groups
    $\set{L_{i,k}}.$  There is an edge connecting $L_{i}$ to $L_{i,k}$
    for all $i$ and $k.$
  \item Inductively, construct the analysis lattice for each group
    $L_{i,j}$ and graft the root of the tree to the vertex labeled
    $L_{i,j}.$
  \item Terminate only when all resulting leaves are terminal.
  \end{enumerate}
  The \term{depth} of a limit group $L$ is the number of levels in its
  cyclic analysis lattice, and is denoted by $\depth(L).$  The depth
  of a sequence of limit groups $\LL$ is the greatest depth of
  a group appearing in
  $\LL$: \[\depth(\LL)\define\max_i\set{\depth(\LL(i))}\]
\end{definition}

Lemma~\ref{lem:depthbound} below follows from the following theorem of
Sela's, but for completeness we provide an alternative proof.

\begin{theorem}[\latin{cf.} \cite{sela::dgog1}, Proposition~4.3]
  \label{thr:rankboundedbybetti}
  There is a function $r$ such that if $L$ is a limit group then
  $\depth(L)\leq r(\betti(L)).$
\end{theorem}

Since the depth of $L$ is bounded by $r,$ and the width of the
analysis lattice is bounded by homological considerations, this
proposition gives a bound on the rank of a limit group in terms of its
first betti number.

\begin{remark}
  The cyclic analysis lattice of a limit group is defined
  in~\cite{sela::dgog1}, where he shows that it is finite, and in fact
  has depth bounded quadratically in the first betti number of the
  given group.

  We don't need the full strength of
  Theorem~\ref{thr:rankboundedbybetti} in this paper, although it
  would make some steps slightly easier.  For this reason we use
  slightly more roundabout logic for our proof of
  Theorem~\ref{thr:krulldimension} than is strictly
  necessary. Specifically, when something can be controlled by
  $\betti$ alone, we may need to use $\betti$ and depth.

  Since all the facts we use about limit groups can be proven
  independently of Sela's Theorem~\ref{thr:rankboundedbybetti}, see
  the treatment in~\cite{bf::lg}, our proof doesn't secretly rely on
  Sela's proposition.
\end{remark}

In order be able to apply the bound on the length of a strict
resolution provided by Theorem~\ref{thr:strictlengthbound} to
sequences constructed throughout the course of the paper, especially
in section~\ref{section::ai}, we need to use the next Lemma to gain
control over ranks.

\begin{lemma}
  \label{lem:rankheightbound}
  There is a function $r(b,d)$ such that if $L$ is a limit group then
  \[\rk(L)\leq r(\betti(L),\depth(L))\]
\end{lemma}

\begin{proof}
  The proof is by induction on the pair $(b,d),$ ordered by comparing
  both coordinates. Suppose the theorem holds for groups of complexity
  less than $(b_0,d_0).$ For $(b,d)<(b_0,d_0)$ let $\rk(b,d)$ be a
  function which bounds the rank. Since we assume a bound on the depth
  of $L,$ we only need to show that the number of vertex groups of the
  cyclic \jsj\ of $L$ is controlled by $\betti(L).$ Suppose $B(b)$ is
  such an upper bound. Since the vertex groups have lower depth, their
  ranks are bounded by
  \[R=\max\set{\rk(b,d)}:{b\leq b_0,d<d_0)}\]
  Then 
  \[\rk(b,d)\leq B(b_0)R+b_0\]
  The term $b_0$ is the largest possible contribution to $\betti(L)$
  made by stable letters from the Bass-Serre presentation of $L$ in
  terms of its cyclic \jsj.

  To find $B,$ let
  $\Delta(\mathcal{R},\mathcal{Q},\mathcal{A},\E)$ be a
  cyclic \jsj\ decomposition of $L.$ Modify $\Delta$ by choosing, for
  each $\qh$ vertex group $Q\in\mathcal{Q},$ a pair of pants
  decomposition $P_Q,$ and cut the \qh\ vertex groups along the simple
  closed curves from $P_Q.$ The underlying graph of $\Delta$ has first
  betti number bounded by $b_0.$ Since the first betti number of a
  nonabelian limit group relative to an abelian subgroup is at least
  one, $\Delta$ has at most $b_0$ valence one vertex groups. The
  number of abelian vertex groups $A$ such that $A\neq P(A)$ is
  bounded above by $b_0$ as well. Thus, to bound the complexity of
  $\Delta,$ we only need to bound the number of valence two vertex
  groups. Since the number of abelian vertex groups which aren't equal
  to their peripheral subgroups is bounded by $b_0,$ we only need to
  bound the number and size of sub-graphs of groups of $\Delta$ with
  the following form:
  \[\dotsb*_{\zee}R_1*_{\zee}\zee^2*_{\zee}R_2*_{\zee}\dotsb\]
  There are at most $4b_0$ maximal disjoint subgraphs of this
  form. Since $\betti(R_i)\geq 2,$ the number of valence two
  nonabelian vertex groups such that the two incident edge groups
  don't generate the first homology is bounded above by $b_0,$ we may
  assume that in such a sub-graph of groups, no incident edge group
  has trivial image in homology. Such vertex groups behave, when
  computing homology, like valence one vertex groups. Since the map
  from $L$ to $L^{ab}$ factors through the graph of groups obtained by
  abelianizing all rigid vertex groups, we only need to find $\betti$
  of graphs of groups of the following form
  \[*_{\zee}\zee^{\geq 2}*_{\zee}\zee^{\geq
    2}*_{\zee}\dotsb*_{\zee}\zee^{\geq 2}*_{\zee}\] relative to the
  copies of $\zee$ at the ends. If a graph of groups of this form has
  length $n,$ then it contributes at least $n-1$ to $\betti(L).$
\end{proof}

\mnote{defining strict}

\par It is important to see a limit group in terms of some of its
quotient limit groups.

\begin{definition}[Strict]
  \label{def:strict}
  Let $G$ be a finitely generated group and $L$ a limit group. A
  homomorphism $\rho\colon G\to L$ is \term{$\Mod(G,\Delta)$--strict}
  if, given a sequence of homomorphisms $f_n\colon L\to\free$
  converging to $L,$ there exists a sequence
  $\varphi_n\in\Mod(G,\Delta)$ such that $f_n\circ\rho\circ\varphi_n$
  is stable and converges to $G.$ Since the abelian \jsj\ ``contains''
  all other \gad s, a $\Mod(G,\Delta)$--strict homomorphisms is
  \latin{a fortiori} $\Mod(G)$--strict, or simply \term{strict}.

  A sequence of epimorphisms $L_0\onto L_1\onto\dotsb\onto L_n$ such
  that each map $L_i\onto L_{i+1}$ is strict is a \term{partial strict
    resolution} of $L_0.$ If $L_n$ is free then the sequence is a
  \term{strict resolution} of $L_0.$ The \term{height} of a (partial)
  strict resolution is its proper length. A (partial) strict
  resolution is \term{proper} if all the epimorphisms appearing are
  proper.
\end{definition}

\cite[Proposition~5.10]{sela::dgog1} asserts that limit groups admit
strict resolutions.

\begin{theorem}[\cite{sela::dgog1,bf::lg}]
\label{thr:strictconditions}
\par Fix a finitely generated group $G$ with \gad\ $\Delta$ and a
limit group $L.$ The following list of conditions is sufficient to
ensure that $\pi\colon G\to L$ is $\Mod(G,\Delta)$--strict.

\begin{itemize}
  \item All edge groups inject.
  \item All \qh\ subgroups have nonabelian image.
  \item All envelopes of rigid vertex groups inject.
  \item If $e$ is an edge of $\Delta$ then at least one of the
    inclusions of $G_e$ into a vertex group of the one edged splitting
    of $G$ induced by $e$ is maximal abelian.
\end{itemize}
\end{theorem}

In particular, if $G$ admits a strict map to a limit group then it is
also a limit group.


\begin{theorem}[\cite{louder::strict}]
  \label{thr:strictlengthbound}
  Let $\free_N\onto L_0\onto\dotsb\onto L_k$ be a sequence of
  proper strict epimorphisms of limit groups. Then $k\leq 3N.$

  If $\free_N\onto L_1\onto\dotsb\onto L_k=\free_M$ is a sequence of
  proper strict epimorphisms, then $k\leq 3(N-M).$
\end{theorem}

This theorem is also implied by~\cite[Theorem~0.4]{houcine-2008},
where Ould-Houcine shows that the the Cantor-Bendixon rank of the
closure of the space free groups free groups marked by $n$ elements,
points of which are the $n$--generated models of the universal theory
of $\free,$ that is, limit groups of rank $n,$ is bounded. The
Cantor-Bendixon rank of this space is exactly the length of a longest
strict resolution of $\free_n.$ Generalizing this fact to hyperbolic
groups is one of the principle difficulties in generalizing
Theorem~\ref{thr:krulldimension} for limit groups over hyperbolic
groups. The aforementioned theorems rely on linearity of $\free$ in an
essential way.

To completely avoid reliance on Theorem~\ref{thr:rankboundedbybetti}
we need to show that the depth of a limit group is controlled by its
rank.

\begin{theorem}
  \label{lem:depthbound}
  The depth of a limit group $L$ is bounded by $6\rk(L).$
\end{theorem}


\begin{proof}
  Let $L=G_1\onto G_2\onto\dotsc\onto G_{n\leq 3\rk(L)}$ be a strict
  resolution of $L.$ Let $H_i$ be the freely indecomposable free
  factors of $L.$ The restriction $G_i\vert_{\img(H_i)}=K_i$ of the
  strict resolution to the images of $H_i$ is a strict resolution of
  $H_i.$ Let $V$ be a vertex group of the cyclic \jsj\ of $H_i,$ and
  consider a freely indecomposable free factor $W$ of $V.$ Since the
  edge groups of $\jsj(L)$ contained in the induced decomposition of
  $W$ are elliptic in $\jsj(G_2),$ the envelopes of rigid vertex
  groups of $\jsj(W)$ are contained in envelopes of rigid vertex
  groups of $G_2,$ and are embedded in $K_3$ under the map $K_2\to
  K_3.$ Thus the sequence $K_3\onto\dotsc,$ restricted to the image of
  $W$ is a strict resolution of $W.$ By induction on the length of a
  shortest strict resolution, $W$ has depth at most $6\rk(L)-2.$ Since
  $W$ is at level one in the analysis lattice of $L,$ $L$ has depth at
  most $6\rk(L).$

\end{proof}

\section{Complexity classes of sequences}

\label{sec:compl-class-sequ}

\par Our proof of Theorem~\ref{thr:krulldimension} is by induction. To
do this we need to have a notion of sequence which allows us to attach
a suitable complexity. Unfortunately, it doesn't seem possible to
simply induct on the rank, first betti number, or depth of a sequence.

\par To cope with this we work with pairs of sequences of limit groups
with maps between them, rather than just with sequences of limit
groups. Let $\LL$ and $\mathcal{G}$ be sequences of limit
groups. Suppose that $j\mapsto i_j$ is a monotonically increasing
function from the index set for $\mathcal{G}$ to the index set for
$\LL.$ If, for all $j,$ there is a homomorphism
$\psi_j\colon\mathcal{G}(j)\to\LL(i_j),$ such that for all $j$
and $k$
\[\varphi_{i_j,i_k}\circ\psi_j=\psi_k\circ\varphi_{j,k}\]
we say that \term{$\mathcal{G}$ maps to $\LL$} and that $\psi$
is a \term{map of sequences}. To express this relationship we use the
familiar notation $\psi\colon\mathcal{G}\to\LL.$ If
$\psi\colon\mathcal{G}\to\H$ and
$\psi'\colon\H\to\LL$ then there is a map
$\psi'\circ\psi\colon\mathcal{G}\to\LL.$ If
$\psi\colon\mathcal{G}\to\LL$ then the \term{image} of
$\mathcal{G}$ is the sequence $((\img\mathcal{G})(j)),$
$\img\mathcal{G}(j)\define\psi_j(\mathcal{G}(j)),$ whose maps are the
restrictions of the maps from $\LL.$ Any
$\psi\colon\mathcal{G}\to\LL$ factors as
$\mathcal{G}\onto\img\mathcal{G}\into\LL.$ Since the notation
is unambiguous, we write ``$\img\G(j)$'' for ``$(\img\G)(j)$''.

A map $\psi$ of sequences is an \term{embedding} if every map $\psi_j$
is an embedding. In this case $\psi$ is written in the normal fashion.
Let $\psi\colon\mathcal{G}\into\LL$ be an inclusion. For each
$j$ let $\LL_{\psi}(i_j)$ be the lowest node in the cyclic
analysis lattice of $\LL(i_j)$ containing a conjugate of
$\psi_j(\mathcal{G}(j)),$ and set $d_j$ equal to the depth of
$\LL_{\psi}(i_j).$ Now let the \term{depth} of $\psi$ be
\[d(\psi)\define\max_j\set{d_j}\] If $\rho\colon\H\to\G$ then
$d(\rho)=d(\img(\H)\into\G).$

We will be interested in pairs of sequences which have maps in
both directions. There is a self-map of a sequence (where it is
defined) given by
$\varphi_{i,i+1}\colon\mathcal{G}(i)\to\mathcal{G}(i+1)$ for all
$i.$ This map is a ``shift.'' We denote it by $\varphi_+.$ We now
define ``resolutions,'' which are essentially sequences of
epimorphisms, to which we add groups, and regard the
additional groups as an auxiliary sequence. After this, we define
resolutions of sequences of subgroups, which are our main object of
interest.

Note that we have used the word ``resolution'' in two different
contexts. Firstly, there are resolutions of limit groups, which are
simply sequences of epimorphisms, there are strict resolutions, which
are resolutions in which every map is strict, and now there are
resolutions of sequences. A resolution of a group is always in the
first sense, and a resolution of a sequence is always a resolution in
the third sense.


\begin{definition}[Resolution]
  \label{def:resolutionofasequence}
  A \term{resolution} of a chain $\LL$ of limit groups is a
  chain $\H,$ indexed by $j=1..n,$ equipped with maps
  $\pi_j\colon\H(j)\onto\LL(i_j)$ for all $j,$ and
  $\psi_j\colon\LL(i_j)\onto\H(j+1),$ $j+1\leq n,$
  such that the following diagrams commute:

  \centerline{%
    \xymatrix{%
      \H(j)\ar@{->>}[d]^{\pi_j}\ar@{->>}[r]^{\varphi_{j,j+1}} & \H(j+1)\ar@{->>}[d]^{\pi_{j+1}} \\
      \LL(i_j)\ar@{->>}[r]^{\varphi_{i_j,i_{j+1}}}\ar@{->>}[ur]^{\psi_j} & \LL(i_{j+1})
}}
\end{definition}

If $\LL'$ is a subsequence of $\LL$ obtained by
deleting groups and composing maps then we write
$\LL'\subset\LL.$ Since the maps
$\LL'(j)\to\LL(i_j)$ are injective and surjective, we
write this suggestively as $\LL'\intoonto\LL.$ If
$\LL$ is a chain and $\LL'\intoonto\LL$ then
$\LL$ is \term{finer} than $\LL'$ and $\LL'$
is \term{coarser} than $\LL.$


If $\LL'$ is a sequence of subgroups of $\LL,$ $\LL'(n)<\LL(n),$ then this
relation is expressed by the notation $\LL'<\LL.$ If $\H$ is a
resolution of $\LL,$ then this relation is expressed by the notation
$\H\rto\LL.$ If the map $\H\to\LL$ is $\rho,$ then we
indicate it as a subscript on the ``$\rto$'':
$\H\rto_{\rho}\LL.$ The notation is supposed to evoke the
commutative diagram above. We may also write the reverse
$\LL\boxbslash\H$ to indicate that $\H$ is a
resolution of $\LL.$ We leave the maps implicit unless there is risk of
confusion.

\begin{definition}
  \label{def:basesequence}
  If $\H\rto_{\rho}\LL,$ then $\LL$ is the \term{base sequence}
  of the resolution. The \term{depth} of the resolution is the depth of
  $\rho.$
\end{definition}

Let $\mathcal{G}\into\LL$ be an inclusion. If $\H\rto\G$ is a
resolution then $\H$ is a \term{resolution of a subsequence of
  $\LL$}. We denote this relation by $\H\rto\LL,$ even though the maps
$\LL\to\H$ are only defined on subgroups. The depth of $\H\rto\LL$ is
the depth of $\mathcal{G}\into\LL.$ A resolution of a subsequence is
simply a pair of maps $\pi\colon\H\to\LL,$
$\psi\colon\img(\pi)\onto\H$\footnote{Strictly speaking we must drop
  the last element from $\img(\pi).$}, such that
$\psi\circ\pi=\varphi_+$ and $\pi\circ\psi=\varphi_+,$ where the first
$\varphi_+$ is the shift map for $\H$ and the second $\varphi_+$ is
the shift map for $\img(\pi).$ Notice that we cannot compose
resolutions without changing the sequences. If $\H\rto\mathcal{G}$ and
$\mathcal{G}\rto\LL$ then there is no resolution $\H\rto\mathcal{G},$
though there is a resolution $\H'\rto\LL,$ where $\H'$ is the sequence
obtained by omitting every other group of $\H.$

Suppose $\H\rto\LL$ and $R<\H(n).$ The
sequence of images $\R^n$ defined by
\[\R^n(m)\define\varphi_{n,m}(R)<\H(m),\quad m\geq n\] 
is a resolution of a subsequence of $\LL$ with the induced
maps.

Let $\H\rto_{\rho}\LL$ be a resolution of a subsequence. The
complexity of $\H\rto_{\rho}\LL$ is the quantity
\[
\comp(\H\rto_{\rho}\LL)=(\betti(\H),d(\rho))
\]
The length and proper length of $\H\rto\LL$ are the
length and proper length of $\H,$ respectively.

To compare complexities we use the following partial order:
\[
(b,d)\leq(b',d')\mbox{ if $b\leq b'$ and $d\leq d'$}
\]
The inequality is strict if at least one of the coordinate
inequalities is strict. Rather than have one base case for the
induction, we use a collection of base cases. The complexities
$\set{(b,2)}:{b\in\mathbb{N}}$ are \term{minimal}, and form the base
cases for the inductive proof of Theorem~\ref{thr:subseqkrull}. The
groups in such sequences are free products of elementary limit
groups. We finish this section by showing that
Theorem~\ref{thr:krulldimension} holds for sequences of minimal
complexity.

\begin{lemma}[Krull dimension: very low depth]
  \label{lem:krull-lowdepth}
  Sequences of proper epimorphisms of limit groups with depth at most
  $2$ have length controlled by the first betti number.
\end{lemma}

\begin{proof}
  A group of depth $2$ is a free product of free abelian, surface, and
  free groups. There are only $n_b<\infty$ such groups with a given betti
  number $b.$ Since limit groups are Hopfian, for a fixed value of
  $b,$ every sequence of proper epimorphisms of limit groups with
  first betti number $b$ has length at most $n_b.$
\end{proof}

\par This is not really necessary since the analysis of sequences of
minimal complexity is a sub-case of our more general analysis of
sequences of resolutions of limit groups and the observation that any
map with nonabelian image from a nonabelian elementary limit group to
another limit group is strict.

\begin{definition}
  $\seq(\LL,b,d)$ is the set of resolutions of subsequences
  \[\H\rto_{\rho}\LL\] 
  such that
  \[\comp(\H\rto_{\rho}\LL)\leq(b,d)\]
\end{definition}

Theorem~\ref{thr:krulldimension} follows formally from the following
theorem, which is what we aim to prove in this paper.

\begin{theorem}[Krull dimension for resolutions of subsequences]
  \label{thr:subseqkrull}
  There is a function $D(b,d),$ independent of $\LL,$ such that if
  ${\H\rto\LL}\in\seq(\LL,b,d)$ then
  $\Vert\H\Vert_{pl}\leq D.$
\end{theorem}

Theorem~\ref{thr:krulldimension} is an immediate consequence of
Theorem~\ref{thr:subseqkrull} since the trivial resolution
$\LL\rto_{id}\LL$ is an element of $\seq(\rk(\LL),\betti(\LL),\depth(\LL))$


\section{Degenerate maps}
\label{section::ai}

\par This section exploits Theorem~\ref{thr:strictlengthbound} for our
approach to Krull dimension for limit groups.  Given a sequence
$\LL$ of epimorphisms of limit groups, we construct a
resolution of $\LL$ whose homomorphisms all respect
\jsj\ decompositions. This will occupy sections~\ref{section::ai}
through~\ref{sec:qcjsjrespecting}. The first step in the construction
is to find, given $\LL,$ a resolution $\H$ of
$\LL$ such that every map whose range is in $\H$ is
as far from strict as possible.

\begin{definition}[Degenerate]
  \label{def:degenerate}
  \par An epimorphism $\varphi\colon L\onto L'$ of limit groups
  \term{has a strict factorization} if there is a quotient limit group
  $\psi\colon L\onto L_s$ and a strict morphism $\varphi_s\colon
  L_s\onto L'$ such that $\varphi_s\circ\psi=\varphi.$

  An epimorphism $\varphi\colon L\onto L'$ of limit groups is
  \term{degenerate} if it has no proper strict
  factorizations.
\end{definition}

The philosophy here is that a surjection of limit groups can be
decomposed into a strict resolution preceded by a degenerate map:

\begin{figure}[h]
  \centerline{%
    \xymatrix{%
      {} & G_k\ar[r]^{\mathrm{strict}} & \dotsb\ar[r]^{\mathrm{strict}} & G_2 \ar[r]^{\mathrm{strict}} & G_1\ar[d]^{\mathrm{strict}} \\
      L\ar[rrrr]\ar[urrrr]\ar[urrr]\ar[urr]\ar[ur]^{\mathrm{degenerate}}
      & & & & L' }}
\end{figure}

If $L\to L'$ isn't degenerate, then there is a $G_1$ adding
to the diagram $L\to L'$ as in the figure. Likewise, if $L\to G_1$
isn't degenerate, then there is a strict $G_2$ which can be
added to the figure. Proceeding in this way we find a sequence of
strict resolutions \[\G_k=(G_k\to\dotsb\to G_1)\] Since (partial) strict
resolutions have bounded length (Theorem~\ref{thr:strictlengthbound}),
this procedure terminates in finite time and the last map constructed,
$L\to G_k,$ is degenerate.


A resolution of a subsequence $\H\rto\LL$ is \term{maximal}
if, for all $j<j',$ the maps $\H(j)\to\H(j')$ and
$\img(\H)(i_j)\to\H(j')$ are degenerate. A chain is
\term{degenerate} if all compositions of maps in the chain are
degenerate.

\begin{theorem}
  \label{thr:nostrict}
  Let $\LL$ be a sequence of $N$--generated limit groups.
  Then for all $K$ there exist $M=M(N,K)$ such that if
  $\Vert\LL\Vert_{pl}>M$ then there exists a maximal
  surjective $\H\rto\LL$ such that
  $\Vert\H\Vert_{pl}>K.$
\end{theorem}


To prove Theorem~\ref{thr:nostrict} we introduce sequences of strict
resolutions as a formal device.

\begin{definition}
  A \term{sequence of strict resolutions} is a sequence
  $\mathscr{S}=(\mathcal{S}_i)_{i=1..n}$ of proper partial strict
  resolutions
  \[\mathcal{S}_i\define(\mathcal{S}_i(k_i)\onto\dotsb\onto\mathcal{S}_i(1))\] 
  equipped with homomorphisms
  \[\psi_i\colon\mathcal{S}_i(1)\onto\mathcal{S}_{i+1}(k_{i+1})\]
  Notice that we chose to index the partial strict resolutions by
  decreasing, rather than increasing, indices.

  The \term{height} of a sequence of strict resolutions is the length
  of the shortest partial strict resolution appearing in the sequence:
  $h(\mathscr{S})=\min\set{k_i}:{k_i=\Vert\mathcal{S}_i\Vert}.$

  A \term{refinement} of a sequence of strict resolutions
  $\mathscr{S}$ is a sequence of strict resolutions $\mathscr{S}'$
  such that if $\mathcal{S}_i$ is the $i$-th strict resolution in
  $\mathscr{S}$ and $\mathcal{S}'_i$ is the $i$-th strict resolution
  in $\mathscr{S}'$ then the resolutions $\mathcal{S}_i$ and
  $(\mathcal{S}'_i(j))_{j\leq k_i}$ coincide and the composition of
  maps $\mathcal{S}'_{i-1}(1)\to\mathcal{S}'_i(k_i)$ agrees with
  $\psi_i.$ A refinement is \term{proper} if $k'_i>k_i$ for some $i.$

  A \term{subsequence of a sequence of strict resolutions}
  $\mathscr{S}$ is a sequence of strict resolutions $\mathscr{S}'$
  such that $\mathscr{S}'$ is obtained from $\mathscr{S}$ by deleting
  some of the strict resolutions appearing in $\mathscr{S}$ and
  composing maps.

  The \term{length of a sequence of strict resolutions} $\mathscr{S}$
  is the number of strict resolutions appearing, and is denoted by
  $\Vert\mathscr{S}\Vert.$
\end{definition}

\par Subsequences and refinements of sequences of strict resolutions
are illustrated in Figure~\ref{fig:seqstrict}, as is the relationship
between sequences of strict resolutions and resolutions of
sequences. If $\mathscr{S}$ is a sequence of strict resolutions such
that $(\mathcal{S}_i(1))$ appears as a subsequence of subgroups of a
sequence $\LL,$ that is, there is an inclusion
$(\mathcal{S}_i(1))\intoonto\LL,$ then the sequence
$(\mathcal{S}_i(k_i))$ is a resolution of $\LL.$

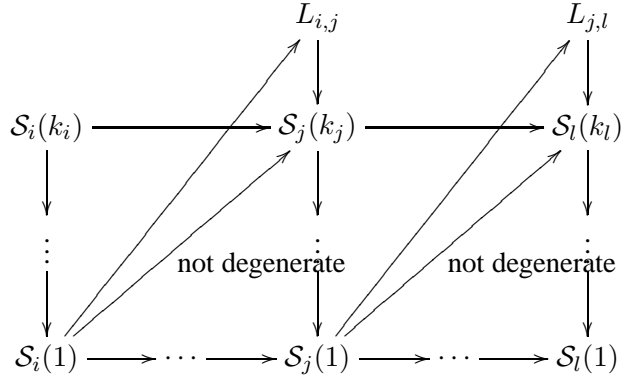
\begin{figure}[ht]
\centerline{%
  \xymatrix{%
                                &         & L_{i,j}\ar[d]                           &        & L_{j,l}\ar[d] \\
 \mathcal{S}_i(k_i)\ar[rr]\ar[d] &        & \mathcal{S}_j(k_j)\ar[rr]\ar[d]         &        & \mathcal{S}_l(k_l)\ar[d] \\
 \vdots\ar[d]                    &         &  \vdots\ar[d]                          &        & \vdots\ar[d] \\
 \mathcal{S}_i(1)\ar[r]\ar[uurr]_{\mbox{not degenerate}}\ar[uuurr] & \dotsb\ar[r] &  \mathcal{S}_j(1)\ar[r]\ar[uuurr]\ar[uurr]_{\mbox{not degenerate}} & \dotsb\ar[r] & \mathcal{S}_l(1)
}}
\caption{Raising the height.}
\label{fig:seqstrict}
\end{figure}

\begin{lemma}
  \label{lem:noproperrefinements}
  Fix $N.$ For all $K$ there exists $M=M(K,N)$ such that if
  $\mathscr{S}=(\mathcal{S}_i)$ is a sequence of strict resolutions
  and $\rk(\mathcal{S}_1(k_1))=N$ then if $\Vert\mathscr{S}\Vert>M$
  then there is a refinement $\mathscr{S}'$ of a subsequence of
  $\mathscr{S}$ such that $\Vert\mathscr{S}'\Vert\geq K$ and no
  subsequence of $\mathscr{S}'$ admits a proper refinement.
\end{lemma}

\par We now prove Theorem~\ref{thr:nostrict}, assuming
Lemma~\ref{lem:noproperrefinements}.

\begin{proof}[Proof of Theorem~\ref{thr:nostrict}]
  Choose $K>0,$ and let $\LL$ be a sequence of length $M(K+1,N).$  Let
  $\mathcal{S}_i$ be the trivial partial strict resolution $(\LL(i))$
  consisting of a single group. Set $\mathscr{S}=(\mathcal{S}_i).$  By
  Lemma~\ref{lem:noproperrefinements} there is a refinement
  $\mathscr{S}'$ of a subsequence of $\mathscr{S}$ with
  $\Vert\mathscr{S}'\Vert\geq K+1$ with the property that no
  subsequence of $\mathscr{S}'$ admits a proper refinement. In
  particular, if $\mathscr{S}'=\set{\mathcal{S}'_i}_{i=1..K+1}$ then
  every map $\mathcal{S}'_i(1)\to\mathcal{S}_j(k'_j)$ is degenerate.
  If $i>1$ then $\mathcal{S}'_i(k'_i)\to\mathcal{S}_j(k'_j)$ is
  degenerate, otherwise a subsequence of $\mathscr{S}'$ admits a
  proper refinement since
  $\mathcal{S}'_{i'}(1)\to\mathcal{S}'_j(k'_j)$ factors through
  $\mathcal{S}'_i(k'_i)\to\mathcal{S}'_j(k'_j)$ for $i'<i<j.$ By
  removing the first element of $\mathscr{S}'$ we obtain the desired
  sequence: set $\H(i)=\mathcal{S}'_i(k'_i).$
\end{proof}

\par We now prove Lemma~\ref{lem:noproperrefinements}. The proof
relies on the Ramsey theorem.\footnote{What is actually needed lies
  somewhere between the Ramsey theorem and the pigeonhole principle.}

\begin{theorem}[Ramsey Theorem (see~\cite{graham90})]
  Let $K_n$ be the complete graph on $n$ vertices, and let $M>0.$ Then
  there exists $R(M)$ so that if $n>R(M)$ and the edges of $K_n$ are
  bicolored, then there exists a complete monochromatic subgraph
  $K_M\subset K_n.$
\end{theorem}

\begin{proof}[Proof of Lemma~\ref{lem:noproperrefinements}]
  The proof is by induction on the height of a sequence of strict
  resolutions. If $\mathscr{S}$ is a sequence of strict resolutions of
  $N$--generated limit groups then $h(\mathscr{S})\leq 3N.$ Let
  $G(\mathscr{S})$ be the complete graph whose vertex set is the index
  set for $\mathscr{S}.$ Color an edge $(i,j),$ $i<j$ white if
  $\mathcal{S}_i(1)\to\mathcal{S}_j(k_j)$ factors through a proper
  strict homomorphism $L_{i,j}\onto\mathcal{S}_j(k_j),$ and black
  otherwise. If $\Vert\mathcal{S}\Vert>R(K+1)$ then there exists a
  complete monochromatic subgraph $G_m$ of $G(\mathscr{S})$ with $K+1$
  vertices. If $G_m$ is colored black, then set $\mathscr{S}'$ to be
  the subsequence of $\mathscr{S}$ indexed by the vertices of $G_m$
  sans first vertex. Then $\mathscr{S}'$ has length $K$ and admits no
  proper refinement.

  If $G_m$ is colored white, then let $\mathscr{S}''$ be the
  refinement of the subsequence of $\mathscr{S}$ indexed by the
  vertices of $G_m,$ and whose strict resolutions are constructed as
  follows: If $i,j\in G_m^{(0)},$ $i<j,$ such that if $l\in G_m^{(0)}$
  then $l\leq i$ or $l\geq j,$ then prepend $L_{i,j}$ to
  $\mathcal{S}_j$ to build a proper strict
  resolution \[L_{i,j}\onto\mathcal{S}_j(k_j)\onto\dotsb\onto\mathcal{S}_j(1)\]

  Remove the first strict resolution from $\mathscr{S}''$ and call the
  resulting sequence of strict resolutions $\mathscr{S}'.$ Since
  $L_{i,j}\onto\mathcal{S}_j(k_j)$ is strict and proper for all
  $j>1,$ $h(\mathscr{S}')>h(\mathscr{S}).$

  If $h(\mathscr{S})=3N$ then $\mathscr{S}$ satisfies the theorem. Set
  $R_1(K)=R(K+1).$ Then if $\mathscr{S}$ has length at least
  $M(N,K)=(R_1)^{3\cdot N}(K)$ it has a subsequence of length
  \mnote{do that thing with the curly brace underneath maybe}
  $(R_1)^{3\cdot N-1}(K)$ admitting a refinement of height at least
  $h(\mathscr{S})+1.$ Inducting on the height (which takes at most $3
  N$ steps by the bound on the length of a strict resolution) we see
  that $\mathscr{S}$ has a subsequence of length $K$ which has a
  refinement which admits no proper refinements.
\end{proof}


\section{Constructing strict homomorphisms}
\label{sec:constructstrict}

\par So far we have only used strictness in a purely formal way,
ignoring its geometric content. In this section we put
Theorem~\ref{thr:strictconditions} to work, and show explicitly how to
factor a homomorphism $G\to H$ of limit groups through a strict
$\Phi_s(G)\to H.$ In section~\ref{sec:degenerations} we construct a
complexity ``$\scott,$'' modeled on the Scott complexity, which is
nondecreasing under degenerate maps and takes boundedly many values
for limit groups with a given first betti number. We then prove a
theorem which says that if equality of $\scott$ holds under a
degenerate map then the \jsj\ decompositions of the groups in question
strongly resemble one another, or are ``aligned''. Combining this with
Theorem~\ref{thr:nostrict} we have a method for aligning \jsj\
decompositions. Before we begin, we give an example which should serve
to motivate the constructions of
subsections~\ref{subsec:freelydecomposable}
and~\ref{subsection::freelyindecomposable}.

\begin{example}
  Let $\varphi\colon G\to H$ be a homomorphism from a finitely
  generated group with a one edged abelian splitting $G=G_1*_EG_2$ to
  a limit group. Let $e\in E$ and let $\tau$ be the generalized Dehn
  twist in $E$ by $e$: $\tau(g)=g$ if $g\in G_1$ and
  $\tau(g)=ege^{-1}$ if $g\in G_2.$ Consider a sequence of
  homomorphisms $f_n\colon H\to\free$ which converges to $H.$ Let
  $g_n$ be the sequence $f_n\circ\varphi\circ\tau^{m(n)}.$ We choose
  $m(n)$ later.

  Pass to a convergent subsequence of $g_n,$ let $\Phi_sG$ be the
  quotient of $G$ by the stable kernel, and let $\eta\colon G\onto
  \Phi_sG$ be the quotient map. Suppose $E$ has trivial image in
  $H.$ Then we define $\Phi_sG$ differently, and declare it to be
  $\varphi(G_1/E)*\varphi(G_2/E).$ Note also that $\tau$ pushes
  forward to an automorphism $\tau'$ of $\Phi_sG.$ If one of $G_1$ or
  $G_2$ has abelian image in $H$ then $\tau$ also pushes forward to
  $\Phi_sG.$

  Suppose that $\tau$ pushes forward to an automorphism $\tau'$ of
  $\Phi_sG$ and that neither $G_1$ nor $G_2$ has abelian image in
  $H.$ Then $\varphi\circ f_n=g_n\circ(\tau')^{-m(n)}\circ\eta.$  Then
  $\ker(\eta)<\ker(\varphi)$ and there is an induced strict
  homomorphism $\Phi_s\varphi \colon \Phi_sG\to H$ whose composition
  with $\eta$ is $\varphi.$ If $m(n)$ is sufficiently large the
  limiting action of $\Phi_sG$ on the $\mathbb{R}$--tree for the
  sequence $g_n$ is simplicial, induces a graph of groups
  decomposition $\Delta$ of $\Phi_sG$ with one edge, and $G_1$ and
  $G_2$ both have ``elliptic'' images in $\Delta$: the graph of groups
  decomposition has the form $\overline{G}_1*_{E'}\overline{G}_2,$
  $G_1$ maps to the envelope of $\overline{G}_1,$ $G_2$ maps to the
  envelope of $\overline{G}_2,$ $E$ maps to the centralizer of $E'$
  and $\eta$ respects incidence and conjugacy data of graphs of
  groups, that is, $\tau$ pushes forward to $\tau'.$
\end{example}

This is a motivating example for our alignment of \jsj\ decompositions
approach to Theorem~\ref{thr:krulldimension}: If the homomorphism from
the example above is degenerate, a condition which can be created
given sufficiently long sequences of epimorphisms of limit groups by
Theorem~\ref{thr:nostrict}, we see roughly that one of three things
can happen to a one-edged splitting of the domain: either the
homomorphism factors through a free product seen by the edge, a vertex
group has abelian image, or the target group splits over the
centralizer of the image of the edge group. The above example uses
limiting actions to suggest ways in which the range of a degenerate
map inherits splittings from the domain. Rather than take this
approach throughout, the kind of information which must be recorded
requires that we manually construct the group $\Phi_sG$ (hence $H$)
from the data $G$ and $\varphi.$

\subsection{Freely decomposable groups} 
\label{subsec:freelydecomposable}

In this subsection we assign to a homomorphism $\varphi\colon G\onto
H$ a strict homomorphism $\Phi_s\varphi\colon \Phi_sG\onto H$ and a
complexity $\scott$ which is nondecreasing for such $G\onto \Phi_sG.$

\begin{definition}[\cite{scottcoherent}]
  \label{def:scottcomplexity}
  Let $\grushko{G}{p}{q}$ be a Grushko factorization of a finitely
  generated group $G.$ The \term{Scott Complexity} of $G$ is the
  lexicographically ordered pair $\scott(G)=(p+q,q).$

  Let $\varphi\colon G\onto H$ be a homomorphism of finitely generated
  groups $G$ and
  $H.$ Then 
  \[\scott(\varphi)\define\max\set{\scott(L/K)}:{\varphi\mbox{ factors through }L/K}\]
\end{definition}

\par Let $G=\grushko{G}{p_{G}}{q_{G}}$ and
$H=\grushko{H}{p_{H}}{q_{H}}.$ Given a homomorphism $\varphi\colon
G\onto H$ we define a quotient $\eta\colon G\onto \Phi_sG,$ and an
induced map $\Phi_s\varphi\colon \Phi_sG\onto H$ such that
$\Phi_s\varphi\circ\eta=\varphi.$ For each freely indecomposable free
factor $G_i$ of $G,$ let $L^i$ be a group with highest Scott
complexity that $\varphi\vert_{G_i}\colon G_i\onto\varphi(G_i)$
factors through.  Set \[\Phi'_sG\define L^1*\dotsb* L^{p_G}*\free_{q_G}\]
Each $L^i$ has a Grushko decomposition $\grushko{L^i}{p_i}{q_i}$;
replace each $L^i_j$ by its image in $H$ and call the resulting group
$\Phi_sG.$ There is an induced map $\Phi_s\varphi\colon \Phi_sG\onto H.$

\par Before we begin show that the Scott complexity behaves well under
degenerate maps, we need a lemma to show that the homomorphism
constructed above is strict.

\begin{lemma}
  \label{lem:strict::freelydecomposable}
  Let $\pi\colon G\to H$ have nonabelian image, $H$ a limit
  group. Suppose $\pi$ is injective on freely indecomposable free
  factors of $G.$ Then $\pi$ is strict. In particular, $G$ is a limit
  group.
\end{lemma}

The proof of Lemma~\ref{lem:strict::freelydecomposable} relies on the
following bit of folklore: Let $g_i,$ $i=0,\dotsc,m,$ and $t$ be
nontrivial words in a free group. If $\left[g_i,t\right]\neq 1$ for
all $i,$ then for sufficiently large $\min_i\set{\vert n_i\vert},$ the
word
\[g_0t^{n_1}g_1\dotsb g_{m-1}t^{n_m}g_m\] 
is not trivial.

\begin{proof}[Proof of Lemma~\ref{lem:strict::freelydecomposable}]
  The strategy is to find automorphisms $\phi_n$ of $G$ and a sequence
  of homomorphisms $f_n\colon H\to\free$ converging to $H$ such that
  $\stabker(f_n\circ\pi\circ\phi_n)=\set{1}.$ Express $G$ as a most
  refined free product
  \[G_1*\dotsb*G_{k-1}*G_k*\dotsb*G_{l-1}*G_l*\dotsb*G_{p+q}\] with
  $G_i$ nonabelian for $i\leq k-1,$ noncyclic free abelian for
  $l-1\geq i\geq k,$ and $G_i\cong\zee$ for $i\geq l.$ Suppose a basis
  element $x$ generating some $G_i$ $i\geq l$ has trivial image in
  $H.$ Let $\set{g_{i,j}}$ be a generating set for $G_i.$ Precompose
  $\pi$ by the automorphism which maps $x$ to $xg$ for some $g\in G_1$
  and is the identity on the rest of $G.$ In this way, arrange that no
  element of a fixed basis for the free part of a Grushko
  decomposition of $G$ has trivial image in $H.$ Under the new map
  $\pi,$ every free factor in some fixed Grushko free factorization
  embeds in $H.$

  Let $B_{i,n}$ be the ball of radius $n$ in the Cayley graph of $G_i$
  with respect to $\set{g_{i,j}}.$ For each $n,$ choose a homomorphism
  $f_n\colon H\to\free$ such that $f_n$ has nonabelian image and
  embeds $\pi(B_{i,n})$ for each $i.$

  Suppose that $G_1$ is freely indecomposable nonabelian and that
  $f_n$ has nonabelian image when restricted to $\pi(G_1).$ Since
  $f_n(\pi(G_1))$ is nonabelian there is an element $c_n$ in the $G_1$
  such that $\left[f_n(\pi(c_n)),f_n(\pi(g))\right]\neq 1$ for all
  $g\in\sqcup B_{i,n}.$ Fix $m$ and choose integers $m_i,$ $i>1,$ such
  that $\vert m_i-m_j\vert>m$ for $i\neq j.$ Let $h=h_1\dotsb h_t$ be
  a word in $\sqcup_i B_{i,n}$ such that for all $l,$ $h_l$ and
  $h_{l+1}$ are contained in distinct $B_{i,n}$'s, let $\Omega_{n,s}$
  be the collection of all such words with length at most $s,$ and let
  $i(h_j)$ be the index $i$ such that $h_j\in B_{i(h_j),n}.$ Let
  $\varphi_m$ be the automorphism of $G$ which is the identity on
  $G_1$ and which maps $G_i$ to $c_n^{m_i}G_ic_n^{-m_i}.$ The image of
  $h$ in $\free$ is
  \[f_n\pi(c_n)^{m_{i(h_1)}}f_n\pi(h_1)f_n\pi(c_n)^{-m_{i(h_1)}}f_n\pi(c_n)^{m_{i(h_2)}}f_n\pi(h_2)\dotsb\] Since $h_i$ and $h_{i+1}$ are contained in distinct 
  factors of $G,$ the terms
  \[f_n\pi(c_n)^{-m_{i(h_j)}}f_n\pi(c_n)^{m_{i(h_{j+1})}}\] are at
  least $m$--th powers of the image of $c_n.$ By the folklore
  mentioned prior to this proof, for sufficiently large $m,$
  $f_n\circ\pi\circ\varphi_m(h)\neq 1.$ Since $\Omega_{n,s}$ has
  finitely many elements, choose $m$ large enough so that
  $f_n\circ\pi\circ\varphi_m$ embeds $\Omega_{n,s}.$  Thus we may
  treat $m$ as a function of $n$ and $s.$ The family
  $\Omega_{n,s}$ exhausts $G,$ and since
  $f_n\circ\pi\circ\varphi_{m(n,s)}$ embeds $\Omega_{n,s},$
  \[\Ker_{n,s\to\infty}(f_n\circ\pi\circ\varphi_{m(n,s)})=\set{1}\]
  Therefore $\pi$ is strict by definition.

  Now suppose all indecomposable factors of $G$ are free abelian. Then
  the image of, without loss of generality, $G_1*G_2$ is
  nonabelian. Simply repeat the argument using the factor $G_1*G_2$
  rather than $G_1.$
\end{proof}

\begin{theorem}[Scott complexity is monotone (\cite{swarup::scott})]
  \label{thr:scott::freelydecomposable}
  If $G\onto H$ is a degenerate map of nonabelian limit groups, then
  $\rk(G)\geq\rk(H).$ If $\rk(G)=\rk(H)$ then
  $\scott(G)\leq\scott(H).$ If $\scott(G)=\scott(H)$ then $p_i=1,$
  $q_i=0$ (Definition~\ref{def:scottcomplexity}), and no map
  $G_i\onto\img(G_i)$ is trivial or factors through a free product.
\end{theorem}

\begin{proof}
  Clearly $\rk(G)\geq\rk(H).$ If their ranks are equal then no $p_i$
  is zero. Let $r_i=\rk(G^i).$ If $p_i=0$ then the induced map
  $\free_{r_i}\onto\free_{q_i}$ cannot be an isomorphism since free
  groups are Hopfian. Thus if some $p_i=0$ then $\rk(G)>\rk(H).$ If
  some $L^i_j$ has trivial image in $H$ then $\rk(G)>\rk(H),$ thus all
  $L^i_j$ have nontrivial image, and \latin{a fortiori}, all $G_i.$
  Since $\scott(L^i)$ is maximal out of all groups
  $\varphi\vert_{G_i}$ factors through, every $L^i_j$ has freely
  indecomposable image in $H.$

  Since $\Phi_sG\to H$ is injective on freely indecomposable free
  factors and has nonabelian image, by
  Lemma~\ref{lem:strict::freelydecomposable}, it is strict, and since
  $\varphi$ is degenerate, $\Phi_s\varphi$ is an isomorphism. Thus
  $p_{H}=\sum_ip_i$ and $q_{H}=q_{G}+\sum_iq_i.$

  Computing that the Scott complexity of $H$ is at most that of $G$ is
  now easy: \[p_{H}+q_{H}=\sum_ip_i+q_{G}+\sum_iq_i\] Since each
  $p_i\geq 1,$ $p_{H}\geq p_{G}$ and $q_{H}\geq q_{G}.$  Thus
  $p_{H}+q_{H}\geq p_{G}+q_{G}.$ If there is equality in the first
  coordinate then $p_i$ must be $1$ for all $i,$ and if there is
  equality in the second coordinate then $q_i=0$ for all $i.$ Thus
  $\scott(G)\leq\scott(H),$ with equality only if
  $\scott(\varphi\vert_{G_i})=(1,0),$ that is, no restriction
  $G_i\onto\img(G_i)$ factors through a nontrivial free product.
\end{proof}

\par In light of this, we set
$c_{fd}(G)=(\rk(G),\betti(G),-\scott(G)).$ If $G\onto H$ is degenerate
and $c_{fd}(G)=c_{fd}(H)$ then for each $i$ there is a unique $j(i)$
such that $G_{j(i)}\onto H_i$ (up to conjugacy), and
$H\cong*_i\img(G_i)*\free_q.$  By
Theorem~\ref{thr:scott::freelydecomposable}, $c_{fd}$ is a
nonincreasing function under degenerate maps.  If $G_j$ is abelian and
$\varphi(G_j)=H_i,$ then $H_i$ is abelian. If $c_{fd}(G)=c_{fd}(H)$
then for all abelian freely indecomposable free factors $G_{j(i)},$
$\varphi\vert_{G_{j(i)}}\colon G_{j(i)}\onto H_i$ is an isomorphism, and for
nonabelian freely indecomposable free factors, is degenerate.

A chain $\LL$ of limit groups is \term{indecomposable} if no map
$\LL(i)\to\LL(j)$ factors through a free product. Not only are the
groups in an indecomposable sequence indecomposable, all compositions
of maps are too. By Theorem~\ref{thr:nostrict} and the fact that for
limit groups of a fixed rank there are only finitely many values the
Scott complexity can take, we have the following ``alignment
theorem.''

\begin{theorem}[Reduction to indecomposable sequences]
  \label{thr:reduction-to-indecomposable}
  Let $\LL$ be a rank $N$ sequence of epimorphisms of limit groups. For
  all $K$ there exists $M=M(K,N)$ such that if $\Vert\LL\Vert_{pl}\geq
  M$ then there exists a maximal resolution $\wt{\LL}\rto\LL$ such that
  $\Vert\wt{\LL}\Vert_{pl}>K$ and $c_{fd}$ is constant along $\wt{\LL}.$

  In particular, $\wt{\LL}$ splits as a graded free product of
  sequences
  \[\widetilde{\LL}=\widetilde{\LL}_1*\dotsb*\widetilde{\LL}_p*\mathcal{F}\]
  where $\mathcal{F}$ is the constant sequence $(\free_q)$ for some
  $q.$ The sequences $\wt{\LL}_i\rto\LL$ are indecomposable maximal
  resolutions of their images. \hfill \qedsymbol
\end{theorem}

\subsection{Freely indecomposable groups} 
\label{subsection::freelyindecomposable}

\par Our goal is to understand degenerate maps of limit groups. In the
previous subsection we saw that for freely decomposable groups, a
degenerate map either decreases the rank, decreases the first betti
number, raises the Scott complexity or respects Grushko
factorizations. Scott complexity is blind to the behavior of
restrictions of degenerate maps to the freely indecomposable factors
in the Grushko decompositions of limit groups. In this subsection we
construct a natural generalization of Scott complexity to
\jsj\ decompositions. In this subsection $\varphi\colon G\onto H$
always has Scott complexity $(1,0),$ that is, $G$ is freely
indecomposable and $\varphi$ doesn't factor through a free product.
Such a map is \term{indecomposable}.

\par In this subsection we mimic the construction of $\Phi_sG$ for
indecomposable homomorphisms: given an indecomposable $\varphi\colon
G\onto H,$ we build a quotient group $\Phi_sG$ of $G$ and a strict
homomorphism $\Phi_sG\onto H$ which $\varphi$ factors through. In
particular, if $\varphi$ is degenerate, then $\Phi_sG\onto H$ is an
isomorphism. The construction of $\Phi_s$ is very explicit, and will
be used in the next section to show that the \jsj\ decomposition of
$\Phi_sG$ strongly resembles the \jsj\ decomposition of $G.$

\par Our main tool is a variation on the Bestvina-Feighn folding
machinery~\cite{bf::bounding}, which is roughly as follows. If $T$ and
$S$ are faithful simplicial $G$-trees and $\phi\colon T\to S$ is a
$G$-equivariant morphism then $\phi$ can be realized as a composition
of elementary folds \french{\`{a} la} Stallings. As in the case of a
free group, one may ignore the equivariance by studying the situation
in the quotient graphs $T/G$ and $S/G.$ The catch for non-free actions
is that vertex and edge stabilizers are not trivial and we must
consider a wider variety of morphisms of graphs of groups. The
elementary folds one needs to consider are listed in in
Figure~\ref{basicfolds}.

\begin{figure}[ht]
  \psfrag{e1e2}{$\group{E_1,E_2}$}
  \psfrag{X}{$X$}
  \psfrag{Y}{$Y$}
  \psfrag{e1}{$E_1$}
  \psfrag{e2}{$E_2$}
  \psfrag{V}{$V$}
  \psfrag{Eg}{$\group{E,g}$}
  \psfrag{Xg}{$\group{X,g}$}
  \psfrag{e}{$E$}
  \psfrag{eg}{$\group{E,g}$}
  \psfrag{pulling}{pulling}
  \psfrag{folding}{folding}
  \psfrag{subgraph}{subgraph}
  \psfrag{type1}{Type I}
  \psfrag{type2}{Type II}
  \psfrag{type3}{Type III}
  \psfrag{type4}{Type IV}
  \psfrag{xye1e2}{$\group{X,Y,E_1,E_2,g}$}
  \psfrag{XY}{$\group{X,Y}$}
  \psfrag{leadsto}{$\rightsquigarrow$}
  \centerline{
    \includegraphics[scale=0.8]{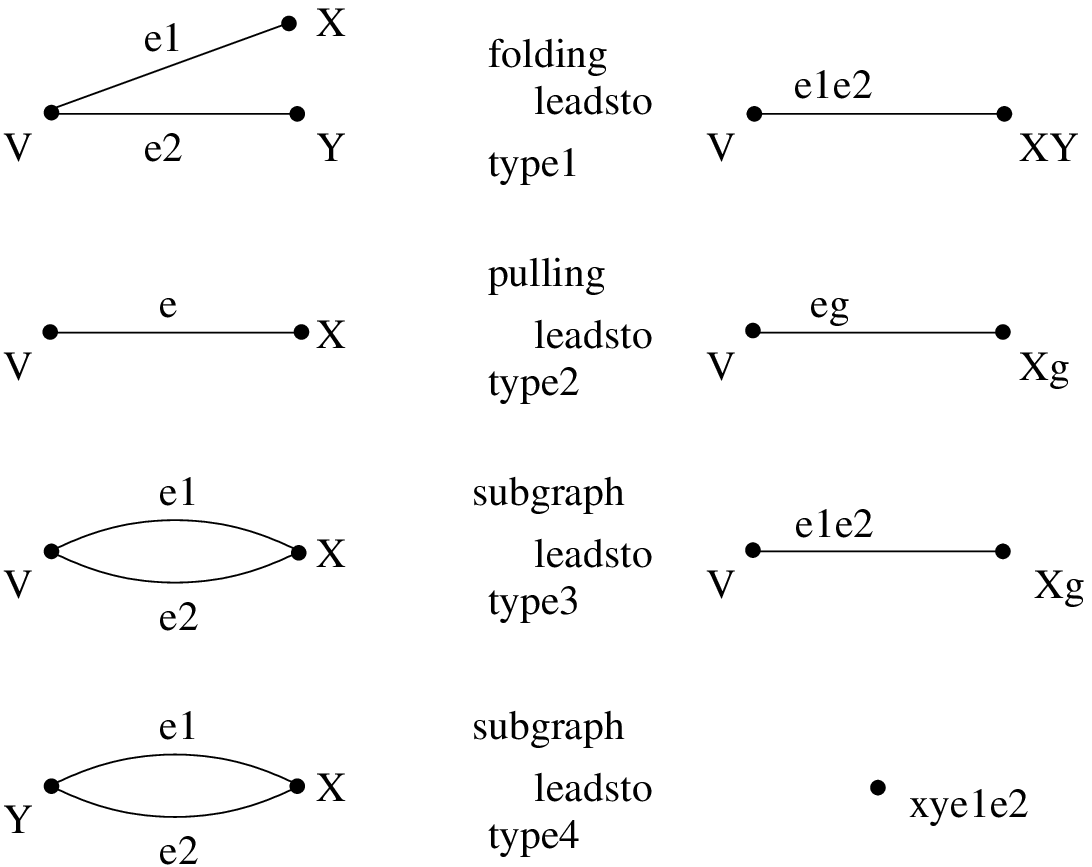}
  }
  \caption{The basic folds. In type III and IV, $g$ is the stable
    letter of the \hnn\ extension. (Lifted from~\cite{bf::bounding}, by
    way of~\cite{dunwoody::folding})}
  \label{basicfolds}
\end{figure}

\par A type I fold is the decorated version of a Stallings fold. Type
II folds have no effect on the quotient graph $T/G,$ but if
$\widetilde{v}$ is a lift of $v$ then the set of lifts of $e$ (as an
oriented edge) adjacent to $v$ are in one to one correspondence with
the cosets of $E$ in $V.$ If $g$ is pulled across the edge then the
cosets of $E$ fall into cosets of $\group{ E,g},$ and $T$ is
folded accordingly. A type III fold is a composition of subdivision
followed by type I and type IV folds.

\par Slightly more general than a type IV fold is the following move,
which we call a type IVB fold: Let $\Delta$ be a graph of groups
decomposition of $G$ and $\Gamma$ a connected subgraph of $\Delta,$
and let $G_{\Gamma}$ be the subgroup of $G$ carried by
$\Gamma.$ Collapse the graph $\Delta$ to form a graph $\Delta/\Gamma$
with distinguished vertex $\gamma$ which is the image of
$\Gamma.$ Assign the group $G_{\Gamma}$ to $\gamma,$ retaining the
labels on all edges and vertices not contained in $\Gamma,$ to form a
graph of groups decomposition $\Delta/\Gamma$ of $G.$ Then
$\Delta\to\Delta/\Gamma$ is a composition of collapses and type IV
folds. If $\Gamma$ is disconnected, it is understood that the move is
carried out component by component.

\par For general $T\to S,$ Dunwoody adds \term{vertex morphisms} to
this list \cite{dunwoody::folding}, wherein one is allowed to pass to
a quotient of a vertex group. We will confine ourselves to a
particular type of vertex morphism, the \term{strict vertex
  morphisms}, tailor-made for producing desirable quotients of limit
groups.


\subsubsection{Almost-strict homomorphisms}%

\par In this sub-subsection we show that if $\varphi\colon G\onto H$ is
indecomposable and $G$ is equipped with a \gad\ $\Delta,$ then
$\varphi$ factors through an \term{almost-strict}
(Definition~\ref{def:almoststrict}) homomorphism $\phias G\onto H$ such
that $\phias G$ has a \gad\ which resembles a degenerated and blown
up $\Delta.$ To begin we show that if an indecomposable homomorphism
of limit groups doesn't factor through a free product then edge groups
don't have trivial image.

\begin{lemma}
  \label{edgeslive}
  Suppose $G$ and $H$ are limit groups, $\varphi\colon G\onto H$
  indecomposable. If $G$ splits nontrivially over an abelian subgroup
  $E$ then $\varphi(E)\neq\set{1}.$
\end{lemma}

\begin{proof}
  If the edge associated to $E$ is nonseparating then $\varphi$
  factors through a group of the form $G'*\zee,$ and if the edge is
  separating, then one of two possibilities occurs: Let $G_1$ and
  $G_2$ be the vertex groups of the one edged splitting. If $E$ has
  trivial image then at least one $G_i$ (say $1$) has trivial image in
  $H.$ If this is the case then $\varphi\vert_{G_1}$ factors through
  abelianization of $G_1.$ If $G_1$ is nonabelian, then since limit
  groups have nontrivial homology relative to abelian subgroups,
  $\varphi$ factors through a nontrivial free product $G_1^{ab}/E*G_2/
  E.$ If $G_1$ is abelian, then $E$ must have index $1$ in $G_1,$
  otherwise $G_1/E$ is nontrivial, again contradicting
  indecomposability of $\varphi.$
\end{proof}


\par Before we begin the construction of $\Phi_s$ in earnest, we give
a construction which vaguely resembles factoring through a group with
maximal Scott complexity
(Subsection~\ref{subsec:freelydecomposable}, particularly the
construction in Definition~\ref{def:scottcomplexity}.).


\begin{definition}
  \label{def:vertexscott}
  Let $G$ be a finitely generated group with a graph of groups
  decomposition $\Delta,$ $G\cong G'*_{A_i}B_i,$ $B_i$ abelian. The
  complexity $c_a(G,\Delta)$ is the sum $\sum_i\rk(B_i/A_i).$ Let
  $\E$ be a collection of abelian subgroups of
  $G.$ Then 
  \[c_a(G;\E) \define 
    \max_{\Delta}\set{c_a(G,\Delta)}:{E\in\E\Rightarrow
      \mbox{ $E$ conjugate into some $B_i$ of $\Delta$}}
    \] 
  If $\varphi\colon G\to H,$ $\E$ a family of abelian
  subgroups of $G,$ let
  $\varphi(\E)=\set{\varphi(E)}:{E\in\E}.$ Then
  set 
  \[
    c_a(\varphi,\E) \define 
    \max_{G'}\set{c_a(G',\varphi'(\E))}:{\varphi\mbox{
        factors through a map $\varphi'\colon G\onto G'$}}
    \] 
  If $\varphi\colon G\to H$ then we say that \term{$G'$ realizes
    $c_a(\varphi)$} if $\varphi$ factors through $\varphi'\colon
  G\onto G'$ and
  $c_a(G',\varphi'(\E))=c_a(\varphi,\E)$
\end{definition}

\par As an intermediate step between $G$ and $\Phi_sG$ we construct an
intermediate quotient $\phias G$ and a homomorphism $\phias G\onto H$
which is \term{almost-strict}.

\begin{definition}[Almost-strict]
  \label{def:almoststrict}
  Let $\varphi\colon G\onto H$ be indecomposable, $H$ a limit group,
  and $\Delta$ a \gad\ of $G.$ If $\varphi$ satisfies the following
  axioms it is \term{almost-strict} (with respect to $\Delta$).
  \begin{itemize}
  \item[AS1:] $\varphi$ embeds rigid vertex groups of $\Delta$ in $H.$
  \item[AS2:] $\varphi$ embeds edge groups of $\Delta$ in $H.$
  \item[AS3:] $\varphi(\mbox{\qh\ subgroups})$ are nonabelian.
  \item[AS4:] No rigid vertex group of $\Delta$ has a nontrivial splitting in
    which all incident edge groups are elliptic.
  \item[AS5:] Every edge of $\Delta$ is incident to an abelian vertex
    group.
  \end{itemize}
  To represent an almost-strict homomorphism we use the notation
  $AS(\varphi,G,\Delta,H).$
\end{definition}

\par The reader should note that this list is missing the third and
fourth conditions from Theorem~\ref{thr:strictconditions}. The last
axiom is for technical reasons which appear in the proofs of
Theorem~\ref{ginfinityisalimitgroup},
Lemma~\ref{lem:admissablerefinement}, and in
Definition~\ref{def:jsjrespecting}.

\par We say that a map of a pair $(G,\E)$ to a group $H$ is
\term{indecomposable relative to the collection $\E$}, or
\term{$\E$--indecomposable}, if all images of elements of
$\E$ are nontrivial and it doesn't factor through a free
product in which all images of elements of $\E$ are elliptic.

We now define two types of \term{strict relative quotients} of groups.

\begin{definition}[Type I strict relative quotient]
  \par Let $(R,\E)$ be finitely generated group with a finite
  collection of abelian subgroups $\E,$ and suppose that
  $\varphi\colon R\to H$ is $\E$--indecomposable and that $H$
  is a limit group. Let $\mathscr{R}$ be the collection of quotients
  of $R$ such that for all $R'\in\mathscr{R},$ $\pi\colon R\onto R',$
  $\varphi$ factors through a $\Mod(R',\pi(\E))$--strict
  $\varphi'\colon R'\to H.$

  Define a partial order on $\mathscr{R}$ as follows. If $R_1$ and
  $R_2\in\mathscr{R}$ and the map $R_1\to R$ factors through $R_1\onto
  R_2$ then $R_1\geq R_2,$ with equality if and only if $R_1\to R_2$
  is an isomorphism. By Theorem~\ref{thr:strictlengthbound} $\mathscr{R}$
  has maximal elements. Choose one such maximal element and call it
  $\strict R.$ The quotient group $\strict R$ is a \term{type I strict
    relative quotient} of $R.$
\end{definition}

\par Let $\varphi\colon G\to H$ be an indecomposable homomorphism of
limit groups, $R$ a vertex group of a \gad\ $\Delta$ of $G.$ The
restrictions $\varphi\vert_{R}$ are $\E(R)$--indecomposable.

\begin{definition}[Type II strict relative quotient]
  Let $R$ and $H$ be limit groups and suppose
  $\varphi\colon(R,\E)\to H$ is
  $\E$--indecomposable Choose a group
  $\overline{R}=S*_{A_i}B_i$ realizing $c_a(\varphi,\E(R)).$
  Without loss, we may assume that for all $E\in\E(R),$ by
  introducing a new edge and valence one vertex, both with the group
  $E$ attached, that there is some $B_i$ into which the image of $E$
  is conjugate. Conjugate the collection $\E$ so that each
  element has image contained in some $B_i,$ rather than just being
  conjugate into one. This operation has no effect on $c_a.$ For each
  $E_i\in\E$ let $F_i$ be the kernel of the map $\img(E_i)\to
  H.$ Suppose that the image of $E_i$ is contained in $B_{j(i)}.$ Let
  $F_i'=F_i\cap A_{j(i)}$ and let $I(j)$ be the collection of indices
  $i$ such that $E_i<B_j.$ Now build the quotient groups
  \[\overline{R}''=S'*_{A'_j}B''_j,\quad S'\define S/\group{
  F'_i},\quad A'_j\define A_j/\group{ F'_i}_{i\in
    I(j)},\quad B''_j\define B_j/\group{ F_i}_{i\in I(j)}\] and
  \[\overline{R}'=S'*_{A_j'}B'_j,\quad B'_j\define B_j/\group{ F'_i}_{i\in I(j)}\] The map $\overline{R}\to H$ factors through the obvious maps 
  $\overline{R}\onto\overline{R}'\onto\overline{R}''\to H.$ Now if
  $c_a(S',\set{A_j'})>0$ then
  $c_a(\overline{R}',\img(\E))>c_a(\overline{R},\img(\E)).$ By
  construction, if $E\in\E,$ then the image of $E$ in
  $\overline{R}''$ embeds in $H.$

  Now pass to the type I strict relative quotient of $S'$ with respect
  to the collection $\mathcal{A}'_j.$ Rather than call the quotient
  $\strict S',$ as above, since $S'$ depends only on $R$ and the map
  to $H,$ call the quotient $\strict_{II}(R).$ Call the image of
  $A'_j$ in $\strict_{II}(R)$ $A''_j,$ and let
  $F''_j\define\ker(A'_j\to H).$ Now build the following quotient of
  $R$: \[(\strictr R=)\strictr (R,\E)\define
  \strict_{II}(R)*_{A''_j}(B''_j/F''_j)\] The quotient $\strictr R$ is
  a \term{type II strict relative quotient} of $R.$ By the same
  reasoning as above, $c_a(\strict S,\set{A''_j})=0.$ The groups $A''_j$
  embed in $H.$ The reader must be warned that in general $\strictr R$
  is \emph{not} a limit group. 
\end{definition}

\begin{remark}
  If $c_a(\varphi,\E(R))=0$ then $A_i$ is finite index in
  $B_i$ for all $i$ in the decomposition $\overline{R}=S*_{A_i}B_i.$
  If this is the case then $\strict(R,\E)$ and
  $\strictr(R,\E)$ agree. In light of this, in the event that
  $c_a=0,$ we freely ignore the distinction between type I and type II
  strict relative quotients whenever convenient.
\end{remark}

\par Let $(L,\E)$ be a limit group and let $\jsj(L,\E)$
be its relative abelian \jsj. We make the following normalizations on
$\jsj(L,\E)$:
\begin{itemize}
\item If $E,E'\in\E$ and $\cent_L(E)$ is conjugate to $\cent_L(E')$
  then $\left[E,E'\right]=\set{1}.$ 
  \item The centralizer of $E$ is always an abelian vertex group of
    $\jsj(L,\E).$ This can be accomplished by subdivision or
    introduction of valence one vertex groups, as necessary.
\end{itemize}

\par We are now ready to construct, given an indecomposable map $G\to
H,$ $\Delta$ a \gad\ of $G,$ a quotient $\phias G$ of $G,$ equipped
with a \gad\ $\phias\Delta$ induced by $\Delta,$ and a
$\phias\Delta$--almost strict $\phias G\to H,$ such that the
following diagram commutes.

\centerline{
  \xymatrix{
    & \phias G\ar[d] \\
G\ar[r]\ar@{->>}[ur] & H
}}

Let $G$ be a group with an abelian decomposition $\Delta$ and an
indecomposable map $\pi\colon G\to H,$ $H$ a limit group. Build a
group $\phias G$ with a splitting $\phias\Delta$ by taking the
following quotients of vertex groups of $G.$ For each edge group of
$\Delta,$ pass to the image in $H.$ For each rigid vertex group of
$\Delta$ pass to $\strictr(R,\E(R))$ if $R$ has nonabelian
image in $H,$ otherwise pass to $\strict(R,\E(R)).$  For all
other cases pass to $\strict(V,\E(V)).$ The gluing data from
$\Delta$ descends to gluing data on this new collection of groups.
Call the resulting group $\phias G.$ For each rigid vertex group $R_i$
such that $\varphi(R_i)$ is nonabelian, $\strictr
(R_i,\E(R_i))$ has the form $S_i*_{A_{i,j}}B_{i,j},$ and for
each element $E\in\E(R_i),$ the image of $E$ maps to some
$B_{i,j}.$ Replace $\strictr R_i$ in $\Delta$ by
$S_i*_{A_{i,j}}B_{i,j},$ and for each $i$ refine the splitting on
$S_i$ with the normalized relative \jsj\ decomposition
$\jsj(S_i,\set{A_{i,j}}).$ The resulting decomposition of $\phias G$ is
called $\phias\Delta.$ Note that every nonabelian vertex group
of $\phias\Delta$ is either a nonabelian vertex group from some
$\jsj(S_i,\set{A_{i,j}})$ or is a \qh\ vertex group inherited from
$\Delta.$ Call the induced homomorphism of $\phias G\to H$
$\phias\pi.$

\begin{definition}[Almost-strict factorization]
  \label{strictvertexmorphism}
  The group $\phias G$ with \gad\ $\phias\Delta$ and homomorphism
  $\phias\pi$ from above is an \term{almost-strict factorization of
    $\pi$}.
\end{definition}

\par The induced homomorphism $\phias\pi\colon \phias G\to H,$ may not
be strict, but it is reasonably close: Since the automorphisms in
$\Mod(\strict S,\E(S))$ fix the incident edge groups, they
extend to automorphisms of $\phias G.$  Hence, for every sequence
$f_n\colon H\to\free$ converging to $H,$ there exists a sequence of
automorphisms $\phi_n\in\Mod(\strict R,\E(\strict R)),$
fixing vertex groups up to conjugacy, such that
$f_n\circ\phias\pi\circ\phi_n$ is stably trivial on every vertex group
of $\phias\Delta.$ Thus if $\phias{\pi}$ isn't strict, the failure
must lie elsewhere in the bullets of Theorem~\ref{thr:strictconditions}.

\par The vertices of $\Delta$ came with labels Q, R, and A,
corresponding to quadratically hanging, rigid, and abelian vertex
groups. In the group $\phias G$ we label the vertices of
$\phias\Delta$ as follows.


\begin{itemize}
  \item Label abelian vertex groups `A'.
  \item Label \qh\ vertex groups `Q'.
  \item Label rigid vertex groups coming from the relative
    decompositions $\jsj(\strict S)$ `R'.
\end{itemize}

That $\phias{\pi}\colon\phias G\to H$ is almost-strict follows
immediately from the definitions. The \jsj\ decomposition of $G,$ in
the event that $G$ has one and $\Delta=\jsj(G),$ gives some information
about the \jsj\ decomposition of $H.$ 

\begin{definition}
  If $G$ is a group then $G^{ab}$ is the abelianization of $G$ modulo
  torsion.  If $E_i$ is a collection of subgroups of $G$ then the
  smallest subgroup of $G^{ab}$ which is closed under taking roots and
  contains the images of $E_i$ is the \term{peripheral} subgroup of
  $G^{ab}$ and is denoted by $P(G^{ab}).$ If $A$ is an abelian vertex
  group of a \gad\ then this definition of the peripheral subgroup
  agrees with the original definition.
\end{definition}

\begin{lemma}
  \label{lem:abelianranks}
  Let $\pi\colon G\to H$ be indecomposable, $G$ a limit group,
  $\Delta$ a \gad\ of $G,$ and $G\onto\phias G,$ a
  $\phias\Delta$--almost strict factorization of $\pi$ as above.

  If $V$ is a vertex of $\Delta$ which has abelian image then
  $\strict{V}\cong V^{ab}/\Ker(P(V^{ab})\to H).$ By indecomposability
  of $\pi,$ if $V$ is a \qh\ subgroup then $V$ is a punctured sphere
  or projective plane.

  If $V$ is a \qh\ vertex group which has nonabelian image in $H$
  then $\strict(V,\partial V)\cong (V,\partial V).$

  Suppose that $V$ has a graph of groups decomposition
  $V\cong\Gamma(A_i,F_j)$ over torsion free abelian vertex groups
  $A_i$ and \term{nontrivial} edge groups $F_j.$  Suppose further that
  each edge group of $\Delta$ incident to $V$ is elliptic in $\Gamma.$
  Let $P(A_i)$ be the subgroup of $A_i$ generated by incident edge
  groups (in $\Gamma$) and those elements of $\E(V)$
  conjugate into $A_i.$ Call the underlying graph of $\Gamma(A_i,F_j)$
  $\Gamma$ as well. Then
  \[\sum_i\rk(A_i/P(A_i))+\betti(\Gamma)\leq\rk(\strict V/\group{\img(\E(V))})\]
\end{lemma}

\begin{proof}
  If $V$ is \qh\ and has any genus then any simple closed curve
  cutting off a handle has trivial image. Thus $\pi$ factors through a
  free product $\overline{G}*\zee^2.$ Contradiction.

  The statement that \qh\ subgroups with nonabelian image are
  isomorphic to their type I strict relative quotients
  is~\cite[Lemma~5.13]{sela::dgog1}.

  Suppose $V$ is as in the last paragraph of the lemma. By
  indecomposability of $\pi$ and Lemma~\ref{edgeslive}, no $F_j$ has
  trivial image in $H.$ Construct an abelian quotient $\overline V$ of
  $V$ as follows: Let $T$ be a maximal tree in $\Gamma$ and let
  $F_{i_1},\dotsb,F_{i_m}$ be the edge groups of $\Gamma$ not carried
  by edges in $T.$ Let $V_T$ be the subgraph of groups of $V$ obtained
  by restriction to $T,$ and let $\E(V_T)$ be the collection
  $\E(V)$ along with the $F_{i_j}.$ The inequality
  \[\sum_i\rk(A_i/P(A_i))\leq\rk(\strict
  V_T/\group{\img(\E(V)\cup\set{F_{i_j}})})\] holds. Let
  $t_j$ be the stable letter associated to $F_{i_j}.$ Then the image
  of $t_j$ in $H$ conjugates the image of $F_{i_j}$ to another
  subgroup of the image of $(P\strict V_T).$ Since abelian subgroups
  of limit groups are malnormal, the two inclusions $F_{i_j}\to
  \strict V_T$ must agree with one another. Since limit groups are
  commutative transitive, the map $V_T*_{F_{i_j}}\to H$ factors
  through $V_T\oplus\group{ t_j},$ a group which satisfies the
  inequality of the lemma. Repeating over all edge groups $F_{i_j}$ we
  find an abelian quotient $V'$ of $V,$ a
  $\Mod(V',\img(\E(V)))$--strict $V'\to H$ satisfying the
  lemma. Since $V'\to H$ factors through $\strict V\to H,$ $V'\cong
  \strict V$ by maximality.
\end{proof}

\subsubsection{From almost-strict to strict}%

\par Fix $AS(\varphi,G,\Delta,H)$ for the remainder of this section.
We build an infinite sequence of groups $G\onto G_1\onto\dotsb
G_i\dotsb\onto H$ such that each induced map $G_i\onto H$ is almost
strict, closer to strict than the previous homomorphism, and the
direct limit homomorphism $\dirlim G_i\onto H$ is strict, and that for
all but finitely many $i,$ $G_i\onto G_{i+1}$ is an isomorphism.

\par We now define two maps which takes as input almost strict
$AS(\varphi,G,\Delta,H)$ and output almost strict
$AS(\Phi_*\varphi,\Phi_*G,\Phi_*\Delta,H)$ such that the induced
homomorphisms $\Phi_*\varphi\colon\Phi_*G\to H$ are closer to
satisfying the bullets from Theorem~\ref{thr:strictconditions}.

\smallskip
\noindent{\bf Defining $\boldsymbol{\phia}$}

\par First, we take an almost strict $AS(\varphi,G,\Delta,H)$ and
adjust the \gad\ $\Delta.$ Let $\sim_a$ be the equivalence relation
generated by adjacency of abelian vertex groups. Let $\left[A\right]$
be a $\sim_a$ equivalence class and let $\Gamma_{\left[A\right]}$ be
the subgraph of $\Delta$ with vertices from $\left[A\right]$ and edges
connecting members of $\left[A\right].$ Now perform a sequence of type
IVB folds to collapse the subgraphs $\Gamma_{\left[A\right]},$ as
$\left[A\right]$ varies over all $\sim_a$ equivalence classes. Call
the vertex associated to $\left[A\right]$ $v_{\left[A\right]}.$

%

\par Let $\phia G$ be the group obtained from $G$ by passing from
$G_{v_{\left[A\right]}}$ to $\strict G_{v_{\left[A\right]}}$ for each
$\sim_a$ equivalence class $\mathcal{A}.$ The abelian splitting $\phia
\Delta$ of $\phia G$ is the one it inherits from the collapsed
$\Delta.$  The induced map $\phia G\to H$ is denoted by $\phia\varphi.$

\par The induced map of the quotient group $\phia G$ with the \gad\ it
inherits from $G$ is almost-strict.

\smallskip
\noindent{\bf Defining $\boldsymbol{\phir}$}

\par We now define a third homomorphism $\phir,$ in addition to
$\phias$ and $\phia,$ which brings us closer to satisfying the third
bullet of Theorem~\ref{thr:strictconditions}.

\begin{definition}
Let $R_1$ be a limit group, $\E$ a family of nonconjugate abelian
subgroups of $R_1$ such that the relative \jsj\ decomposition
$\jsj(R_1;\E)$ is trivial. An \term{envelope} of $R_1$ is a limit
group $R_2$ such that
\begin{itemize}
\item There is a collection $\set{P_i}_{i\in I}$ of free abelian
  groups such that $R_2$ is a limit quotient of $R_1*_{E_i}P_i,$
\item There is a map from $R_1*_{E_i}P_i$ to a fixed limit group $H$
  called the \term{target}.
\item $R_1\to R_2$ is injective and the map $R_1*_{E_i}P_i\to H$
  factors through the map to $R_2.$
\item The map from $R_2$ to $H$ is $\Mod(R_2;\set{\img(P_i)})$--strict.
\end{itemize}
We call the group $R_1*_{E_i}P_i$ a \term{pre-envelope}. We call $R_1$
the \term{core} of the pre-envelope. The complexity of a pre-envelope
is the ordered triple
\[\left(\sum_i\rk(P_i/E_i),\vert I\vert,m\right)\] Where $m$ is the number of
indices $i$ such that $E_i$ is not maximal abelian in $R_1.$
\end{definition}

\par Our main lemma regarding envelopes is that they can be written
nicely as the output of the process of iteratively adjoining roots
(Definition~\ref{def:iteratedadjunction}) and extension of
centralizers.

\begin{definition}[Iteratively adjoining roots]
  \label{def:iteratedadjunction}
  Let $L$ be a limit group. Then $L'$ is \term{obtained from $L$ by
    iteratively adjoining roots} if there is a finite sequence of
  limit groups $L_i,$ $L_0=L,$ $L_n=L'$ such that
  \begin{itemize}
  \item $L_{i+1}$ is obtained from $L_i$ by adjoining roots to the
    collection $\E_i.$
  \item If $E\in\E_i$ then there
    is an $F\in\E_{i-1}$ such that $E=\cent_{L_i}(\img(F)).$
  \end{itemize}
\end{definition}

\begin{lemma}
  \label{lem:relativerigid}
  Let $R_2$ be an envelope of $R_1,$ with $\Mod(R_2;\set{\img(P_i)})$
  strict homomorphism $\varphi\colon R_2\to H$ which embeds $R_1.$
  Then $R_2$ can be decomposed as an iterated adjunction of roots and
  an extension of centralizers, and can be realized as a quotient
  pre-envelope of the pre-envelope $R_1*_{E_i}P_i$ of lower complexity.

  Let $\set{\mathcal{F}_j}$ be the set of equivalence classes of $E_i$
  in $R_2$ such that $E_i,E_{i'}\in\mathcal{F}_j$ if and only if $E_i$
  and $E_{i'}$ have conjugate centralizers. For each class
  $\mathcal{F}_j$ choose a single representative element $F_j,$ and
  for $j$ let $I_j$ be the set of indices $i$ such that
  $E_i\in\mathcal{F}_j.$ Then there are direct sum decompositions
  $P_i\cong P'_i\oplus C_i$ such that $E_i<C_i,$ quotients $D_j$ of
  $\oplus_{i\in I_j}P'_i,$ and $R_2$ can be written as
  \[
  \overline{R}_1\define R_1\left[\sqrt E_i\right],\quad
  R_2\cong\overline{R}_1*_{\cent_{\overline{R}_1}(F_j)}(\cent_{\overline{R}_1}(F_j)\oplus
  D_j)
  \]
  Moreover, the group $\overline{R}_1$ is a quotient of
  $R_1*_{E_i}C_i,$ has trivial \jsj\ relative to the centralizers of
  the images of the $E_i,$ and $\varphi$ is an embedding.
  \mnote{clumsy also}
\end{lemma}

First we need a basic lemma.

\begin{lemma}
  \label{lem:simpleenvelope}
  Let $H$ be a limit group, $L$ a pre-envelope $G*_{A_i}B_i,$ $A_i$
  maximal abelian in $G,$ closed under taking roots in $B_i,$ and
  $A_i$ not conjugate to $A_j$ for $i\neq j$ in $G.$ Let $\phi\colon
  L\to H$ be a homomorphism such that all restrictions
  $\phi\vert_{B_i}$ and $\phi\vert_G$ are injective. Then $\phi$ is
  injective.
\end{lemma}

The argument is standard, and follows from normal forms, induction on
the height of the analysis lattice, and the fact that if two elements
of a group are nonconjugate then they remain nonconjugate after
extension of centralizers.

\begin{proof}[Proof of Lemma~\ref{lem:relativerigid}]
  We define three homomorphisms of pre-envelopes. Each homomorphism
  yields a pre-envelope of lower complexity, or the pre-envelope
  embeds in the target.

  Suppose $E_i$ is not maximal abelian in $R_1.$ Let $C_i$ be the
  smallest direct summand of $P_i$ containing $E_i,$ and let $D_i$ be
  a complimentary direct summand. By commutative transitivity the map
  $R_1*_{E_i}P_i\to R_2$ factors through
  \[P_{j\neq
    i}*_{E_j}R_1*_{\cent_{R_1}(E_i)}((\cent_{R_1}(E_i)*_{E_i}C_i)^{ab}\oplus D_i)
  \] The last coordinate of the complexity decreases.

  If, say, $E_1$ and $E_2$ have conjugate centralizers, then the map
  $R_1*_{E_i}P_i\to H$ factors through
  \[R'_1\define P_{j>2}*_{E_{j>2}}R_1*_{\cent_{R_1}(E_1)}(\cent_{R_1}(E_1)\oplus
  D_1\oplus D_2)\] In this case the pre-envelope $R'_1$ has lower
  complexity than $R_1,$ as the second coordinate of the complexity
  strictly decreases.

  We now define the third quotient of a pre-envelope. This procedure
  is only to be applied if the previous two cannot.

  Let $C_i$ be the largest direct summand of $P_i$ such that the image
  of $E_i$ has finite index image in the image of $C_i,$ let $D_i$ be
  a complimentary direct summand, and let $\overline{C}_i$ be the
  image of $C_i$ in $R_2.$ Let $R'_1$ be the image of
  $R_1*_{E_i}*\overline{C}_i$ in $R_2.$ Since $E_i$ is finite index in
  $\overline{C}_i,$ the relative \jsj\ decomposition of $R'_1$ is also
  trivial and $R'_1$ must embed in $R_2$ and the target $H.$

  Then $R_2$ is a limit quotient of a new pre-envelope
  \[R''_1\define R'_1*_{\overline{C}_i}(\overline{C}_i\oplus D_i)\]
  Since $C_i$ contains $E_i,$ the complexity of $R''_1$ is at most
  that of $R'_1*_{E_i}P_i,$ and the core $R'_1$ is obtained from $R_1$
  by adjoining roots. If the complexity doesn't decrease, then each
  $P_i$ must embed in $R_2,$ each $E_i$ must be maximal abelian in
  $R_1,$ and by Lemma~\ref{lem:simpleenvelope}, $R_1*_{E_i}P_i$ embeds in
  the target $H,$ $R_2=R_1*_{E_i}P_i,$ and $R_2$ trivially has the
  structure from the lemma.
\end{proof}

\par Let $AS(\varphi,G,\Delta,H)$ be almost strict. The homomorphism
$\varphi$ may not be strict because it may not embed envelopes of
rigid vertex groups. We first adjust $\Delta.$ Let $e$ be an edge
incident to a rigid vertex group $R,$ $\tau(e)=r,$ and let $A$ be the
abelian vertex group attached to $\iota(e)$ and provided by axiom AS5.
Subdivide $e$ and pull $P(A)$ across the newly introduced edge
adjacent to $A.$ Repeat for all edges incident to rigid vertex groups,
and call the new abelian decomposition of $G$ $\Delta'.$ Let $A(E)$ be
the abelian vertex group adjacent to and centralizing $E.$ For each
rigid vertex group of $\Delta',$ the star of $R$ has the following
form
\[
\st(R)\cong R*_{E\in\E(R)}P(A(E))
\]
Now replace $\st(R)$ by its type~I strict relative quotient
$\strict(\st R,\set{P(A_i)}).$ By Lemma~\ref{lem:relativerigid}
\[
\strict(\st R,\set{P(A_i)})\cong\overline{R}*_{\cent_(F_j)}(C(F_j)\oplus D_j)
\] and each $P(A_i)$ has image contained in some $C(F_j)\oplus D_j.$
Since $P(A_i)$ embeds in $H$ it embeds in $C(F_j)\oplus D_j.$ Repeat
this process for all rigid vertex groups of $\Delta'$ and call the
quotient group $\phir G,$ the induced decomposition $\phir\Delta,$ and
the induced homomorphism $\phir\varphi.$  That the tuple
$(\phir\varphi,\phir G,\phir\Delta,H)$ is almost strict follows
immediately from the definitions.

\par We can now define the direct limit promised at the beginning of
this sub-subsection. Let $G\to H$ be a homomorphism of limit groups
which doesn't factor through a free product. Let $G_i$ be the
group \[ G_0=\phias G,\quad G_{2n+1}=\phia G_{2n},\quad G_{2n}=\phir
G_{2n-1}\] Similarly define $\Delta_{2n+1}$ and $\Delta_{2n+2}.$ Let
$G_{\infty}$ be the direct limit of the sequence $(G_n).$ The
homomorphism $G\to H$ factors through $G_{\infty}.$

\begin{theorem}
  \label{ginfinityisalimitgroup}
  The direct limit $G_{\infty}$ is a limit group, $G_n\onto G_{n+1}$
  is an isomorphism for all but finitely many $n,$ and the rigid
  vertex groups in the decomposition of $G_{\infty}$ induced by the
  \jsj\ decomposition of $G$ are obtained from the vertex groups of
  $\phias G$ by iteratively adjoining roots.
\end{theorem}

\begin{proof}[Proof of Theorem~\ref{ginfinityisalimitgroup}]
  By Lemma~\ref{lem:abelianranks} $c_a(G_n,\Delta_n)$
  (Definition~\ref{def:vertexscott}) is nondecreasing. Since
  $\betti(G_n)\leq\betti(G)$ we may drop finitely many terms from the
  beginning of the sequence and assume that the sequence
  $c_a(G_n,\Delta_n)$ is constant. If this is the case then the
  underlying graphs of $\Delta_i$ have constant betti number. Let
  $\set{R^i_j}$ be the rigid vertex groups of $\Delta_i.$  By
  Lemma~\ref{lem:relativerigid} $R^{i+1}_j$ is obtained from $R^i_j$ by
  iteratively adjoining roots. Let $v^i_j$ be the number of edges
  incident to $R^i_j.$ The sequence $(v^i_j)_{i=1..\infty}$ is
  nonincreasing. Again, by dropping finitely many terms from the
  beginning of the sequence, we may assume that all the sequences
  $v^i_j$ are constant. Once this is the case, the sequence of numbers
  of conjugacy classes of abelian vertex groups of $G_{2n+1}$ is
  constant.

  Let $A^{2n+1}_k$ be the collection of sequences of abelian vertex
  groups, $A^{2n+1}_k<G_{2n+1}$ such that $A^{2n+1}_k\to A^{2n+3}_k$
  for all $n.$ Inside $\Delta_{2n+1}$ there are subgraphs of groups of
  the form 


\centerline{%
  \xymatrix{%
    R^{2n+1}_{i(l)}\ar@{-}[rr]_{E^{2n+1}_l} & & A^{2n+1}_{k(l)}
}
}

\noindent In passing to $G_{2n+2},$ by subdividing and pulling,
subgraphs of this form are transformed to subgraphs of groups of the
form


\centerline{%
  \xymatrix{%
    R^{2n+1}_{i(l)}\ar@{-}[rr]_{E^{2n+1}_l} & & P(A^{2n+1}_{k(l)})\ar@{-}[rr]_{P(A^{2n+1}_{k(l)})} & & A^{2n+1}_{k(l)}
}}

\noindent  Which, after an application of Lemma~\ref{lem:relativerigid}, become the
  subgroups


\centerline{%
  \xymatrix{%
    R^{2n+2}_{i(l)}\ar@{-}[rr]_{\mathcal{Z}(E^{2n+1}_l)} &  & \mathcal{Z}(E^{2n+1}_l)\oplus D^{2n+2}_{l}\ar@{-}[rr]_{P(A^{2n+1}_{k(l)})} & &  A^{2n+1}_{k(l)}
}}

\noindent contained in $G_{2n+2}.$ The centralizer
$\mathcal{Z}(E^{2n+1}_l)$ should be taken in ${R^{2n+2}_{i(l)}}.$ Now
to pass to $G_{2n+3},$ the edges labeled $P(A^{2n+1}_k)$ are crushed
when collapsing the subtrees (they are trees since $c_a$ is constant)
$\Gamma_{\left[A\right]}\subset \Delta_{2n+1}.$ Since $E^{2n+1}_l\into
P(A^{2n+1}_{k(l)})$ and $P(A^{2n+1}_{k(l)})$ is generated by incident
edge groups, and since
$E^{2n+1}_l<{\mathcal{Z}_{{R^{2n+2}_{i(l)}}}(E^{2n+1}_l)}\cong
E^{2n+2}_l,$ $P(A^{2n+1}_{k(l)})<P(A^{2n+3}_{k(l)}).$ Since
$P(A^{2n+1}_{k(l)})$ embeds in $H$ by construction, the sequences
$P(A^{2n+1}_{k(l)})<P(A^{2n+3}_{k(l)})$ map to sequences of subgroups
of a finitely generated free abelian subgroup of $H.$ Thus, for
sufficiently large $n,$ $P(A^{2n+1}_k)=P(A^{2n+3}_k).$ Likewise, the
sequences $\mathcal{Z}_{R^{2n+2}_{i(l)}}(E^{2n+1}_l)$ are stable for
sufficiently large $n,$ hence $R^{2n+2}_i\to R^{2n+4}_i$ is an
isomorphism for all $i$ and sufficiently large $n.$

\par Let $n$ be large enough to satisfy the above. If $E$ is adjacent
to $R$ and $E$ doesn't have maximal abelian image in $R$ then
$\phia\circ\phir$ strictly increases the rank of some peripheral
subgroup, contradicting the stability of ranks of peripheral subgroups
(Recall the normalization that every edge group be adjacent to a
maximal abelian vertex group.). If the envelope of a rigid vertex
group doesn't embed, then the either the rank of an edge group must
increase under $\phir$ or a peripheral subgroup must fail to embed,
neither of which is possible.  By Lemma~\ref{lem:simpleenvelope} the
envelopes must embed. By Theorem~\ref{thr:strictconditions}, $G_n\to H$ is
strict for sufficiently large $n.$

\par That the rigid vertex groups are obtained by iteratively
adjoining roots follows immediately from the construction of $\phir$
and Lemma~\ref{lem:relativerigid}. That $G_n\to G_{n+1}$ is an isomorphism
for sufficiently large $n$ follows from the fact that stability of
edge and peripheral subgroups implies that the quotients $\phir$ and
$\phia$ are isomorphisms.
\end{proof}

\par Set $\Phi_{s}G\define G_{\infty},$ likewise for $\Phi_s\Delta,$
$\Phi_s\varphi.$ We say that $\Phi_s\Delta$ is the \term{push-forward}
of $\Delta.$ 


\section{Degenerations of \jsj\ decompositions}
\label{sec:degenerations}

\par We saw in the previous section that given an indecomposable
homomorphism $\pi\colon G\to H,$ one can construct a quotient
$\Phi_s(G)$ of $G$ and a strict homomorphism $\Phi_s(G)\to H$ such
that the composition of the quotient map and the strict homomorphism
is $\pi.$ Our approach to Krull dimension for limit groups is to use
Theorems~\ref{thr:scott::freelydecomposable} and~\ref{thr:nostrict} to
reduce the problem of existence of arbitrarily long chains of
epimorphisms of limit groups to an analysis of degenerate chains. Let
$\LL$ be a chain of freely indecomposable limit groups such that all
maps $\LL(i)\onto\LL(j),$ $i<j,$ are indecomposable.  Such a chain is
\emph{indecomposable}. The next step in our analysis is to control the
degenerations of the sequence of \jsj\ decompositions $\jsj(\LL(i)).$

\par The strict homomorphism $\Phi_s(G)\onto H$ is a (finite) direct
limit of quotients of $G$ obtained by iterating the constructions
$\Phi_a$ and $\Phi_r.$  These homomorphisms were designed to ensure
that in the direct limit the bullets from
Theorem~\ref{thr:strictconditions} were satisfied.  The possible
degenerations of the abelian \jsj\ of $G$ under $\Phi_s$ were not
completely characterized since we were only interested in constructing
a strict factorization.

\par The chains of freely indecomposable free factors produced by
Theorem~\ref{thr:scott::freelydecomposable} are necessarily
degenerate.

\par In this section we construct a complexity, computed from the
\jsj\ decomposition and modeled on the Scott complexity, which is
nondecreasing on degenerate indecomposable chains and which takes
boundedly many values for limit groups with a given first betti
number.  It will be constructed out of three kinds of data from the
\jsj\ decomposition: A complexity which manages $\betti,$ the
complexity of \qh\ vertex groups not reflected in the first betti
number, and some combinatorics of the underlying graphs of
\jsj\ decompositions.

\par We start by measuring as much of the first betti number of a
limit group which is immediately detectable in its \jsj\
decomposition.  Define the following quantities:
\begin{itemize}
\item The sum of relative ranks of abelian vertex groups: \[
  c_a(G)\define\sum_A\betti(A/P(A)) \]
\item The complexity of surface vertex groups: If $\Sigma$ is a
  surface with boundary then
  \[c_g(G)\define\sum_{\Sigma(Q)}\betti(\Sigma,\partial\Sigma)\] with
  the sum over all surfaces $\Sigma(Q)$ representing \qh\ vertex
  groups $Q$ of $G.$ If the computation is done in a \gad\ $\Delta$ we
  write this as $c_g(G,\Delta),$ otherwise the abelian \jsj\ is
  implied.
\item Let $\Gamma$ be the underlying graph of the \jsj\ decomposition.
  The betti number of the underlying graph of the abelian \jsj\
  decomposition: \[c_b(G)\define\betti(\Gamma)\]
\end{itemize}

Our first approximation to the Scott complexity of a limit group is
the quantity \[\scott_1(G)\define(c_a(G),c_g(G),c_b(G))\] We have
already shown, in Lemma~\ref{lem:abelianranks}, that the first
coordinate of $\scott_1$ is nondecreasing under indecomposable
degenerate maps of limit groups.  Lemma~\ref{lem:bettitogenus} is the
analog of Lemma~\ref{lem:abelianranks} for surfaces with boundary.
First, a simple lemma.

\begin{lemma}
  \label{lem:surfacegenus}
%
  Let $\Sigma$ be a surface with boundary, written as a graph of
  surfaces with boundary $\Gamma(\Sigma_i).$ Then
  \[\betti(\Sigma,\partial\Sigma)\geq\betti(\Gamma)+\sum_i\betti(\Sigma_i,\partial\Sigma_i)\]
\end{lemma}

\begin{proof}
  There is a map $\Sigma\onto\zee^{\betti(\Gamma)}$ which kills
  $\partial\Sigma$ and all $\Sigma_i.$
\end{proof}


\begin{lemma}
  \label{lem:bettitogenus}
  Suppose $c_a(G)=c_a(H).$  Then $c_g(G)\leq c_g(H).$  If equality
  holds then $c_b(G)\leq c_b(H).$
\end{lemma}

\begin{proof}
  Let $\Delta=\Phi_s(\jsj(G))$ be the push forward of the abelian
  \jsj\ of $G$ under $\Phi_s.$  Since $G\onto H$ is degenerate, the
  map $\Phi_s(G)\to H$ is an isomorphism.  Thus, we only need to
  deduce how the \jsj\ of $H$ is obtained from $\Delta.$

  The main difficulty is that there may be splittings of $H$ which are
  hyperbolic-hyperbolic with respect to one-edged splittings
  corresponding to edge groups incident to rigid vertex groups of
  $\Delta.$  We claim the following: If $R$ is a rigid vertex group of
  $\Delta$ then either $R$ is a rigid vertex group in the \jsj\ of $H$
  or $R$ is represented by a non-\qh\ subsurface of a \qh\ vertex
  group of $H.$ Such a subsurface group must be either a twice
  punctured projective plane or a thrice punctured sphere.

  It is a well known fact from the \jsj\ theory that if $\Delta$ is a
  \gad\ of $H$ and $H$ has an abelian \jsj\ decomposition, then
  $\Delta$ can be obtained (essentially) from the \jsj\ by choosing a
  family of simple closed curves $c_i$ on \qh\ vertex groups and
  regarding the resulting subsurfaces as vertex groups and the curves
  $c_i$ as edge groups.  After slicing the surface vertex groups like
  this, a sequence of folds produces $\Delta.$

  Thus every rigid vertex group $R$ of $\Delta$ is generated by rigid,
  abelian, subsurface groups and stable letters of the \jsj\ of $H.$
  If it consists of more than two, then $R$ has a splitting relative
  to its incident edge groups, contrary to rigidity of $R.$  Thus $R$
  is either a subsurface or a rigid vertex group from the \jsj\ of
  $H.$  If $R$ is a subsurface group then, since $R$ isn't a \qh\
  vertex group of the relative abelian \jsj\ of $R,$ it must be one of
  the two non-\qh\ surface groups above.  We recover the abelian \jsj\
  of $H$ by gluing the \qh\ vertex groups of $\Delta$ and the rigid
  subsurface groups along their boundary components.

  Define, in analogy with $\sim_a,$ an equivalence relation $\sim_q$
  on subsurface groups $R,$ \qh\ vertex groups $Q,$ and cyclic abelian
  vertex groups $Z$: If $V$ is a cyclic abelian vertex group of
  valence two, regard it as the fundamental group of an annulus, and
  of valence one, of a \mobius\ band.  If two such groups $V_1$ and
  $V_2$ are adjacent, being the endpoints of an edge $e,$ then
  $V_1\sim_qV_2$ if the inclusions $E\into V_i$ are isomorphisms with
  boundary components.  As for $\sim_a$ equivalence classes
  $\left[A\right],$ let $\Gamma_{\mathcal{Q}}$ be the subgraph spanned
  by edges connecting vertices of a $\sim_q$ equivalence class
  $\mathcal{Q}.$

  If a cycle $C$ of such vertex groups appears in $\Gamma_{Q},$ then
  by Lemma~\ref{lem:surfacegenus}
  \[
  c_g(H,\Delta/C)+c_b(H,\Delta/C)\geq c_g(H,\Delta)+c_b(H,\Delta)
  \]
  Collapsing all such cycles through a sequence of collapses and type
  IV folds, we see that $c_g(G)+c_b(G)\leq
  c_g(H,\Delta)+c_b(H,\Delta)\leq c_g(H)+c_b(H).$  If equality holds
  then no cycles are collapsed, and if a cycle is collapsed then the
  inequality is strict.
\end{proof}

With this the proposition that $\scott_1$ is nondecreasing under
degenerate maps of freely indecomposable limit groups is established.
We now turn to the problem of adding terms to the complexity
$\scott_1$ which handle the possibility that nonabelian vertex groups
may have abelian image under degenerate maps.  For the remainder of
the section, $G\onto H$ is a degenerate map such that
$\scott_1(G)=\scott_1(H).$  Note that if $R$ is a valence one vertex
group then, since the first betti number relative to an abelian
subgroup is greater than $0,$ $R$ cannot have abelian image in $H,$
otherwise $c_a(H)>c_a(G).$  Likewise, the underlying graph of the
\jsj\ of any $\strict R,$ for a rigid vertex group $R$ of $G,$ cannot
contain a loop (which makes a contribution to $c_b$) or, through
abelian vertex groups in their relative \jsj's, make any contributions
to $c_a,$ in other words, they must be trees with abelian subgroups
equal to their peripheral subgroups (modulo the images of incident
edge groups).  Let $\Gamma_{\mathcal{A}}$ be a subgraph of the
underlying graph of the \jsj\ of $G$ such that $\mathcal{A}$ has
abelian image in $H.$  If $c_a(G)=c_a(H)$ then $\Gamma_{\mathcal{A}}$
must be a tree.  If $\mathcal{A}*_{E_i}G'$ is the splitting of $G$
obtained by collapsing all edges of $G$ except those adjacent to
$\Gamma_{\mathcal{A}}$ but not contained in $\Gamma_{\mathcal{A}},$
then no two images $\img(E_i),\img(E_j)$ can have conjugate
centralizers in $\img(G'),$ otherwise $c_a(H)>c_a(G).$

\begin{lemma}
  Suppose $G\onto H$ is degenerate and $\scott_1(G)=\scott_1(H).$
  Suppose $R$ is a nonabelian vertex group of $G$ (either rigid or
  \qh) with abelian image in $H.$  Let $e_1,\dotsc,e_n$ be the edges
  of $G$ such that $e_i$ doesn't connect $r,$ the vertex carrying $R,$
  to a valence one abelian vertex group of $G.$  Then the collection
  of images of $E_i,$ $\group{E_i}<R,$ generate
  $\mathrm{H}_1(R),$ and if an incident edge is separating, then
  neither vertex group in the corresponding one-edged splitting has
  abelian image in $H.$ In particular, if $R$ is \qh, then at most one
  edge incident to $R$ connects it to a valence one abelian vertex
  group.  In either case, the abelian vertex group of $H$ containing
  the image of $R$ has valence at least two.
\end{lemma}

\begin{proof}
  The statement about incident edges generating $R$ homologically
  follows from the observation that if not, then there is a map
  $R\onto\mathbb{Z}$ which kills every incident edge.  Let $A_1$ and
  $A_2$ be valence one abelian vertex groups adjacent to $Q,$ a
  \qh\ vertex group of $G.$  Let $P$ be a pair of pants with boundary
  components $l_1,$ $l_2$ and $w,$ such that $l_i$ is the leg of $P$
  attached to $A_i.$  Then the group
  $\mathcal{A}=A_1*_{l_1}P*_{l_2}A_2$ satisfies
  $\betti(\strict\mathcal{A},w)>\betti(A_1,l_1)+\betti(A_2,l_2),$
  leading to a strict increase in $c_a.$

  Since the first betti number of a nonabelian limit group relative to
  an abelian subgroup is at least one, $n$ is at least two.  Suppose
  that some edge $e_i$ is separating, $G=G_1*_{E_i}G_2,$ and $R$ (or
  $Q$) is contained in $G_1.$  If $G_2$ has abelian image in $H$ then
  the relative \jsj\ decomposition $\jsj(G_2,E_i)$ is a tree.  If
  $G_2$ contains a rigid vertex group $R'$ of $G$ then $R'$ isn't a
  valence one vertex group, is a cut-point in the underlying graph of
  $\jsj(G),$ and has as many complimentary components as it does
  incident edge groups.  Likewise, if $Q'$ is a \qh\ subgroup of $G$
  contained in $G_2,$ then the vertex associated to $Q'$ must be a
  cut-point in the underlying graph of the \jsj\ and in fact must have
  as many complimentary components as it does boundary components.

  Thus, if $G_2$ contains a nonabelian vertex group of $G,$ then it
  must contain one for which all incident edges but one are attached
  to valence one abelian vertex groups of $G.$  This is of course an
  impossibility, as $c_a(H)$ is then strictly larger than $c_a(G).$
\end{proof}

If $Q$ is a \qh\ vertex group of $G$ which has abelian image in $H$
then, since $G\to H$ doesn't factor through a free product, $Q$ can
have no genus and is either a punctured sphere or projective plane.
Since $Q$ isn't a punctured torus, $\chi(Q)\leq -2$ and $Q$ has at
least three punctures.  Let $\Delta$ be the splitting obtained by
collapsing all edges of $\jsj(G)$ not adjacent to $Q.$  If $c$ and $d$
are boundary components of $Q$ which are attached to the same vertex
group in $\Delta,$ then their images have nonconjugate centralizers,
otherwise $c_a$ strictly increases.

The vertex group in $\Phi_a\Phi_{as}G$ which is the image of $Q$ is
abelian, and, since $c_a(G)=c_a(H),$ the remark about one edged
splittings and the fact that no $\Gamma_{\mathcal{A}}$ can contain a
loop the vertex group containing $\img(Q)$ has valence at least three
if $Q$ is not adjacent to a valence one abelian vertex group, has
valence at least two if $Q$ is adjacent to a single valence one
abelian vertex group, and if $Q$ is adjacent to two valence one
abelian vertex groups $A_1$ and $A_2,$ $c_a(H)>c_a(G).$  The last part
can be seen by finding a pair of pants $P$ in $Q$ with one cuff
attached to each abelian vertex group.  \mnote{the ``cuff'' thing is
  redundant. in the previous proposition, i think} Let $w$ be the
waist of the pair of pants.  Passing to the abelianization of
$\mathcal{A}=A_1*_{\zee}P*_{\zee}A_2,$ an elementary computation shows
that the relative betti number satisfies
$\betti(\strict\mathcal{A},w)>\betti(A_1,l_1)+\betti(A_2,l_2),$
leading to an increase of $c_a.$

\par The next step in our analysis is to show that a certain quantity
can be appended to $\scott_1$ such that under degenerate maps which
don't raise $\scott_1,$ the quantity appended is nondecreasing, and if
equality holds in the new coordinates, then only valence two
nonabelian rigid vertex groups can have abelian image (and ruling out
the possibility that \qh\ subgroups have abelian image).  There is
also the possibility that \qh\ subgroups can grow, in the sense that
their Euler characteristics can decrease, and that rigid vertex groups
can engender multiple child rigid vertex groups.  We also control
these phenomena.

\par For the remainder of the section, $G\onto H$ is degenerate and
$\scott_1(G)=\scott_1(H).$

\begin{definition}
  If $\Gamma$ is a finite graph, $v$ a vertex of $\Gamma,$ then
  $\kappa(v)$ is \[1-\frac{1}{2}\mathrm{valence}(v)\] by definition
  $\chi(\Gamma)=\sum_{v\in\Gamma^{0}}\kappa(v).$  Let $\Gamma$ be the
  underlying graph of an abelian decomposition $\Delta$ of a limit
  group $G.$  Then define
  \[\kappa_N(\Delta)\define\sum_{G_v\mbox{ nonabelian}}\kappa(v)\] and
  \[\kappa_A(\Delta)\define\sum_{G_v\mbox{ abelian}}\kappa(v)\]
  By definition, $\kappa_A(\Delta)+\kappa_N(\Delta)=\chi(\Gamma).$
\end{definition}

Suppose $v$ is a valence one vertex of $\Delta.$  Then, regardless of
whether or not $G_v$ is abelian, $G_v$ contributes at least $1$ to
$\betti(G).$  Let $\kappa_{N|A}^+$ be the total contribution of
nonabelian$\vert$abelian valence one vertex groups to $\kappa_{N|A},$
and let $\kappa_{N|A}^-$ be the total contribution of
nonabelian$\vert$abelian vertex groups with valence at least three.
All other vertex groups have valence two and make no contribution to
either $\kappa_N$ or $\kappa_A.$  We now show that $\kappa_N(\Delta)$
takes only boundedly many values for abelian decompositions of limit
groups with a given first betti number.

\begin{lemma}[Bounding $\kappa_N(\Delta)$]
  \label{lem:boundingkappan}
  Let $G$ be a limit group.  Then \[ \frac{1}{2}\betti(G)\geq
  \kappa_N(\Delta)\geq 1-\frac{3}{2}\betti(G)\] for all abelian
  decompositions $\Delta$ of $G.$
\end{lemma}

\begin{proof}
  Since each valence one vertex group of $\Delta$ contributes at least
  $1$ to $\betti(G),$ we have that $\kappa_{N|A}\leq
  \frac{1}{2}\betti(G)$ since $\kappa_{N|A}\leq\kappa_{N|A}^+.$  Since
  $\kappa_N+\kappa_A=1-\betti(\Delta)\geq 1-\betti(G)$ we have
  that \[\kappa_N+\frac{1}{2}\betti(G)\geq 1-\betti(G)\] Thus
  $\kappa_N\geq 1-\frac{3}{2}\betti(G).$  The same also holds for
  $\kappa_A.$
\end{proof}

\begin{definition}[Very weakly \jsj\ respecting]
  Suppose $G\onto H$ is degenerate and $\scott_1(G)=\scott_1(H).$
  Suppose that for all rigid vertex groups $R$ of $G,$ $R$ has
  nonabelian image in $H$ unless it has valence two, the relative
  \jsj\ decompositions $\Delta_{\strictr R}$ are trees, edge groups
  incident to $R$ have images in valence one abelian vertex groups of
  $\Delta_{\strictr R},$ and no two incident edge groups have
  conjugate centralizer in $\strictr R.$  Moreover, in the process of
  taking the direct limit $\dirlim G_i\onto H,$ at no point do edges
  incident to a rigid vertex group of $G_i$ have conjugate
  centralizers, after application of Lemma~\ref{lem:relativerigid}, in
  the associated rigid vertex group of $G_{i+1}.$  Such a homomorphism
  is \emph{very weakly \jsj\ respecting}.
\end{definition}

\begin{lemma}
  \label{lem:kappasequal}
  Suppose $\varphi\colon G\onto H$ is degenerate and
  $\scott_1(G)=\scott_1(H).$  Then $\kappa_N(G)\leq\kappa_N(H),$ with
  equality only if $\varphi$ is very weakly \jsj\ respecting.
\end{lemma}

\begin{proof}
  Let $R$ be a rigid vertex.  By the assumption on $\scott_1,$ if $R$
  has valence one then $R$ doesn't have abelian image.  Thus we only
  need to consider rigid vertices with valence at least three: those
  with valence two are free to have abelian image without disturbing
  $\kappa_N.$  Let $R$ have valence at least three.  If $R$ has
  abelian image in $H$ then $\kappa_N$ must increase by at least
  $\frac{1}{2}.$  We only need analyze rigid vertex groups with
  nonabelian image.  Let $\Delta_R$ be the relative \jsj\
  decomposition of $\strict R.$

  By our standing assumption that \jsj\ decompositions be bipartite,
  with one class the abelian vertex groups, the other the nonabelian
  vertex groups, and if two edge groups are incident to a rigid vertex
  group $R,$ then they have nonconjugate centralizer, we may assume
  that $\Delta_R$ has a form such that all images of edge groups
  incident to $R$ map to abelian vertex groups of $\Delta_R$ (If an
  incident edge group has image in a rigid vertex group and its
  centralizer isn't conjugate to the centralizer of an incident edge
  group in $\Delta_R,$ simply introduce a new edge and pull the
  centralizer).  Moreover, in $\Phi_{as}(G),$ since $c_a(G)=c_a(H),$
  all abelian vertex groups of $\Phi_{as}(G)$ contributed by $\strict
  R$ are equal to their peripheral subgroups.  In particular, there
  are no valence one abelian vertex groups which don't contain the
  image of some edge group incident to $R.$

  Let $v_1,\dotsc,v_n$ be the vertices of $\Delta_R$ corresponding to
  abelian vertex groups, and let $T$ be the subtree of $\Delta_R$
  spanned by the collection $\set{v_i},$ and let $T_j$ be the
  collection of closures of complimentary components of $T$ in
  $\Delta_F.$  Note that if $T_j\cap T_{j'}\neq\emptyset$ then the
  intersection is a single point and coincides with the intersections
  $T\cap T_j$ and $T\cap T_{j'}.$

  We now compute the contribution of $T$ to $\kappa_N.$  First, we
  claim that if two images of edge groups incident to $R$ have
  conjugate centralizing elements in $\strictr R,$ or if some incident
  edge group maps to a non valence one vertex of $T$ then
  \[\sum_{v\in T\mbox{, $\strictr R_v$
      nonabelian}}\kappa(v)>\kappa_N(R)\] Equality holds if we
  pretend that every vertex group from $T$ is nonabelian and that the
  centralizers of images of incident edge groups are nonconjugate.  If
  any vertex of valence greater than two is abelian, the inequality
  above must hold.  If all vertices of valence at least three are
  nonabelian, then no two images of incident edge groups can have
  conjugate centralizers, since these correspond to new abelian vertex
  groups (of $\Phi_{as}\Delta$) with valence at least three.  If an
  incident edge group has image in a non-valence one abelian vertex
  group of $T$ then the corresponding vertex of $\Phi_{as}$ has
  valence at least three, and the corresponding contribution to the
  computation of $\kappa_N$ is strictly positive.

  Now we compute the contribution of the trees $T_j$ to $\kappa_N.$
  As above, if we pretend that each tree consists of only nonabelian
  vertex groups, then the contribution each tree makes to $\kappa_N$
  over that of $R$ is at least $\frac{1}{2}.$  If any vertex groups
  with valence at least three are abelian, then the contribution of
  $T_j$ to $\kappa_N$ can only increase since no valence one vertex
  groups not centralizing images of edge groups incident to $R$ are
  abelian, by the hypothesis that $\scott_1(G)=\scott_1(H).$
  
  To verify that $\kappa_N(G)\leq\kappa_N(H)$ we only need to check
  that iterated application of $\Phi_a\circ\Phi_r$ to $\Phi_{as}(G)$
  cannot decrease $\kappa_N.$  This follows immediately from the fact
  that no rigid vertex group of $\Phi_{as}(G)$ splits after an
  application of $\Phi_r$ (lemma~\ref{lem:relativerigid}) and the
  observation that folding edges of nonabelian vertex groups with
  conjugate centralizers can only increase $\kappa_N.$  If equality of
  $\kappa_N$ holds then no such folding can ever occur.

  The last step is to observe that when crushing trees of \qh\ and
  subsurface groups of the abelian decomposition $H$ inherits from
  $\Phi_{as}(G)$ to build the \qh\ subgroups of $H,$ as in the proof
  of Lemma~\ref{lem:bettitogenus}, $\kappa_N$ doesn't change.
\end{proof}

\par The final step in this coarse analysis of the degenerations of
\jsj\ decompositions is the following observation.  The quantity
\[c_q(G)\define \sum_{Q\in\jsj(G)}\vert\chi(Q)\vert\] is the total
complexity of \qh\ subgroups.

\begin{lemma}
  If $\scott_1(G)=\scott_1(H)$ and $\kappa_N(G)=\kappa_N(H)$
  then \[c_q(G)\leq c_q(H)\] If equality holds then no vertex groups
  of $\strict R$ correspond to sub-surface groups of \qh\ vertex
  groups of $H.$
\end{lemma}

\begin{proof}
  Since $\kappa_N(G)=\kappa_N(H),$ no nonabelian vertex group with
  valence one or at least three has abelian image.  If $\strict R$
  contributes a subsurface group to a \qh\ subgroup of $H$ then the
  total Euler characteristic must decrease since the subsurface group
  must be either a multiply punctured sphere or projective plane with
  Euler characteristic at most $-1.$
\end{proof}

\begin{definition}[Essential, Vulnerable]
  \label{def:essential-vulnerable}
  We say that a vertex of the abelian \jsj\ decomposition of a limit
  group is \emph{essential} if it satisfies at least one of the
  following:
  \begin{itemize}
  \item It is \qh.
  \item It isn't valence two.
  \item It is abelian and isn't equal to its peripheral subgroup.
  \end{itemize}
  The number of essential vertices of $G$ is denoted by
  \[v_e(G)\] Let $\beta(G)$ be the set of essential vertices in
  $\jsj(G),$ and let $p(G)$ be the set of unoriented reduced edge
  paths $\alpha\colon\left[0,1\right]\to\jsj(G)$ such that
  $\alpha^{-1}(\beta(G))=\set{0,1}.$  The elements of $p(G)$ fall into
  six classes, corresponding to the types of endpoints: for example,
  $AA(G)$ is the set of elements of $p(G)$ whose endpoints are both
  essential abelian vertices.  A path in $AA(G),$ $AQ(G),$ or $QQ(G),$
  is \emph{vulnerable} if it crosses a rigid vertex group.  Let
  $c_v(G)$ be the number of vulnerable paths in $AQ(G)\sqcup QQ(G).$

  Let $c_{\#r}(G)$ be the number of essential rigid vertex groups of
  $G.$
\end{definition}

\par We can now define the Scott complexity.

\begin{definition}[Scott complexity]
  The \emph{Scott complexity} of a freely indecomposable limit group
  is the following lexicographically ordered quantity:
  \[
  \scott(G)\define
  (c_a(G),c_g(G),c_b(G),\kappa_N(G),c_q(G),c_{\#r}(G),-v_e(G),-c_v(G))
  \]
\end{definition}

\par Note that $\scott$ takes only boundedly many values for a given
first betti number.  That the last four take only boundedly many
values follows from the part of the proof of
Lemma~\ref{lem:rankheightbound} bounding the complexity of the
\jsj\ decomposition in terms of $\betti$ ($c_q$ and $c_{\#r}$) and
Lemma~\ref{lem:boundingkappan} ($\kappa_N$).

\par Fix a degenerate $\varphi\colon G\onto H$ such that $\scott_1$
and $\kappa_N$ are constant.  By Lemma~\ref{lem:kappasequal},
$\varphi$ is very weakly \jsj\ respecting.  If, additionally,
$c_q(G)=c_q(H),$ then for all rigid vertex groups $R$ of $G,$
$\strictr R= \strict S *_{F_j}B_j,$ no vertex group of $\jsj(\strict
S,\E(\strict S))$ is a subsurface group of a \qh\ subgroup of
$H,$ and we conclude by Lemma~\ref{lem:relativerigid} that all rigid
vertex groups of $H$ are obtained by iteratively adding roots to rigid
vertex groups from the collection $\set{\strict_{II} R_i}:{R_i\mbox{ rigid
in }G}.$  If the relative \jsj\ of some $\strict_{II} R$ has more than
one non-valence two vertex then $c_{\#r}(G)<c_{\#r}(H).$  Thus, if
$\scott(G)=\scott(H)$ then, for all rigid vertex groups $R_i,$ the
underlying graph of $\Delta_{\strict_{II} R_i}$ is a tree, has only
one vertex $v$ of valence at least three, the valence of $v$ is
precisely the valence of $R,$ and every valence one vertex contains
the image of exactly one incident edge group.

\par Let
\[\scott_2(G)=(c_a(G),c_g(G),c_b(G),\kappa_N(G),c_q(G),c_{\#r}(G))\]
and suppose that $\scott_2(G)=\scott_2(H)$ and $-v_e(G)<-v_e(H).$
There must be a path $p$ in $AA(G)$ or $QQ(G)$ such that every rigid
vertex group crossed has abelian image in $H.$ (No path in $AQ(G)$
contains an essential vertex, so we need not consider this case.)
Abelianizing these groups leaves $\kappa_N$ and $c_q$ unchanged.  If
$p\in AA(G)$ then the number of essential abelian vertices of $H$ is
less than that of $G.$ Now we consider the quantity $c_v,$ under the
assumption that $-v_e$ is constant.  The number of vulnerable paths
cannot increase, as the topology and labeling of essential vertices of
the underlying graphs of \jsj's is unchanged under $G\onto H.$

\begin{lemma}
  \label{lem:scottnovulnerable}
  If $\scott(G)=\scott(H)$ then at least one rigid vertex group of $G$
  crossed by a vulnerable path in $AQ(G)\sqcup QQ(G)\sqcup AA(G)$ has
  nonabelian image in $H.$
\end{lemma}

The lemma follows immediately from the definitions.

\begin{definition}[Weakly \jsj\ respecting]
  If $\varphi\colon G\onto H$ is degenerate and $\scott(G)=\scott(H),$
  then $\varphi$ is \emph{weakly \jsj\ respecting}.
\end{definition}

\par In conclusion, we have the following.  Let $C_n$ be the number of
values $\scott$ takes for limit groups with first betti number $n.$

\begin{lemma}
  Let $\LL$ be a degenerate chain of limit groups.  If
  $\Vert\LL\Vert>K\cdot C_n$ and $\betti(\LL)=n,$ then there is a weakly
  \jsj\ respecting subchain of $\LL$ of length $K.$
\end{lemma}

We can now define the \gad's which are respected under maps of
constant Scott complexity.

\begin{definition}[$\Delta$--admissible; $\Delta$--stable;
  almost-\jsj\ respecting]
  \label{def:jsjrespecting}
  A \emph{virtual \jsj\ decomposition} of a limit group $L$ is an
  abelian decomposition $\Delta$ such that the Scott complexity of $L$
  measured with respect to $\Delta$ is the same as the Scott
  complexity of $L$ measured with respect to $\jsj(L).$  \mnote{what
    about normalizations}

  Let $\varphi\colon G\onto H$ be degenerate.  If $\Delta$ is a
  virtual \jsj\ decomposition of $G$ then $\varphi$ is
  \emph{$\Delta$--admissible} if $\scott(G)=\scott(H)$ and every
  nonabelian rigid vertex group of $\Delta$ has nonabelian image in
  $H.$  If $\varphi$ is $\Delta$--admissible then there is a well
  defined push forward of $\Delta$ to $H$ which is also a virtual
  \jsj\ decomposition.

  Let $\LL$ be a weakly \jsj\ respecting chain, and let $\Delta$ be a
  virtual \jsj\ decomposition of $\LL(1).$ We say that $\LL$ is
  \emph{$\Delta$--admissible} if and all $\varphi_{1,j}$ are
  $\Delta$--admissible.

  Suppose $\varphi\colon G\onto H$ is $\Delta$--admissible.  Then
  $\varphi$ is \emph{$\Delta$--stable} if, for every rigid vertex
  group $R$ of $\Delta,$ $c_a(\varphi\vert_{R},\E(R))=0$ (See
  Definition~\ref{def:vertexscott}).

  Let $\LL$ be a weakly \jsj\ respecting chain such that
  $\varphi_{i,j}$ is $\jsj(\LL(i))$--stable for all $i$ and $j.$ If
  $\LL(i)$ and $\LL(j)$ have the same number of rigid vertex groups for
  all $i$ and $j,$ and the betti numbers of rigid vertex groups of
  \jsj\ decompositions are constant then $\LL$ is \emph{almost-\jsj\
    respecting}.

\end{definition}



The following lemma follows easily from the definitions and
conventions on \jsj\ decompositions.

\begin{lemma}
  \label{lem:formofvirtual}
  A virtual \jsj\ decomposition is obtained by collapsing a forest $F$
  in $\jsj(G)$ of the following form, and all such forest collapses
  yield virtual \jsj\ decompositions:
  \begin{itemize}
  \item All valence one vertices of $F$ are rigid, all valence two
    vertices are rigid or abelian.
  \item Each component of $F$ contains at most one essential vertex.
    If $F$ contains an essential vertex then that vertex is rigid.  In
    fact, every component of $F$ is star-shaped, and all
    non-valence-two vertices are correspond to rigid vertex groups of
    the \jsj\ of $G.$
  \end{itemize}
\end{lemma}

\par If $G\onto H$ is degenerate, indecomposable, and
$\scott(G)=\scott(H)$ then the \jsj\ decompositions of $G$ and $H$
resemble one another quite strongly.  The next step is to show that
under weakly \jsj\ respecting chains, one can choose, uniformly in
$\betti,$ subchains such that no nonabelian rigid vertex groups
have abelian image and such that rigid vertex groups are obtained from
images of previously occurring rigid vertex groups by iteratively
adjoining roots and passing to limit quotients.

\begin{lemma}
  \label{lem:admissablerefinement}
  Fix $n.$  For all $K$ there exists $M=M(K,n)$ such that if $\LL$ is
  $\Delta$--admissible degenerate chain with
  $\betti(\LL)=n,$ $\Vert\LL\Vert_{pl}\geq M,$ then there is a
  subchain $\LL'$ of $\LL$ such that
  \begin{itemize}
  \item $\Vert\LL'\Vert_{pl}\geq K$
  \item If $\Delta$ is the push-forward of $\Delta$ to $\LL'(1)$ then
    $\LL'$ is $\Delta$--stable.
  \item Let $k$ be an integer and let $\LL'_H$ (head?) and $\LL'_T$
    (tail?) be the two subchains of $\LL$ obtained by restricting to
    the first $k$ and last $\Vert\LL'\Vert-k$ indices, respectively,
    ($k$ may be $0$ or $\Vert\LL'\Vert$) such that $k$ is the last
    index for which the push forward of $\Delta$ to $\LL(k)$ is
    $\LL(k)$'s abelian \jsj\ decomposition.  Then there is a virtual
    \jsj\ decomposition $\Delta'$ of $\LL(k+1)$ for which $\LL'_T$ is
    $\Delta'$--admissible and the number of vertex groups of $\Delta'$
    is strictly greater than the number of vertex groups of $\Delta.$
  \end{itemize}
\end{lemma}

\begin{proof}
  Let $\Delta$ be the virtual \jsj\ decomposition of $\LL(1).$  We
  construct virtual abelian decompositions $\Delta_i$ for $\LL(i)$
  inductively.  Rather than constructing $\LL(i+1)$ from $\LL(i),$ first
  by applying strict vertex morphisms to rigid vertex groups from the
  \jsj\ of $\LL(i),$ we apply them only to the vertex groups of the
  virtual decomposition $\Delta_i.$  Let $R$ be a rigid vertex group
  of $\Delta_i,$ and construct
  $AS(R,\E(R))\cong\strict_{II}(R)*_{F_j}B_j$ as the
  construction of $\Phi_{as}(\LL(i))$ with respect to $\Delta$ demands.
  If $c_a(\varphi_{i,i+1}\vert_{R},\E)>0,$ i.e., if
  $\varphi_{i,i+1}$ is not $\Delta$--stable,
  $\betti(\strict_{II}(R))<\betti(R).$

  Now consider the relative \jsj\ decomposition of $\strict_{II}R.$  We
  claim that since $\Delta$ is a virtual \jsj\ decomposition of
  $\LL(i),$ the relative \jsj\ of $\strict_{II} R$ has at most one essential
  vertex and that vertex has the same valence as $R.$  If $\strict_{II} R$
  contains more than one essential vertex then, since
  $\scott(\LL(i))=\scott(\LL(i+1)),$ all but one of them is a valence
  one abelian vertex.

  Consider the direct limit $\dirlim \LL(i)_n=\LL(i+1).$ Repeated
  application of Lemma~\ref{lem:relativerigid} shows that the rigid
  vertex groups of $\LL(i+1)$ are obtained by iteratively adjoining
  roots to the vertex groups of $\strict_{II} R,$ as $R$ varies over
  all vertex groups of $\Delta_i.$ Since the Scott complexity doesn't
  increase, for all vertex groups of $\LL(i),$ nonconjugacy of
  incident edge groups is maintained and no inessential vertex group
  of $\LL(i)$ can give rise to a essential vertex of $\LL(i+1).$ For
  each rigid vertex group of $\Delta_i,$ declare the subtree of groups
  of the \jsj\ decomposition of $\jsj(\LL(i+1))$ induced by
  $\strict_{II} R$ a vertex group of an abelian decomposition
  $\Delta_{i+1}$ of $\LL(i+1).$ Since no new essential vertices are
  created, the decomposition $\Delta_{i+1}$ is a virtual
  \jsj\ decomposition.  Call the vertex of $\Delta_{i+1}$ associated
  to $R,$ the push forward of $R,$ $\varphi_{i,i+1,*}(R).$ Suppose
  $c_a(R,\E(R))>0,$ then $\betti(\strict R)<\betti(R).$ Since
  $\varphi_{i,i+1,*}(R)$ is obtained by iteratively adjoining roots to
  $\strict R,$ the betti number cannot increase and
  $\betti(\varphi_{i,i+1,*}(R))<\betti(R).$

  Fix $n,$ let $b_n$ be the largest number such that a vertex group
  $R$ in a virtual \jsj\ decomposition of a limit group $L$ with
  $\betti(L)=n$ can have $\betti(R)=b_n,$ and let $r_n$ be the maximum
  number of nonabelian rigid vertex groups in a virtual
  \jsj\ decomposition of a limit group $L$ with $\betti(L)=n.$  Let
  $M=K\cdot b_n^{r_n}.$  If $\LL$ has length $M$ then it has a
  subchain $\LL',$ of length $K,$ such that for every rigid vertex
  group $R_{i,j}$ of every decomposition $\Delta_i$
  $c_a(\varphi_{i,i+1}\vert_{R_{i,j}})=0$ (incident edge groups are
  implied).  Thus every homomorphism $\LL'(i)\to\LL'(j)$ is
  $\Delta$--stable, for $\Delta$ inherited from $\LL(1).$

  Now examine $\LL'_T.$  Let $R$ be a vertex group of
  $\LL'_T(1),$ and suppose that the relative \jsj\ of $R$ is
  nontrivial.  Then $R$ must have a splitting $R=R'*_EA*_{E'}R'',$
  $R'$ and $R''$ nonabelian, $A=\group{E,E'}$ and maximal abelian
  (recall that $A$ is the centralizer of its incident edge groups by
  our choice of normalization of the \jsj), and with all edge groups
  incident to $R$ elliptic and not centralized by $A.$  Then neither
  $R'$ nor $R''$ has abelian image in any $\LL'_T(k)$: if this
  is the case then, since limit groups have nontrivial homology
  relative to any abelian subgroup, the map $R\to\LL'_T(k)$
  would factor through (say) $R'*_EA*_{E'}(R'')^{ab},$ which has
  nonzero $c_a$ since the edges incident to $R$ are elliptic.  Since
  $\Delta_{i+1}$ is the outcome of a forest collapse, removal of the
  edges labeled $E$ and $E''$ from the forest gives a new forest which
  yields a new virtual \jsj\ decomposition $\Delta'$ of
  $\LL'_T(1).$  The new decomposition has more vertices than
  $\Delta,$ and, since neither $R'\to\LL'_T(k)$ nor
  $R''\to\LL'_T(k)$ has abelian image for any $k,$ the chain
  $\LL'_T$ is $\Delta'$--admissible.
\end{proof}

\par We use Lemma~\ref{lem:admissablerefinement} to control those
degenerations of \jsj\ decompositions which are invisible to the Scott
complexity.  Repeated application of the lemma uniformly many times in
the first betti number gives us the following theorem.

\begin{theorem}
  \label{thr:almostjsjpreservation}
  Let $\LL$ be a degenerate chain of limit groups with
  $\scott(\LL(i))=\scott(\LL(i+1))$ for all $i.$  Then for all $K$ there
  exists $M=M(K,\betti(\LL))$ such that if $\Vert\LL\Vert>M$ then there
  is an almost-\jsj\ respecting subchain $\LL'$ of $\LL$ with length at
  least $K.$
\end{theorem}

\begin{proof}
  We only need to show that such chains are $\Delta$--admissible
  for some $\Delta$ and that Lemma~\ref{lem:admissablerefinement} only
  needs to be used boundedly many times, depending only on $n.$

  By Lemma~\ref{lem:scottnovulnerable} no rigid vertex groups crossed
  by vulnerable paths in $AQ(\LL(1))$ or $QQ(\LL(1))$ might have abelian
  image in some $\LL(k).$ We construct a forest in $\jsj(\LL(1))$ as
  follows.  Let $p$ be a vulnerable path, and let $p'$ be the longest
  proper sub-path of $p$ with endpoints which are rigid vertex groups
  of $G.$  The sub-graph of groups spanned by the image of $p'$
  doesn't have abelian image in any $\LL(k),$ by
  Lemma~\ref{lem:scottnovulnerable}.

  For $p\in RR(\LL(1))\sqcup RA(\LL(1))\sqcup RQ(\LL(1))$ choose an
  orientation of $p$ let $p'$ be the subpath which begins at $p(1)$
  and ends at the last inessential rigid vertex group crossed by $p.$
  \mnote{redundant?}

  The forest which is the union of all images of $p'$'s constructed
  gives a virtual \jsj\ decomposition $\Delta$ of $\LL(1)$ by
  Lemma~\ref{lem:formofvirtual}, and the chain $\LL$ is
  $\Delta$--admissible by the previous paragraphs.

  Let $M(K,n)=K\cdot b_n^{r_n}$ be the constant from
  Lemma~\ref{lem:admissablerefinement}.  If $\Vert\LL\Vert>K\cdot
  b_n^{r^2_n}$ then $\LL$ is long enough to apply the lemma $r_n$
  times, yielding a subchain of length $K$ which is both
  \jsj--admissible and \jsj--stable.  If $\LL'$ is
  \jsj--admissible and \jsj--stable and for some rigid vertex
  group $R$ of $\LL'(k),$ the decomposition $\jsj(\strict_{II}
  R,\E)$ isn't trivial, then the \jsj\ decomposition of
  $\LL'(k+1)$ has strictly more rigid vertex groups than $\LL'(k).$
  Since $c_a(\varphi_{n,n+1})=0$ for all $n,$ in the group
  $\strictr(R)=\strict_{II}(R)*_{A_j}B_j,$ $B_j=A_j$ and
  $\strict_{II}(R) =\strictr(R),$ hence it's safe to call
  $\strictr(R)$ $\strict(R).$  If $\LL'$ has length $K\cdot r_n$ then
  it contains a subchain of length $K$ such that all relative
  decompositions of $\strict(R)$ are trivial, as the theorem asks.
\end{proof}

In the remainder of the section we define some technical refinements
of the notion of almost-\jsj\ respecting which are used in
Section~\ref{sec:decrease-comp}.


\begin{definition}
  A \gad\ $\Delta$ of a limit group $L$ \emph{misses $A$} if $A$ is an
  abelian vertex group of $\jsj(L)$ and no splitting of $L$ over a
  subgroup of $A$ is visible in $\Delta.$

  A \gad\ which misses $A$ is obtained from the abelian \jsj\ by
  collapsing the star of the vertex which carries $A$ and possibly
  collapsing further edges. If $\mathcal{A}$ is a collection of
  abelian vertex groups of $L$ then the \gad\ of $L$ obtained by
  collapsing all stars of elements of $L$ is denoted by
  $\jsj_{\mathcal{A}}(L).$ It follows from the construction of
  $\Phi_s$ that if $\varphi\colon G\onto H$ is degenerate and
  almost-$\jsj$ respecting then the vertex groups of
  $\jsj_{\phi_{\#}(\mathcal{A})}(H)$ are obtained from the images of
  the vertex groups of $\jsj_{\mathcal{A}}(G)$ by iteratively
  adjoining roots.
\end{definition}

We leave the following lemmas as an exercise.

\begin{lemma}
  Suppose $\varphi\colon G\onto H$ is almost-\jsj\ respecting, and that
  $R$ is a nonabelian non-\qh\ vertex group of
  $\jsj_{\mathcal{A}}(G).$ Then
  $\betti(R)\geq\betti(\varphi_{\#}(R)).$ If
  $c_a(\varphi,\E(R))>0$ then the inequality is strict.
\end{lemma}

Suppose $\varphi\colon G \onto H$ is almost-\jsj\ respecting. If
equality of betti numbers holds in the above lemma for all vertex
groups $R$ of $\jsj_{\mathcal{A}}(G),$ as $\mathcal{A}$ varies over
all collections of abelian vertex groups of $G,$ then we call
$\varphi$ \emph{$\betti$-respecting}.

\begin{theorem}
  \label{thr:existenceofjsjrespecting}
  \label{thr:jsjpreservation}
  \label{thr:jsjrespecting}
  For all $K$ there exists $M=M(K,\betti(\LL)),$ such that if $\LL$ is a
  \jsj\ stable degenerate chain and $\Vert\LL\vert\geq M,$ then there
  is a subchain $\LL'$ of $\LL$ of length at least $K$ such that all
  maps from $\LL'$ are $\betti$--respecting.
\end{theorem}

We call chain satisfying Theorem~\ref{thr:jsjrespecting}
\emph{\jsj\ respecting}.


\section{\qcjsj\ respecting}
\label{sec:qcjsjrespecting}

We start by proving a lemma about ranks of abelian subgroups of limit
groups. As a consequence, we can assume that degenerate
\jsj\ respecting chains have subsequences whose subsequences of
abelian subgroups are well behaved.

\begin{lemma}
  \label{lem:boundedabelianrank}
  Let $G$ be a limit group with $\betti(G)=n.$ Then all abelian
  subgroups of $G$ have rank at most $n.$ If $G\onto H,$ $A<G$
  abelian, if $\zee^2<\ker(A\to H)$ then $\betti(G)>\betti(H).$
\end{lemma}

\begin{proof}
  Let $G=L_0\onto L_1\onto\dotsb$ be a strict resolution of $G.$ Let
  $A<G$ be an abelian subgroup. If $A$ is cyclic then the lemma
  holds. Suppose $A$ has rank greater than two. Since $A$ isn't
  infinite cyclic it acts elliptically in the abelian \jsj\
  decomposition of $G,$ and is thus contained in an abelian vertex
  group or a rigid vertex group. If $A$ is contained in rigid vertex
  groups all the way down the resolution $L_0\onto\dotsb$ then, since
  the last group in the strict resolution is free, $A$ must be
  cyclic. Let $i_0$ be the first index such that $A$ is contained in
  an abelian vertex group but not completely contained in the
  peripheral subgroup of an abelian subgroup $B$ of $L_{i_0}.$ If
  $L_{i_0}$ is freely decomposable then we are done by induction since
  $A$ is contained in a free factor of $L_{i_0}.$ By linear algebra

  \[\rk(A\cap P(B))+\rk(B/P(B))\geq\rk(A)\] 

  Let $C$ be a complementary direct summand such that $A\cong (A\cap
  P(B))\oplus C.$ The summand $C$ has rank at most $\rk(B/P(B))$ and
  every element of $B/P(B)$ represents a nontrivial element of
  $\mathrm{H}^1(L_{i_0};L_{{i_0},P}).$ Continue this process on
  $L_{{i_0},P}$ with respect to $A\cap P(B),$ peeling off direct
  summands until $A\cap P(B)$ has rank $1.$

  By induction applied to $A\cap P(B)<L_{{i_0},P},$ which has lower
  betti number than $L_{i_0},$ $\rk(A\cap P(B))\leq n-\rk(B/P(B)).$
  Combining this with the above inequality, $\rk(A)\leq n.$ At the
  last step in the induction, when $A\cap P(B)$ has rank $1$ or $0,$
  the group chosen in the resolution has nontrivial homology relative
  to $A.$

  The above argument establishes that $\Ker(A\to \mathrm{H}_1(G))$ has
  rank at most one. 
\end{proof}

Suppose that $\LL$ is \jsj\ respecting. Let $\set{A^n_s}$ be
the set of abelian vertex groups, the set of conjugacy classes of
centralizers of edge groups of $\LL(n),$ and let
$\mathcal{A}_s$ be the sequence \[(A^n_s\to A^{n+1}_s)\] The
homomorphism $\varphi_{n,m}$ may not be injective on $A^n_s,$ but
since $\varphi_{n,m}$ doesn't factor through a free product, it's
image is not trivial. Define in a similar manner $E^n_t$ and
$\E_t$ for the edge groups $E^n_t$ of $\LL(n)$ and
sequences of edge groups of $\LL.$ If
$\rk(A^m_s)<\rk(A^n_s)-2$ or $\rk(E^m_t)<\rk(E^n_t)$ then
$\betti(\LL(m))<\betti(\LL(n)).$ Thus if $\rk(E^n)>2$
and the sequence $\betti(L_i)$ is constant, then $A^m_s$ has rank at
least two, as does $A^{m'}_s$ for all $m'>m.$

\begin{definition}
  \label{def:qcjsjrespecting}
  A degenerate \jsj\ respecting chain $\LL$ is
  \qcjsj\ respecting if, for each sequence
  $\mathcal{S}\in\set{\mathcal{A}_s,\E_t},$ exactly one of
  the following holds:
  \begin{enumerate}
    \item $\rk(S^n)=\rk(S^m)>2$ and $S^n\into S^m$ for all $m$ and $n$
    \item $\rk(S^n)=\rk(S^m)>2$ and $S^n\to S^m$ has infinite
      cyclic kernel for all $n$ and $m.$
    \item $\rk(S^n)=\rk(S^m)=2$ and $S^n\into S^m$ for all $m$ and $n$
    \item $\rk(S^n)=2$ and $S^n\to S^{n+1}$ has infinite cyclic image
    \item $S^n\cong\zee$ for all $n.$
  \end{enumerate}
\end{definition}

A sequence of edge groups satisfying one of the first three bullets
is \term{big}, otherwise it is \term{small}. A sequence of edge
groups satisfying either of the second or fourth bullets is
\term{flexible}, otherwise it is \term{rigid}.

The main application of Lemma~\ref{lem:boundedabelianrank} is that
\qcjsj\ respecting sequences can be derived from \jsj\ respecting
sequences.

\begin{lemma}\mnote{This can be made much better. The bigger edges are a red herring.}
  \label{lem:almostcyclicedges}
  For all $K$ there exists $M=M(K,b)$ such that if $\LL$ is a
  degenerate \jsj\ respecting sequence and
  $\Vert\LL\Vert>M(K,\betti(\LL)),$ then $\LL$ has a
  \qcjsj\ respecting subsequence with length at least $K.$
\end{lemma}

\begin{proof}
  Follows immediately from the discussion prior to the lemma, the
  bound on the number of edge groups depending only on $\betti,$ and
  Lemma~\ref{lem:boundedabelianrank}.
\end{proof}

\begin{definition}
  \label{def:quasicyclicjsj}
  Degenerate \qcjsj\ respecting chains are cooked up so that certain
  important sequences of subgroups are respected under all maps.  Let
  $\LL$ be \qcjsj\ respecting. For each \jsj\ decomposition
  $\jsj(\LL(i)),$ form a new decomposition $\jsj_B(\LL(i))$ by folding
  together all big incident edges incident to abelian vertex groups,
  as in Figure~\ref{fig:bigfolding}. For each abelian vertex group $A$
  of $\LL(k),$ let $P_B(A)$ be the subgroup of $A$ generated by big
  incident edges.

\begin{figure}[ht]
  \psfrag{A}{$A$}
  \psfrag{pba}{$P_B(A)$}
  \psfrag{big}{big}
  \psfrag{small}{small}
  \psfrag{fold big edges}{fold big edges}
  \centerline{%
    \includegraphics{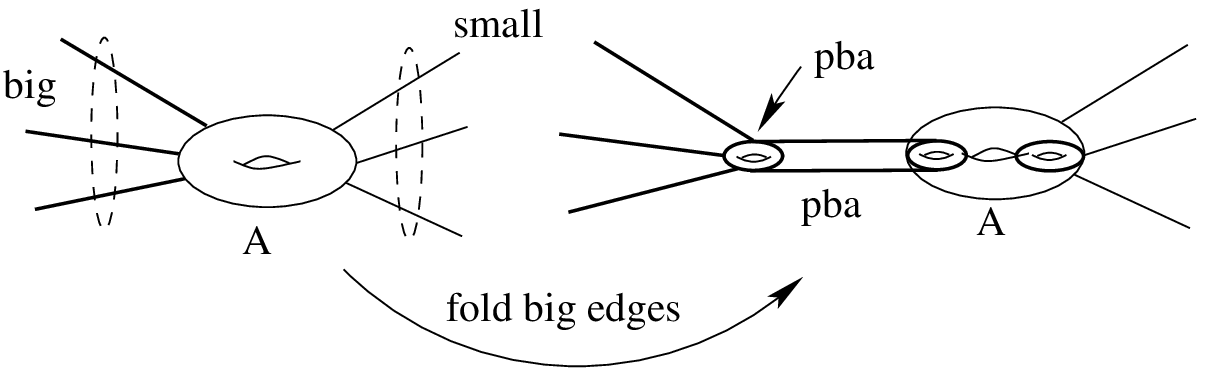}
  }
  \caption{Folding big edges together.}
  \label{fig:bigfolding}

\end{figure}

Let $\LL$ be a \qcjsj\ respecting sequence, and let
$b^k_1,\dotsb,b^k_n$ be the edges of $\jsj_B\LL(k)$ which carry big
edge groups (the edges labeled $P_B(A),$ for some abelian vertex group
$A$ from the \jsj), and let $\Gamma^k_1,\dotsb,\Gamma^k_m$ be the
connected components of the union of the $b^k_i.$ The subgroups
carried by the $\Gamma^k_j,$ are the \term{rigid vertex groups of the
  \qcjsj\ decomposition}, and the graph of groups obtained by
collapsing the subgraphs $\Gamma^k_j$ of $\jsj_B(\LL(k))$ is the
\term{quasicyclic \jsj\ decomposition}, or \qcjsj, for short.
\end{definition}

For each abelian vertex group $A^n$ of $\LL(n)$ we have isolated
$P_B(A^n),$ the subgroup generated by big incident edges. There is
another special subgroup, $P_S(A^n),$ the subgroup generated by small
incident edges.  Since big$\vert$small edges of $\LL(n)$ map to
big$\vert$small edges of $\LL(n+1),$ there is a map 
\[P_{B\vert S}(A^n)\to P_{B\vert S}(A^{n+1})\]

\par Let $\set{R^k_i}$ be the collection of vertex groups in
$\qcjsj(\LL(k)).$  Since $\varphi_{k,l}$ sets up a one to one
correspondence between vertex groups and big$\vert$small edges of of
$\jsj(\LL(k))$ and $\jsj(\LL(l)),$ and the induced maps $\LL(k)\to\LL(l)$
respect vertex groups and big$\vert$small edges and the incidence
conditions of the \jsj\ decompositions, there is a one-to-one
correspondence between vertex groups of $\qcjsj(\LL(k))$ and
$\qcjsj(\LL(l)),$ i.e., there is a unique vertex group $R^l_i$ in
$\qcjsj(\LL(l))$ such that $\varphi_{k,l}(R^k_i)<R^l_i.$

Recall our normalization that every edge group be centralized by an
abelian vertex group and that every edge group is primitive.  Since
$\LL$ is \qcjsj\ respecting, the kernels of the induced maps on
$P_B(A^n)$ and $P_S(A^n)$ are at most infinite cyclic, and if both of them
have kernel, the kernel is contained in the intersection $P_B(A^n)\cap
P_S(A^n).$

Define the following sequences of subgroups of
$\LL$:
\[\R^n_i(m)\define\varphi_{n,m}(R^n_i)\]
There are obvious inclusions $\R^n_i\into\LL.$

\begin{lemma}
  \label{lem:qcjsjvertexgroups}
  {$\mbox{}$}
  \begin{itemize}
  \item $\R_i(n+1)=\R^{n+1}_i(n+1)$ is obtained by iteratively
    adjoining roots to $\R^n_i(n+1).$
  \item $\betti(R^n_i)\leq\betti(\LL(n))$
  \item If $\betti(\R_i)=\betti(\LL)$ for some $i$ then there is only
    one nonabelian vertex in each \qcjsj\ decomposition.
  \end{itemize}
\end{lemma}

This essentially follows from the definitions.

\begin{proof}


  The first bullet follows from the fact that the vertex groups of
  $\qcjsj(\LL(n))$ are generated by vertex groups from $\jsj(\LL(n))$
  and stable letters from $\jsj(\LL(n)),$ and that the vertex groups of
  $\jsj(\LL(n+1))$ are obtained from the images of vertex groups of
  $\jsj(\LL(n))$ by iteratively adjoining roots, by
  Lemma~\ref{lem:relativerigid}.

  Let $A$ be an abelian vertex group of $\LL(n)$ such that
  $\Ker(A\to\LL(n+1))\cong\zee.$ Then at least one edge incident to
  $A$ is flexible. The only case which needs consideration is when
  none of the big edges incident to $A$ is flexible. In this case the
  rank of $P_B(A)$ is strictly less than the rank of $A,$ and the
  image of any small flexible edge adjacent to $A$ has at most rank one
  intersection with $P_B(A).$

  The inequality of betti numbers follows from the fact that if $E$ is
  a small flexible edge, then either $E$ is incident to an abelian
  vertex group $A$ with no big flexible edges or with at least one big
  flexible edge. In the former case, the vertex group has betti number
  at least one less than that of the subgroup obtained by attaching
  $A$ to the rigid vertex group along $P_B(A),$ which has betti number
  at most one greater than the ambient group. In the latter case, both
  inclusions of $E$ into vertex groups of the one edged splitting
  induced by $E$ have, homologically, one dimensional images.

  The third bullet follows from the same argument.
\end{proof}


The following follow from Theorem~\ref{thr:jsjrespecting},
Lemma~\ref{lem:almostcyclicedges}, and
Lemma~\ref{lem:rankheightbound}.

\begin{theorem}[Alignment Theorem]
  \label{thr:alignmenttheorem}
  Let $\LL$ be an indecomposable sequence of limit groups,
  $\rk(\LL)=N.$ For all $K$ there exists $M$ such that if
  $\Vert\LL\Vert>M$ then there is a maximal \qcjsj\ respecting
  resolution $\wt{\LL}\rto\LL$ with $\Vert\wt{\LL}\Vert\geq K.$
\end{theorem}

\begin{corollary}[Alignment Corollary]
  \label{cor:alignmentcorollary}
  Let $(\iota\colon\G\into\LL)\in\seq(\LL,b,d)$ be indecomposable. For
  all $K$ there exists $M=M(\comp(\iota))$ such that if
  $\Vert\G\Vert>M$ then there is a maximal \qcjsj\ respecting
  resolution $\wt{\G}\rto\G$ with $\Vert\wt{\G}\Vert\geq K.$
\end{corollary}


\section{Lifting dimension bounds}
\label{sec:liftdimensionbound}

\par Suppose $\LL$ is a sequence of epimorphisms of
$N$--generated limit groups. There is a trivial resolution
$\LL\rto_{id}\LL$ which is simply the identity
resolution. This resolution has complexity
$\comp(\LL\rto\LL)\leq(N,6N).$


As a consequence of Theorem~\ref{thr:reduction-to-indecomposable}
we have the following important corollary.

\begin{corollary}[Reduction to indecomposable sequences]
  \label{cor:reduction-to-indecomposable}
  Let $\iota\colon\G\into\LL$ be an inclusion of sequences. For all $K$
  there exists $M=M(K,\comp(\iota))$ such that if $\Vert\G\Vert>K$
  there exists a maximal resolution $\wt{\G}\rto\G\into\LL$ of $\G$
  such that $\Vert\wt{\G}\Vert>K$ such that $c_{fd}$ is constant along
  $\wt{\G}.$

  In particular, $\wt{\G}$ splits as a graded free product of
  sequences
  \[\widetilde{\G}=\widetilde{\G}_1*\dotsb*\widetilde{\G}_p*\mathcal{F}\]
  where $\mathcal{F}$ is the constant sequence $(\free_q)$ for some
  $q.$ The sequences $\wt{\G}_i\rto\G$ are indecomposable maximal
  resolutions of their images. If $\scott(\wt{\G})>(1,0)$ then
  \[\comp(\wt{\G}_i\rto\LL)<\comp(\G\into\LL)\]
\end{corollary}

\begin{proof}[Proof of Corollary~\ref{cor:reduction-to-indecomposable}]
  By Lemma~\ref{lem:rankheightbound}, the rank of $\G$ is bounded
  above by some function of $\comp(\G).$ Now apply
  Theorem~\ref{thr:reduction-to-indecomposable}.
\end{proof}


\begin{remark}
\label{rem:thisistheproof}
Let $(b_0,d_0)$ be a minimal complexity for which
Theorem~\ref{thr:subseqkrull} fails, should there be one.  Let
$\wt{\G}\rto\LL$ be a maximal resolution provided by the corollary.

Suppose that $\wt{\G}$ consists of freely decomposable groups and
decomposes as a nontrivial free product of
resolutions \[(\wt{\G}_1*\dotsb*\wt{\G}_p*\mathcal{F}_q)\rto\LL\] Each
of the free factors gives rise to a maximal
resolution \[\wt{\G}_i\rto\LL,\quad\mathcal{F}_q\rto\LL\] of complexity
less than $\comp(\G\rto\LL),$ and with equality only if $(p,q)=(1,0)$

\par Suppose we are in the case $p=1,q>0$ or $p\geq 2.$ In this case
there are at most $N$ (this is an overestimation) free factors of
$\wt{\LL},$ each of which has complexity strictly less than that of
$\LL.$ Let $B$ be the maximal proper length of a sequence with
complexity less than that of $\G.$ Then if $M>3BN$ there exist three
consecutive indices $i,i+1,i+2$ such that the maps
\[\wt{\G}_j(i)\to\wt{\G}_j(i+1)\to\wt{\G}_j(i+2)\] are isomorphisms
for all $j.$ The same is true for the free part of $\wt{\G}.$ It
follows immediately that $\G(k_i)\to\G(k_{i+1})$ must be an
isomorphism.

Thus if Theorem~\ref{thr:subseqkrull} fails, by the corollary and the
above, if $\G_i\into\LL\in\seq(\LL,b,d)$ is a sequence such that
$\Vert\G_i\Vert_{pl}>i$ then the maximal resolutions $\wt{\G_i}\rto\G_i$ of
$\G_i\into\LL$ must have $\scott(\wt{\G})=(1,0).$


\par Suppose now that we are in the case $(p,q)=(1,0).$ As before, if
$\G$ has length $M=M(K,N),$ the sequence $\wt{\G}$ has length at least
$K$ and consists of indecomposable maps. Since $\wt{\G}$ is
\qcjsj\ respecting, the sequence $c_a(\wt{\G})$ is constant.  We now
show how to lift a dimension bound for sequences simpler than $\G$ in
the event that $c_a(\wt{\G})>0.$


\par Since $c_a$ is constant along $\wt{\G},$ the maps
$\wt{\G}(i)\to\wt{\G}(j)$ map $\wt{\G}(i)_P$ onto $\wt{\G}(j)_P.$ If
$\wt{\G}(i)\to\wt{\G}(j)$ is \jsj\ respecting, every automorphism in
$\Mod(\wt{\G}(i),\wt{\G}(i)_P)$ pushes forward to an element of
$\Mod(\wt{\G}(j),\wt{\G}(j)_P),$ hence if the map
$\wt{\G}(i)_P\to\wt{\G}(j)_P$ is an isomorphism then
$\wt{\G}(i)\to\wt{\G}(j)$ is
$\Mod(\wt{\G}(i),\wt{\G}(i)_P)$--strict. Since every automorphism
pushes forward, this map is an isomorphism. Now if $c_a(\wt{\G})>0$
then $\betti(\wt{\G}_P)<\betti(\wt{\G})$ and
\[\comp(\wt{\G}_P\rto\G)<\comp(\wt{\G}\rto\G)\]

The same sort of analysis can be carried out if $\wt{\G}$ has a
sequence of \qh\ subgroups: Let $L$ be a limit group and let
$\jsj_Q(L)$ be the graph of groups decomposition of $L$ obtained by
collapsing all edges not adjacent to \qh\ subgroups. Then $\wt{\G}$
respects the decompositions $\jsj_Q(L),$ and maps the vertex groups of
$\jsj_Q(\wt{G}(i))$ \emph{onto} the associated vertex groups of
$\jsj_Q(\wt{\G}(i+1)).$ Thus we derive boundedly many (in
$\betti(\wt{G})$) sequences $\mathcal{V}_i\into\wt{\mathcal{G}}$ of
vertex groups. The resolutions $\mathcal{V}_i\rto\LL$ have lower
complexity than $\wt{\G}\rto\LL.$ Thus, given $K$ there exists
$M(\betti(\mathcal{G}),K)$ such that of $\Vert\wt{\G}\Vert>M$ then it
has a subsequence of length $K$ such that all maps on vertex groups
are injective. As above, modular automorphisms from
$\Mod(\G,\jsj_Q(\G))$ push forward and we discover that some
$\G(i)\onto\G(i+1)$ must be an isomorphism. Thus we may assume that
the \qcjsj\ respecting sequences derived have no \qh\ subgroups.
\end{remark}


The above remark essentially contains a proof of
Theorem~\ref{thr:subseqkrull}.  If a \qcjsj\ $\wt{\G}\rto\G$
respecting resolution has no \qh\ subgroups and all abelian vertex
groups are equal to their peripheral subgroups we must work a little
harder and use an analysis of sequences of images of vertex groups of
$\qcjsj(\wt{\G})$ to conclude that if $\wt{\G}$ is sufficiently long
then it contains an isomorphism. In the next section we show how to
apply the construction of \qcjsj\ respecting resolutions multiple
times\footnote{Twice.} in order to find sequences of strictly lower
complexity.


\par Our application of Theorem~\ref{thr:addrootstolimitgroups} is a
way to produce this ``lift'' of a dimension bound from sequences of
vertex groups to a dimension bound for the ambient
chains. Theorem~\ref{thr:alignmenttheorem} is used to express \qcjsj\
respecting sequences as graphs of sequences of groups obtained by
passing to quotients and iteratively adjoining
roots. Theorem~\ref{thr:addrootstolimitgroups}, on the other hand, is
only stated (and possibly only true for) sequences obtained by
adjoining roots a single time along a single fixed collection of
elements. To cope with this deficiency we construct a collection of
subsequences of subgroups to which
Theorem~\ref{thr:addrootstolimitgroups} can be applied.

Let $\R_i$ be a sequence of vertex groups of $\qcjsj(\G).$ Let
$E_1,\dotsb,E_m$ be the edge groups incident to some vertex group
$\R^{n-1}_i(n-1)<\G(n-1),$ and let $F_1,\dotsb,F_m$ be the
corresponding edges of $\G(n).$ Let $C_j$ be the closure of the image
of $E_j$ in $F_j,$ the subgroup of $F_j$ consisting of all elements
which have powers lying in the image of $E_j.$ Now let
$\mathcal{S}_i(n)=\group{\R^{n-1}_i(n)*_{\img(E_j)}C_j}<\R^n_i(n)$
and $\mathcal{S}^n_i(m)=\varphi_{n,m}(\mathcal{S}_i(n)).$ We leave it
as an exercise for the reader to show that $\mathcal{S}_i(n)$ is
obtained from $\mathcal{S}^{n-1}_i(n)$ by adjoining roots to the
collection
\[
  \E^i_n=\set{\mathcal{S}^{n-1}_i(n)\cap C_j}_{j=1..m}
\] 
\latin{A priori}, it is only obtained from $\R^{n-1}_i(n)$ by adjoining
roots. Note that $\Vert\E^i_n\Vert\leq 2\betti(\G).$

\begin{theorem}[Krull assuming short sequences of vertex groups]
  \label{thr:stratified}
  Let $\LL$ be a freely indecomposable \qcjsj\ respecting $\jsj$
  respecting chain of limit groups, and
  suppose \[\LL\rto\H,\quad\comp(\LL\rto\H)=(b,d)\]
  Furthermore, assume that, for all sequences $\R^n_i<\LL$ of images of
  vertex groups of quasi-cyclic \jsj\ decompositions,
  $\Vert\R^n_i\Vert_{pl}<D.$

  Then there is a constant $D'=D'(D,b,d)$ such that
  $\Vert\LL\Vert_{pl}<D'.$
\end{theorem}

Note that the presence of $\H$ is not strictly necessary. We
only include it so the statement of the theorem meshes more smoothly
with the way it is used.

\begin{proof}[Proof of Theorem~\ref{thr:stratified}.]
  Suppose there exist such $\LL$ of arbitrary proper length. First,
  suppose that the sequences $\R^n_i$ are in fact constant.  Then
  there are no flexible edges and the quasi cyclic \jsj\ decomposition
  agrees with the cyclic \jsj\ decomposition and all peripheral
  subgroups of abelian vertex groups embed in $\LL(n+1).$  By
  Lemma~\ref{lem:simpleenvelope} the envelopes of all rigid vertex
  groups of $\LL(n)$ embed in $\LL(n+1),$ therefore $\LL(n)\to\LL(n+1)$ is
  strict.


  Observe now that every element of the modular group of $\LL(n)$
  pushes forward to $\Mod(\LL(n+1)).$ It is then an easy exercise to
  show that such a strict epimorphism is an isomorphism.

  Since $\mathcal{S}^n_i$ is a sequence of subgroups of
  $\R^n_i,$ they have proper length bounded by $D$ as well.

  The number of sequences of vertex groups of $\qcjsj(\LL)$ depends
  only on $(b,d)$ by the proof of Lemma~\ref{lem:rankheightbound}
  (Or acylindrical accessibility. See~\cite{sela::acyl}
  or~\cite{weid::acyl}). Call the bound $B=B(b,d).$ Then if
  $\Vert\LL\Vert>D^{BK}$ then there is a subsequence $\LL'$ of
  length $K$ such that for each $n$ and $i$ the sequences
  \[\mathcal{S}^n_i(n+1)\onto\dotsb\eqno{(\star)}\] are constant. By
  Theorem~\ref{thr:addrootstolimitgroups} applied to the pair of sequences
  (indexed by $n$)
  \[
  ((\mathcal{S}^{n-1}_i(K)),(\mathcal{S}^n_i(n)),\E^i),\quad n<K
  \]
  for all but $C=C(b,d,\vert\E(\R^n)\vert)$ indices, the maps
  $\mathcal{S}^{n-1}_i(n-1)\onto\mathcal{S}^{n-1}_i(K)\cong\mathcal{S}^{n-1}_i(n)$
  are isomorphisms. Since the number of edge groups incident to a
  rigid vertex group only depends on $b,$ $C$ is independent of $\LL.$

  By construction,
  $\mathcal{S}^n_i(n)$ intersects each small flexible edge incident to
  $\R^n_i(n)$ in a finite index subgroup.

  For each index $n$ such that
  $\mathcal{S}^{n-1}_i(n-1)\onto \mathcal{S}^{n-1}_i(n)$ is an
  isomorphism, we then have that a small flexible edge incident to
  $\R^n_i(n)$ must embed in $\R^{n+1}_i(n+1),$
  contradicting the construction of the \qcjsj\ decompositions. We
  conclude that the \qcjsj\ decompositions coincide with the cyclic
  \jsj\ decompositions, and that $\R^n_i=\mathcal{S}^n_i.$

  Since the rank of $\R^n_i$ is bounded in terms of the rank of the
  resolution $\LL\rto\H,$ $C$ depends only on $(b,d).$ Since
  the number of rigid vertex groups is bounded by $B,$ if
  $\Vert\LL\Vert>C^B$ and $\LL$ satisfies $(\star),$ then $\LL$ contains
  an isomorphism by the first two paragraphs of the proof. For general
  $\LL\rto\H$ satisfying the hypotheses of the theorem, if
  $\Vert\LL\Vert>D'=D^{B(n)C(N)^{B(N)}}$ then $\LL$ contains an
  isomorphism.
\end{proof}



\section{Decreasing the complexity}
\label{sec:decrease-comp}


We prove a useful lemma before we begin.

\begin{lemma}
  \label{lem:smallishbettinumbers}
  Suppose $\varphi\colon G\onto H$ is degenerate and
  indecomposable. Let $\Delta$ be a cyclic decomposition of $G,$ and
  suppose that a noncyclic vertex group of $\Delta$ with first betti
  number at least two has cyclic image in $H.$ Then the vertex groups
  of the cyclic \jsj\ decomposition, and therefore of the \qcjsj\ as
  well, of $H$ satisfy $\betti(H_v)<\betti(H).$
\end{lemma}

\begin{proof}
  This is proven by following the construction of $\Phi_sH.$ Build
  $\phias(\varphi)\colon\phias(G)\onto H$ with respect to the
  decomposition $\Delta'$ obtained by collapsing all edges not
  adjacent to the distinguished vertex $v.$ Since the theorem is
  trivially true if $H$ has any \qh\ vertex groups or principle cyclic
  splittings of the form $A*_{\zee}G',$ we may assume that
  $\phias(\Delta'),$ the decomposition obtained by pushing forward
  $\Delta'$ and blowing up vertex groups into their relative \jsj\
  decompositions, has no \qh\ vertex groups. There is a bijection
  between nonabelian vertex groups of $\phias(\Delta')$ and nonabelian
  vertex groups of $H.$ Let $W$ be a nonabelian vertex group of
  $\Phi_{as}(G).$ The associated vertex group of $H$ is obtained from
  $W$ by iteratively adjoining roots. Since $\Phi_{as}H$ is almost
  strict, it embeds $W,$ hence if $E$ is an edge group adjacent to $W$
  and if the centralizer of $E$ in $W$ is noncyclic, the centralizer
  of the associated edge of $H$ has noncyclic centralizer. Since every
  edge of $H$ centralizes the image of some edge of $\phias(\Delta'),$
  and since every limit group has a principle cyclic splitting, there
  is some $\sim_a$ equivalence class $\mathcal{A}$ (recall that
  $\sim_a$ is the equivalence relation on abelian vertex groups and
  edge groups generated by adjacency) which contains only infinite
  cyclic edge and vertex groups, and such that if an edge $e$ of
  $\mathcal{A}$ is adjacent to a nonabelian vertex group $W$ of
  $\phias(\Delta'),$ then the centralizer of $e$ in $W$ is infinite
  cyclic. Moreover, at least two (oriented) edges from such a class
  must be adjacent to nonabelian vertex groups.
  
  Since $\varphi$ is degenerate, we may write $H$ as a direct limit \[
    H=\dirlim G_n,\quad G_n\define(\phir\circ\phia)^n\phias(G)\] Let
  $\mathcal{A}_{n,i},i=1..m_n$ be the collection of $\sim_a$
  equivalence classes for $G_n.$ By the discussion above, every
  equivalence class for $G_n$ contains the image of some equivalence
  class of $G_{n-1}$ and there is a (possibly one-to-many) map
  $\set{\mathcal{A}_{n,i}}_{i=1..m_n}\to\set{\mathcal{A}_{n-1,j}}_{j=1..m_{n-1}}.$

  Since the direct limit has finite length, $H=G_m$ for some
  $m.$ Choose $\sim_a$ equivalence classes $\mathcal{A}_{j,i(j)},$
  $j=0..m$ such that the image of $\Gamma_{\mathcal{A}_{j,i(j)}}$ is a
  vertex of $\mathcal{A}_{j+1,i(j+1)}.$ If the distinguished vertex
  $v$ is an element of $\mathcal{A}_{0,i(0)}$ then the subgraphs of
  groups carried by the complimentary components of $v$ have lower
  first betti number than $\phias(G).$  Otherwise, an easy homological
  argument shows that the complimentary components of
  $\mathcal{A}_{0,i(0)}$ have lower first betti number than
  $\phias(G).$ This state of affairs is unchanged by an application of
  $\phia,$ and remains unchanged after an application of $\phir.$ We
  see in the limit that to $\mathcal{A}_{m,j(m)}$ there is an
  associated principle cyclic splitting (there may be more than one),
  and that the vertex group of this splitting has strictly lower first
  betti number than the ambient group $H.$
\end{proof}

\par To prove Theorem~\ref{thr:subseqkrull}, of which
Theorem~\ref{thr:krulldimension} is a consequence, we mimic Sela's
construction of the cyclic analysis lattice, but for sequences of
limit groups. We don't build an entire analysis lattice, only the few
branches necessary for the inductive proof of
Theorem~\ref{thr:subseqkrull}.

\par At each level of the cyclic analysis lattice of a limit group, we
pass either to the freely indecomposable free factors of a limit
group, or, if freely indecomposable, we pass to the vertex groups of
the cyclic \jsj\ decomposition. The construction of
\qcjsj\ decompositions and maximal Grushko--respecting resolutions
give us the ability to mimic this construction for sequences of limit
groups.

\par Let $V$ be a vertex group in the cyclic \jsj\ of a limit group
$L.$ The \term{neighborhood} of $V,$ as opposed to the envelope, is
the subgroup of $L$ generated by $V$ and the centralizers of incident
edge groups. The neighborhood of such a $V$ always has the form
$V*_{E_i}A_i,$ where $E_i$ are the edge groups incident to $V$ such
that the centralizer of $E_i$ isn't contained in $V$ (it can be
contained in a conjugate of $V$), and the $A_i$ are the centralizers
of such $E_i.$

\par If $G<L$ is a subgroup of $L$ which is freely indecomposable
relative to a collection of $\jsj_C(L)$--elliptic subgroups then every
abelian vertex group of the decomposition of $G$ inherited by its
action on the minimal $G$--invariant subtree of the Bass-Serre tree
for the cyclic \jsj\ of $L$ is adjacent to a nonabelian vertex group.

\begin{lemma}
  \label{lem:decreasecomplexity-maybethisworkseasily}
  Let $\G\rto\img(\G)\into\LL$ be a maximal \qcjsj\ respecting
  resolution of a subsequence of $\LL.$ Let $\R$ be a sequence of
  vertex groups of $\qcjsj(\G),$ and suppose that $\mathcal{P}$ is an
  indecomposable maximal \qcjsj\ respecting resolution of the image
  $\img(\R)\into\LL$ and $c_a(\mathcal{P}(n))=0$ for all $i.$ Suppose
  further that $\mathcal{P}$ has no \qh\ sequences of vertex groups,
  only one sequence $\mathcal{M}$ of vertex groups in its
  \qcjsj\ decomposition, and that
  $\betti(\mathcal{M})=\betti(\mathcal{G}).$ Let $\set{1,\dotsc,m}$ be
  the index set for $\R.$ Then the vertex groups of $\mathcal{M}(n)$
  map to a neighborhood of a vertex group of $\jsj_C(\img(\G)(n))$ for
  $m-2\geq n\geq 2.$
\end{lemma}

\begin{proof}
  Let $\mathcal{A}_i$ be the set of sequences of big
  (Definition~\ref{def:qcjsjrespecting}) conjugacy classes of abelian
  vertex groups of $\mathcal{P}.$ Let
  $\Delta_n=\Delta_n(V^n_i,A^n_j,E^n_k)$ be the decomposition of
  $\mathcal{P}(n)$ obtained by collapsing all edges not adjacent to
  some $\mathcal{A}^n_i(n).$ This decomposition has the property that
  every vertex group $A^n_j$ has elliptic image in
  $\img(\G)(n).$ Moreover, since $\mathcal{P}$ is \qcjsj\ respecting,
  $c_a(\varphi_{n,n+1}\vert_{V^n_i})=0.$ By definition $\mathcal{P}$
  is $\Delta$--stable.


  The first thing to show is that the image of each $V^n_i$ in
  $\G(n+1)$ is contained in a vertex group of $\jsj(\G(n+1)).$ This
  fact will be used in the second half of the proof to show that if
  the vertex groups of a special decomposition of $\mathcal{P}$ don't
  map to neighborhoods of vertex groups of $\mathcal{G},$ then there
  is an intermediate group between $\mathcal{P}(n)$ and
  $\mathcal{P}(n+2)$ which has more noncyclic abelian vertex groups
  than $\mathcal{P}(n),$ which we use to derive a contradiction to the
  assumption that $\mathcal{P}$ is \jsj\ respecting. Consider the
  commutative diagram in Figure~\ref{fig:whatwasthisfor}.

\begin{figure}[h]
  \centerline{%
    \xymatrix{%
      \mathcal{P}(n-2)\ar[r]^{\varphi_{n-2,n-1}}\ar[d]_{\pi_{n-2}} & \mathcal{P}(n-1)\ar[r]^{\varphi_{n-1,n}}\ar[d] & \mathcal{P}(n) \\
      \img(\R)(n-2)\ar[r]\ar[ur]\ar[dr]_{\psi'_{n-2}} & \img(R)(n-1)\ar[ur]_{\psi_{n-1}} \\
                    & \R(n-1)\ar[u]_{\pi'_{n-1}}<\G(n-1)
}}
\caption{Deducing that the abelian \jsj\ of $\mathcal{P}$ resembles
  the induced decomposition of $\R.$}
\label{fig:whatwasthisfor}
\end{figure}

Let $\Delta$ be the decomposition of $\R(n-1)$ induced by
$\jsj_B(\G(n-1)).$ We can write $\Delta=\Delta(W_i,B_j,F_k),$ where
$W_i$ are rigid vertices from $\jsj(\G(n-1)),$ the $B_j$ are subgroups
$P_B(A_j)$ of abelian vertex groups of $\G(n-1),$ and the $F_k$ are
big edge groups connecting them. To be included in $\Delta,$ a vertex
group $P_B(A)$ of $\jsj_B(\G(n-1))$ must have at
least two big incident edges, otherwise the associated one-edged
splitting is trivial. Since $\mathcal{P}$ is degenerate, the
composition $\psi_{n-1}\circ\pi'_{n-1}$ is degenerate as well, since
\[\varphi_{n-1,n}\circ\varphi_{n-2,n-1}=\psi_{n-1}\circ\pi'_{n-1}\circ\psi'_{n-2}\circ\pi_{n-2}\]
Now construct
$\Phi_s(\psi_{n-1}\circ\pi'_{n-1})\colon\Phi_s(\R(n-1))\onto\mathcal{P}(n),$
starting with the decomposition $\Delta.$ We start by building an
almost-strict $\phias(\R(n-1))\to\mathcal{P}(n).$ For each vertex
group $W_i$ we build first
$\strictr(W_i)=\strict_{II}(W_i)*_{E_{i,l}}B_{i,l}.$ If
$F\in\E(W_i)$ then the image of $F$ is contained in some
$B_l.$ Since the edge groups incident to $W_i$ are nonconjugate, and
remain nonconjugate in $\G(n),$ each $B_{i,l}$ contains a unique such
$F.$ Let $W_{i,1},\dotsb,W_{i,n_i}$ be the vertex groups of the
relative \jsj\ of $\strict_{II}(W_i).$ The image of $W_{i,j}$ in
$\G(n+1)$ is contained in a vertex group of the \jsj, and since the
nonabelian vertex groups of $\mathcal{P}(n)$ are obtained from the
vertex groups $W_{i,j}$ by iteratively adjoining roots, the vertex
groups of $\mathcal{P}(n)$ have images contained vertex groups of
$\G(n+1).$ Moreover, for $i\neq j,$ $W_{i,j}$ and $W_{i',j'}$ have
image contained in district vertex groups of $\G(n+1).$ In order to
check the condition that the vertex groups of $\Delta_n$ have images
contained in vertex groups of $\G(n+1)$ all we need to check is that
if some $W_{i,j}$ and $W_{i',j'}$ have incident edges which are in the
same $\sim_a$ equivalence class $\left[\mathcal{A}\right],$ then the
vertex group of $\mathcal{P}(n)$ which centralizes the image of the
subgroup carried by $\left[\mathcal{A}\right]$ isn't infinite
cyclic. Since $c_a(\mathcal{P})=0$ the topology of the underlying
graph of the decomposition of $\phias(\R(n-1))$ is identical to the
topology of the underlying graph of the abelian \jsj\ of
$\mathcal{P}(n),$ hence if there is such a pair of groups then there
is a path $p$ between $W_{i,j}$ and $W_{i',j'}$ in the graphs of
groups decomposition of $\phias(\R(n-1))$ which corresponds to the
path between their image vertex groups which crosses the centralizer
of their intersection. If $i\neq i'$ then $p$ passes through some
vertex $\strictr(B_k)$ for some $k.$ Since $B_k$ is big, the image of
$B_k$ in $\mathcal{P}(n)$ doesn't have cyclic image in any further
element of $\G.$ Furthermore, if we apply $\phia,$ altering $p$
appropriately, this state of affairs is unchanged, and a further
application of $\phir,$ also altering $p$ appropriately, also leaves
this state of affairs unchanged, thus any path from the vertex group
of $\mathcal{P}(n)$ associated to $W_{i,j}$ to the vertex group
associated to $W_{i',j'}$ must pass through a big abelian vertex
group.

Now consider the commutative diagram in
Figure~\ref{fig:analyzeblowup}. Define
  \[\mathcal{P}(n)'\define\Delta_n(\pi_n(V^n_i),\pi_n(A^n_j),\pi_n(E^n_l))\]

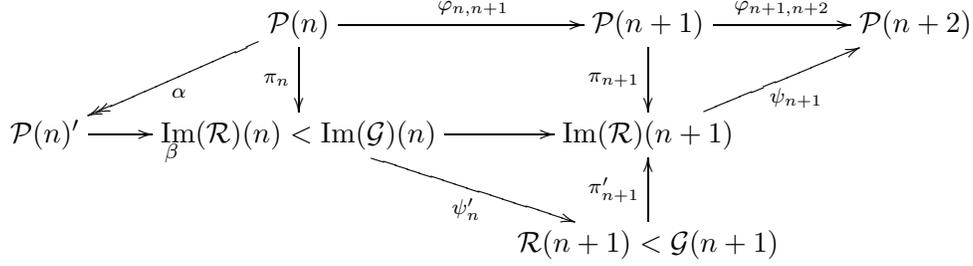
\begin{figure}[h]
  \centerline{%
    \xymatrix{%
    &   \mathcal{P}(n)\ar@{->>}[dl]^{\alpha}\ar[d]_{\pi_n}\ar[r]^{\varphi_{n,n+1}} & \mathcal{P}(n+1)\ar[d]_{\pi_{n+1}}\ar[r]^{\varphi_{n+1,n+2}}& \mathcal{P}(n+2) \\
 \mathcal{P}(n)'\ar[r]_{\beta}  &   \img(\R)(n)<\img(\G)(n)\ar[r]\ar[dr]_{\psi'_{n}}    &  \img(\R)(n+1)\ar[ur]_{\psi_{n+1}} & \\
      &            {}    & \R(n+1)<\G(n+1)\ar[u]^{\pi'_{n+1}} & 
}}
\caption{Some complicated diagram}
\label{fig:analyzeblowup}
\end{figure}

  Let $\overline{V}^n_i=\pi_n(V^n_i),$ and let $\Delta_{n,i}$ be the
  decomposition over cyclic groups that $\overline{V}^n_{i}$ inherits
  from its action on the Bass-Serre tree associated to the cyclic
  \jsj\ decomposition of $\img(\G)(n).$ Since $A^n_j$ is big, it has
  elliptic image in $\img(\G)(n),$ therefore the edges $E^n_i$ have
  elliptic images in $\img(\G)(n),$ and the decompositions
  $\Delta_{n,i}$ can be used to refine the decomposition $\Delta_n$ of
  $\mathcal{P}(n)'$ to a decomposition $\Delta.$

  Let $\Gamma_{n,i}$ be the underlying graph of $\Delta_{n,i},$ and
  let $\Gamma_n$ be the underlying graph of the cyclic \jsj\ of
  $\img(\G)(n).$ The natural map of graphs of groups $\Delta_{n,i}$ to
  the cyclic \jsj\ of $\img(\G)(n)$ induces a combinatorial map of
  underlying graphs $\iota_{n,i}\colon\Gamma_{n,i}\to\Gamma_n.$ Since
  $\img(G)(n)\to\G(n+1)$ is degenerate and the image of $V_i$ in
  $\G(n+1)$ is contained in a vertex group of the abelian \jsj\ of
  $\G(n+1),$ $\iota_{n,i}$ is homotopically trivial, therefore
  $\iota_{n,i}$ factors through a map
  $\wt{\iota}_{n,i}\colon\Gamma_{n,i}\to\wt{\Gamma}_n,$ the universal
  cover of $\Gamma_n.$ Furthermore, for at most one set of preimages
  $I=\wt{\iota}_{n,i}^{-1}(w),$ $w\in\wt{\Gamma}_n,$ can the vertex
  groups of $\Delta_{n,i}$ have nonabelian image in $\G(n+1).$ Let
  $I_e$ be this set of vertices, and let $e$ be the associated vertex
  of $\wt{\Gamma}_n.$ See Figure~\ref{fig:decreasecomplexityfigure}.

  \begin{figure}[ht]
    \psfrag{Gammat}{$\wt{\Gamma}$}
    \psfrag{Gammai}{$\Gamma_{n,i}$}
    \psfrag{e}{$e$}
    \psfrag{I}{$I_e$}
    \centerline{%
      \includegraphics{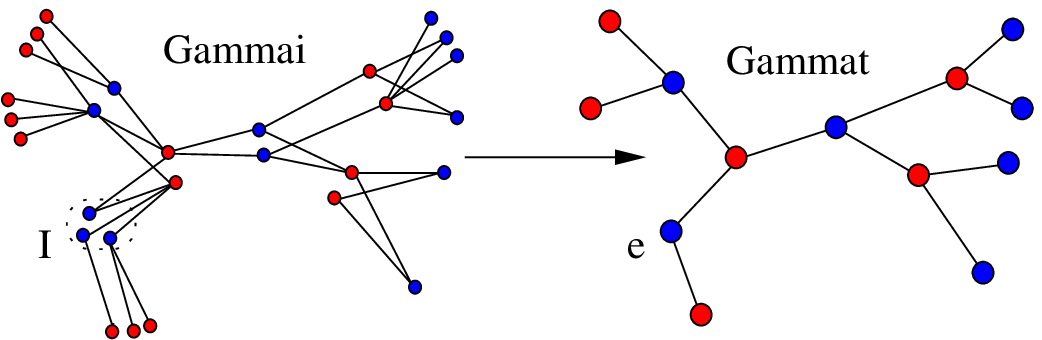}
      }
      \caption{Blowing up $\overline{V}^n_i$}
      \label{fig:decreasecomplexityfigure}
    \end{figure}

    Choose $F\in\E(V^n_i).$ Since $F$ is adjacent to some
    $A^n_j,$ and since the $A^n_j$ have elliptic images in $\G,$ it is
    elliptic in $\Delta_{n,i}.$ Since $c_a(\mathcal{P})=0,$ if a
    valence one vertex group $K$ of some $\Delta_{n,i}$ is abelian
    (and is necessarily noncyclic) or has abelian image in $\G(n+1),$
    then some $F_i$ must have image in $K,$ a contradiction to the
    fact that $\mathcal{P}$ is $\jsj$ respecting.

    Suppose an abelian vertex group of $\Delta_{n,i}$ with valence at
    least two contains the image of $F.$ If the vertex is
    nonseparating then $\Delta_{n+1}$ has strictly more edges than
    $\Delta_n$ since $c_a(\mathcal{P})=0,$ a contradiction. If it is
    separating then either $\Delta_{n+1}$ has more edges than
    $\Delta_n,$ an impossibility, or one of the complementary
    components has abelian image in $\mathcal{G}(n+1),$ then no other
    edge from $\Delta_n$ has image in the component with abelian image
    and we conclude again that $c_a(\mathcal{P})>0,$ another
    contradiction. Thus $F$ has image in a nonabelian vertex group of
    $\Delta_{n,i}$ and this vertex group doesn't have abelian image in
    $\G(n+1).$ If any vertex group $K$ (abelian or not) of
    $\Delta_{n,i}$ which \emph{does not} contain the image of any
    $F_i$ has abelian image in $\G(n+1)$ then, since $\Delta_{n+1}$
    has the same number of abelian vertex groups as $\Delta_n,$ $K$
    must have cyclic image in $\mathcal{P}(n+2).$ By
    Lemma~\ref{lem:smallishbettinumbers} applied to the composition
    $\mathcal{P}(n)'\to\mathcal{P}(n+2),$ the vertex group of
    $\qcjsj(\mathcal{P}(n+2))$ has strictly lower betti number than
    $\mathcal{P}(n+2),$ contradictory to hypothesis. We conclude that
    all vertex groups carried by vertices from
    $\Gamma^{(0)}_{n,i}\setminus I_e$ are infinite cyclic.

    We must then have that all noncyclic vertices of $\Delta_{n,i}$
    are in $I_e.$ Since all noncyclic vertices are in $I_e,$ every
    cyclic vertex must be at most distance one from $I,$ and we
    conclude that $\overline{V}^n_i$ \emph{is contained in a neighborhood of
    a vertex group of the cyclic \jsj\ of $\img(\G)(n)$}.

  Let $\mathcal{M}(n)$ be the vertex group of the \qcjsj\ of
  $\mathcal{P}(n).$ Every vertex group of $\jsj_B(\mathcal{P}(n))$ is
  contained in some $V^n_i,$ hence every vertex group of
  $\mathcal{M}(n)$ has image in a neighborhood of a vertex group of
  the cyclic \jsj\ of $\img(\G)(n).$ Suppose $M_1$ and $M_2$ are
  vertex groups of $\jsj_B(\mathcal{P}(n))$ which are adjacent in
  $\mathcal{M}(n).$ Then $M_1$ and $M_2$ are connected by a big edge
  $E.$ If $M_1*_EM_2$ didn't map to a neighborhood of a vertex group
  of the cyclic \jsj\ of $\img(\G)(n)$ then, since all edges incident
  to $M_1$ are big, $M_1$ is elliptic, and the edges of
  $\jsj_c(\img(\G)(n))$ are infinite cyclic,
  $c_a(\varphi_{n,n+1}\vert_{M_1},\E(M_1))>0,$ contradicting
  the assumption that $\mathcal{P}$ is \jsj\ respecting.
\end{proof}


\begin{lemma}
  \label{lem:neighborhoodlemma}
  Let $L$ be a limit group and $N$ be a neighborhood of a vertex group
  $V$ of the cyclic \jsj. Let $A_1,\dotsc,A_k$ be the centralizers of
  edge groups incident to $V$ which aren't contained in $V.$ A freely
  indecomposable subgroup $G$ of $N$ has the property that its
  vertex groups are contained in conjugates of $V$ or it has a
  principle cyclic splitting of the form $G=G'*_{\zee}A',$ for some
  $A'<A_j,$ for some $j.$
\end{lemma}

In the previous section we showed how to lift a dimension bound for
sequences of vertex groups of \qcjsj\ decompositions to a dimension
bound for \qcjsj\ respecting sequences of limit groups. The next
theorem allows us to apply the construction of maximal resolutions of
sequences of images of vertex groups twice, in the event that no
``obvious'' reductions in complexity are possible (nontrivial $c_a,$
freely decomposable maximal resolutions) to arrive at resolutions of
lower complexity.

\begin{theorem}[Decrease the complexity]
  \label{thr:decreasecomplexity-final}
  Let $\LL,\G,\R,\P,\mathcal{M}$ be as in
  Lemma~\ref{lem:decreasecomplexity-maybethisworkseasily}. Let
  $\H\rto\img(\mathcal{M})$ be a \qcjsj\ respecting maximal
  resolution of the image of $\mathcal{M}.$
  
  Let\footnote{These are the same hypotheses as
    Lemma~\ref{lem:decreasecomplexity-maybethisworkseasily}.}
  $\mathcal{S}$ be a sequence of vertex groups of
  $\qcjsj(\H),$ and suppose that $\mathcal{Q}$ is an
  indecomposable maximal \qcjsj\ respecting resolution of the image
  $\img(\mathcal{S})\into\LL$ and $c_a(\mathcal{Q}(n))=0$ for all
  $i.$ Suppose further that $\mathcal{Q}$ has no \qh\ sequences of
  vertex groups, only one sequence $\mathcal{N}$ of vertex groups in
  its \qcjsj\ decomposition, and that
  $\betti(\mathcal{N})=\betti(\mathcal{Q}).$ Let $\set{1,\dotsc,m}$ be
  the index set for $\mathcal{S}.$ Then the vertex groups
  $\mathcal{N}(n)$ map to a vertex group of $\jsj_C(\img(\G)(n))$ for
  $m-4\geq n\geq 4.$

  In particular,
  \[\comp(\mathcal{N}\vert_{(4,\dotsc,m-4)}\rto\LL)<\comp(\G\rto\LL)\]
\end{theorem}
\begin{figure}
  \centerline{%
    \xymatrix{%
      \G\ar[d]^{\rto} & \R\ar[d]^{\rto}\ar@{_(->}[l] & \mathcal{P}\ar[dl]^{\rto} & \mathcal{M}\ar[dl]^{\rto}\ar@{_(->}[l] \\
      \img(\G) & \img(\R)\ar@{_(->}[l] & \img(\mathcal{M})\ar@{_(->}[l] & \H\ar[d]^{\rto} & \mathcal{S}\ar[d]^{\rto}\ar@{_(->}[l] & \mathcal{Q}\ar[dl]^{\rto} & \mathcal{N}\ar@{_(->}[l]\ar[dl]^{\rto} \\
 & & & \img(\mathcal{M})\ar@{=}[ul] & \img(\mathcal{S})\ar@{_(->}[l] & \img(\mathcal{N})\ar@{_(->}[l]
}}
  \caption{Illustration for
    Theorem~\ref{thr:decreasecomplexity-final}. The vertical/slanted
    arrows labeled ``$\rto$'' are resolutions, and the horizontal
    arrows are all inclusions of subsequences.}
\label{fig:decreasecomplexity-2-figure}
\end{figure}
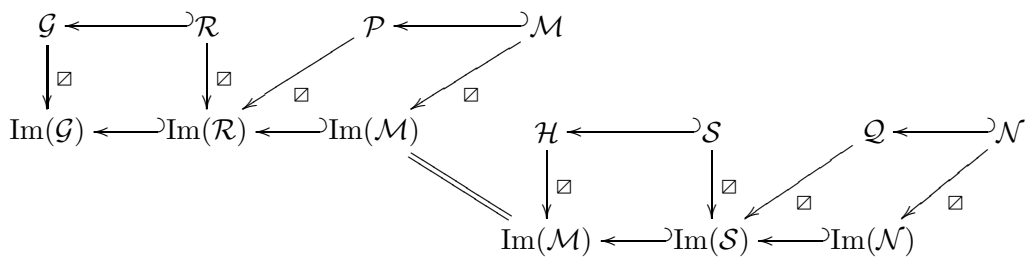

\begin{proof}
  Consider the diagram in
  Figure~\ref{fig:decreasecomplexity-2-figure}. Each block, separated
  by the long equals sign, represents a use of
  Lemma~\ref{lem:decreasecomplexity-maybethisworkseasily}.

  By Lemma~\ref{lem:neighborhoodlemma}, since $c_a(\mathcal{Q})=0,$
  the neighborhoods of vertex groups of the images
  $\img(\mathcal{M})(n)$ are completely contained in vertex groups of
  the cyclic \jsj\ of $\img(\G).$ By
  Lemma~\ref{lem:decreasecomplexity-maybethisworkseasily} applied to
  the tuple $(\H,\mathcal{S},\mathcal{Q},\mathcal{N}),$
  $\img(\mathcal{N})(n)$ is contained in a vertex group of the cyclic
  \jsj\ of $\img(\G)(n),$ hence
  $\comp(\mathcal{N}\vert_{(4,\dotsc,m-4)}\rto\LL)<\comp(\G\rto\LL).$
\end{proof}

Now we can continue the analysis using the sequences
$\img(\mathcal{N})\vert_{(4,\dotsc,m-4)},$ which have strictly lower
depth than $\LL.$ 

To finish the proof of Theorem~\ref{thr:subseqkrull} we combine the
work from previous sections. Corollary~\ref{cor:alignmentcorollary},
Corollary~\ref{cor:reduction-to-indecomposable},
Remark~\ref{rem:thisistheproof}, Theorem~\ref{thr:stratified},
Theorem~\ref{thr:decreasecomplexity-final}, and the uniform bound on
depths of rank $n$ limit groups provided by
Theorem~\ref{lem:depthbound} formally imply
Theorems~\ref{thr:subseqkrull} and~\ref{thr:krulldimension}.



\bibliographystyle{amsalpha} \bibliography{krull}

\providecommand{\bysame}{\leavevmode\hbox to3em{\hrulefill}\thinspace}
\providecommand{\MR}{\relax\ifhmode\unskip\space\fi MR }
\providecommand{\MRhref}[2]{%
  \href{http://www.ams.org/mathscinet-getitem?mr=#1}{#2}
}
\providecommand{\href}[2]{#2}
\begin{thebibliography}{{Swa}04}

\bibitem[BF91]{bf::bounding}
Mladen Bestvina and Mark Feighn, \emph{Bounding the complexity of simplicial
  group actions on trees}, Invent. Math. \textbf{103} (1991), no.~3, 449--469.
  \MR{MR1091614 (92c:20044)}

\bibitem[BF03]{bf::lg}
Mladen Bestvina and Mark Feighn, \emph{Notes on {S}ela's work: {L}imit groups
  and {M}akanin-{R}azborov diagrams}.

\bibitem[BMR00]{baumslag::ag}
Gilbert Baumslag, Alexei Myasnikov, and Vladimir Remeslennikov, \emph{Algebraic
  geometry over groups}, Algorithmic problems in groups and semigroups
  (Lincoln, NE, 1998), Trends Math., Birkh\"auser Boston, Boston, MA, 2000,
  pp.~35--50. \MR{MR1750491 (2000m:20066)}

\bibitem[CG05]{guirardel-champetier}
Christophe Champetier and Vincent Guirardel, \emph{Limit groups as limits of
  free groups}, Israel J. Math. \textbf{146} (2005), 1--75. \MR{MR2151593
  (2006d:20045)}

\bibitem[DS99]{dunwoodyjsj}
M.~J. Dunwoody and M.~E. Sageev, \emph{J{SJ}-splittings for finitely presented
  groups over slender groups}, Invent. Math. \textbf{135} (1999), no.~1,
  25--44. \MR{MR1664694 (2000b:20050)}

\bibitem[Dun98]{dunwoody::folding}
M.~J. Dunwoody, \emph{Folding sequences}, The Epstein birthday schrift, Geom.
  Topol. Monogr., vol.~1, Geom. Topol. Publ., Coventry, 1998, pp.~139--158
  (electronic). \MR{MR1668347 (2000f:20037)}

\bibitem[FP06]{fuji::jsj}
K.~Fujiwara and P.~Papasoglu, \emph{J{SJ}-decompositions of finitely presented
  groups and complexes of groups}, Geom. Funct. Anal. \textbf{16} (2006),
  no.~1, 70--125. \MR{MR2221253 (2007c:20100)}

\bibitem[GR90]{graham90}
Ronald~L. Graham and Bruce~L. Rothschild, \emph{Ramsey theory (2nd ed.)},
  Wiley-Interscience, New York, NY, USA, 1990.

\bibitem[Gui04]{guirardel::fp}
Vincent Guirardel, \emph{Limit groups and groups acting freely on {$\Bbb R\sp
  n$}-trees}, Geom. Topol. \textbf{8} (2004), 1427--1470 (electronic).
  \MR{MR2119301 (2005m:20060)}

\bibitem[Hou08]{houcine-2008}
Abderezak~Ould Houcine, \emph{Note on the cantor-bendixon rank of limit
  groups}, 2008, \url{http://lanl.arxiv.org/abs/math/0804.2841}.

\bibitem[KM06]{km}
Olga Kharlampovich and Alexei Myasnikov, \emph{Elementary theory of free
  non-abelian groups}, J. Algebra \textbf{302} (2006), no.~2, 451--552.
  \MR{MR2293770}

\bibitem[Lou08a]{louder::strict}
Larsen Louder, \emph{Krull dimension for limit groups {I}: {B}ounding strict
  resolutions}, 2008, \url{http://arXiv.org/abs/math/0702115v3}.

\bibitem[Lou08b]{louder::scott}
\bysame, \emph{Krull dimension for limit groups {III}: {S}cott complexity and
  adjoining roots to finitely generated groups}, 2008,
  \url{http://arXiv.org/abs/math/0612222v3}.

\bibitem[Lou08c]{louder::roots}
\bysame, \emph{Krull dimension for limit groups {IV}: {A}djoining roots}, 2008,
  \url{http://arXiv.org/abs/math/NUMBER}.

\bibitem[RS97]{sela::jsj}
E.~Rips and Z.~Sela, \emph{Cyclic splittings of finitely presented groups and
  the canonical {JSJ} decomposition}, Ann. of Math. (2) \textbf{146} (1997),
  no.~1, 53--109. \MR{MR1469317 (98m:20044)}

\bibitem[Sco73]{scottcoherent}
G.~P. Scott, \emph{Finitely generated {$3$}-manifold groups are finitely
  presented}, J. London Math. Soc. (2) \textbf{6} (1973), 437--440.
  \MR{MR0380763 (52 \#1660)}

\bibitem[Sela]{sela::dgog5}
Zlil Sela, \emph{Diophantine geometry over groups {V}: Quantifier elimination},
  Israel J. Math.

\bibitem[Selb]{sela::dgog6}
\bysame, \emph{Diophantine geometry over groups {VI}: Quantifier elimination},
  GAFA.

\bibitem[Sel97]{sela::acyl}
\bysame, \emph{Acylindrical accessibility for groups}, Invent. Math.
  \textbf{129} (1997), no.~3, 527--565. \MR{MR1465334 (98m:20045)}

\bibitem[Sel01]{sela::dgog1}
\bysame, \emph{Diophantine geometry over groups. {I}. {M}akanin-{R}azborov
  diagrams}, Publ. Math. Inst. Hautes \'Etudes Sci. (2001), no.~93, 31--105.
  \MR{MR1863735 (2002h:20061)}

\bibitem[Sel03]{sela::dgog2}
\bysame, \emph{Diophantine geometry over groups. {II}. {C}ompletions, closures
  and formal solutions}, Israel J. Math. \textbf{134} (2003), 173--254.
  \MR{MR1972179 (2004g:20061)}

\bibitem[Sel04]{sela::dgog4}
\bysame, \emph{Diophantine geometry over groups. {IV}. {A}n iterative procedure
  for validation of a sentence}, Israel J. Math. \textbf{143} (2004), 1--130.
  \MR{MR2106978 (2006j:20059)}

\bibitem[Sel05]{sela::dgog3}
\bysame, \emph{Diophantine geometry over groups. {III}. {R}igid and solid
  solutions}, Israel J. Math. \textbf{147} (2005), 1--73. \MR{MR2166355
  (2006j:20060)}

\bibitem[{Swa}04]{swarup::scott}
G.~A. {Swarup}, \emph{{Delzant's variation on Scott complexity}}, January 2004.

\bibitem[Wei02]{weid::acyl}
Richard Weidmann, \emph{The {N}ielsen method for groups acting on trees}, Proc.
  London Math. Soc. (3) \textbf{85} (2002), no.~1, 93--118. \MR{MR1901370
  (2003c:20029)}

\end{thebibliography}

\end{document}